\documentclass[reqno,11pt]{amsart}
\usepackage{amsmath,amssymb,amsthm, amscd, braket}
\usepackage[labelformat=empty]{caption}
\usepackage{array}
\usepackage{soul} 
\usepackage{color, xcolor}

\usepackage{amsbsy}
\usepackage{latexsym}
\usepackage[all]{xy}
\usepackage{amsthm}
\usepackage{mathrsfs}
\usepackage{dsfont}

\usepackage{color}

\usepackage{hyperref}
\usepackage{color}
\hypersetup{
    colorlinks=true,
    citecolor=blue,
    linkcolor=black,
  urlcolor=blue,
}

\usepackage{amsbsy}
\usepackage{latexsym}
\usepackage[all]{xy}
\usepackage{amsthm}
\usepackage{mathrsfs}
\usepackage{dsfont}

\usepackage{color}

\usepackage{color, xcolor}
\usepackage{hyperref}
\usepackage{color}
\hypersetup{
    colorlinks=true,
    citecolor=blue,
    linkcolor=black,
  urlcolor=blue,
}

\usepackage{ulem}

\usepackage{amscd,amsthm,amsfonts,amssymb,amsmath}
\usepackage{amsmath,tikz-cd}
\usepackage{fullpage}
\usepackage[all]{xy}
\usepackage{mathtools}
\usepackage{mathrsfs}
\usepackage{nccmath}
\usepackage{amscd} 
 
\newcommand{\ep}{\epsilon}
\newcommand{\x}{\times}
\newcommand{\C}{\BC}

\newcommand{\ot}{\otimes}

\newcommand{\Diff}{\delta}
\newcommand{\CMP}{\theta}
\newcommand{\cmJ}{J}
\newcommand{\cmptv}{\varsigma}

\newcommand{\pme}{q}
\newcommand{\lnew}{\varphi_q}

\newcommand{\vp}{\varphi}
\newcommand{\bksl}{\backslash}

\newcommand{\Thetam}{\Theta}

    \newcommand{\BA}{{\mathbb {A}}} 
    \newcommand{\BC}{{\mathbb {C}}} 
     \newcommand{\BF}{{\mathbb {F}}}

     \newcommand{\BP}{{\mathbb {P}}}
    \newcommand{\BQ}{{\mathbb {Q}}} \newcommand{\BR}{{\mathbb {R}}}

     \newcommand{\BZ}{{\mathbb {Z}}}

     \newcommand{\CH}{{\mathcal {H}}}
     
    \newcommand{\CK}{{\mathcal {K}}} \newcommand{\CL}{{\mathcal {L}}}

     \newcommand{\CR}{{\mathcal {R}}}
    \newcommand{\CS}{{\mathcal {S}}}

 \newcommand{\SF}{{\mathscr {F}}}

     \newcommand{\fl}{{\mathfrak{l}}}
     
     \newcommand{\fp}{{\mathfrak{p}}}
    \newcommand{\fq}{{\mathfrak{q}}}

     \newcommand{\fN}{{\mathfrak{N}}}

    \newcommand{\ad}{{\mathrm{ad}}}
    \newcommand{\alg}{{\mathrm{alg}}}
    \newcommand{\Aut}{{\mathrm{Aut}}}

    \newcommand{\cond}{\mathrm{cond^r}}

    \newcommand{\Coker}{{\mathrm{Coker}}}
    \newcommand{\coker}{{\mathrm{coker}}}

    \newcommand{\End}{{\mathrm{End}}} \newcommand{\Eis}{{\mathrm{Eis}}}

    \newcommand{\Gal}{{\mathrm{Gal}}} \newcommand{\GL}{{\mathrm{GL}}}
    
    \newcommand{\Hom}{{\mathrm{Hom}}}

     \newcommand{\Red}{{\rm{Red}}}
    \newcommand{\Ker}{{\mathrm{Ker}}}

    \newcommand{\PGL}{{\mathrm{PGL}}}

    \renewcommand{\mod}{\ \mathrm{mod}\ }
    
    \newcommand{\ram}{{\mathrm{ram}}}

    \newcommand{\SL}{{\mathrm{SL}}}

    \newcommand{\St}{{\mathrm{St}}}
    \newcommand{\sgn}{{\mathrm{sgn}}}
    \newcommand{\Stab}{{\mathrm{Stab}}}

    \newcommand{\tr}{{\mathrm{tr}}}

    \newcommand{\vol}{{\mathrm{vol}}}

    \newcommand{\Q}{\mathbb{Q}}
    \newcommand{\Z}{\mathbb{Z}}

   \newcommand{\wtd}{\widetilde}

    \font\cyr=wncyr10

    \newcommand{\Sha}{\hbox{\cyr X}}\newcommand{\wt}{\widetilde}
    
    \newcommand{\wh}{\widehat}
    \newcommand{\cmpt}{\varsigma}
    \newcommand{\CMspace}{{\rm{CM}}}
    \newcommand{\ol}{\ov}
    
    \newcommand{\pair}[1]{\langle {#1} \rangle}

    \newcommand{\ov}{\overline}

    \newcommand{\sk}{\medskip}
    
    \newcommand{\ra}{\rightarrow} 
    \newcommand{\bs}{\backslash}
    \newcommand{\nequiv}{\equiv\hspace{-10pt}/\ }
    \newcommand{\s}{\sk\noindent}

    \theoremstyle{plain}
    \newtheorem{thm}{Theorem}[section] \newtheorem{cor}[thm]{Corollary}
    \newtheorem{lem}[thm]{Lemma}  \newtheorem{prop}[thm]{Proposition}
     \newtheorem{defn}[thm]{Definition}
    
    \newtheorem{fact}[thm]{Fact}
    \newcommand{\uf}{\varpi}
\theoremstyle{remark} \newtheorem{remark}[thm]{Remark}
\theoremstyle{remark} 
\theoremstyle{remark} \newtheorem{example}{Example}

    \numberwithin{equation}{section}
  \date{\today}

\begin{document}  
\title{Hecke $L$-values, definite Shimura sets and\\
Mod $\ell$ non-vanishing}

\author{Ashay A. Burungale, Wei He, Shinichi Kobayashi and Kazuto Ota}

\address{Ashay A. Burungale:  Department of Mathematics, The university of Texas at Austin,
2515 Speedway, Austin TX 78712, USA, and Max Planck Institute for Mathematics, Vivatsgasse 7, 53111 Bonn, Germany.} 
\email{ashayburungale@gmail.com}

\address{Wei He: School of Mathematics and Statistics, Xi'an Jiaotong University, Xi'an 710049, China} 
\email{hewei0714@xjtu.edu.cn}

\address{Shinichi Kobayashi: Faculty of Mathematics,
Kyushu University, 744, Motooka, Nishi-ku, Fukuoka, 819-0395, Japan, and Max Planck Institute for Mathematics, Vivatsgasse 7, 53111 Bonn, Germany.}
\email{kobayashi@math.kyushu-u.ac.jp}

\address{Kazuto Ota: \textsc{Department of Mathematics, Graduate School of Science, Osaka University Toyonaka, Osaka 560-0043, Japan}} 
\email{
kazutoota@math.sci.osaka-u.ac.jp}

\begin{abstract} 
Let $\lambda$ be a self-dual Hecke character over an imaginary quadratic field $K$ of infinity type $(1,0)$. 
Let $\ell$ and $p$ be primes 
coprime to $6{\rm N}_{K/\BQ}(\cond{\lambda})$.
We determine the $\ell$-adic valuation of Hecke $L$-values 
$L(1,\lambda\chi)/\Omega_K$ as $\chi$ varies over $p$-power order anticyclotomic characters 
over $K$. 
As an application, for $p$ inert in $K$, we prove vanishing of the $\mu$-invariant of Rubin's $p$-adic $L$-function, leading to the  first result on the classical problem of determination of the $\mu$-invariant of imaginary quadratic fields at non-split primes.

Our results and approach complement the prior work, primarily due to Hida and Finis. 
The approach is rooted in the arithmetic of a CM form on a definite Shimura set, relying on Ratner's ergodicity of unipotent flows and a new local result on test vectors for supercuspidal representations. 
The application to Rubin's $p$-adic $L$-function builds on recent progress in CM Iwasawa theory initiated by  resolution of Rubin's  conjecture. Along the way, we present an automorphic view on Rubin's supersingular Iwasawa theory. 

\end{abstract}
\maketitle
\tableofcontents

\section{Introduction}\label{s:Intro}
Special values of $L$-functions mysteriously encode arithmetic. 
 As underlying motives vary in a family,
 a basic problem is whether the $L$-values are generically non-zero.
 If so, a finer problem: mod $\ell$ non-vanishing of algebraic part of the $L$-values for a fixed prime $\ell$. 
The main results of 
  this paper establish the mod $\ell$ non-vanishing 
  of central Hecke $L$-values in self-dual families over imaginary quadratic fields (cf.~Theorems~\ref{thmA},  ~\ref{thmB} and ~\ref{thmC}). 
  They lead to the first result on the classical problem of determination of the $\mu$-invariant of imaginary quadratic fields at non-split primes (cf.~Corollary~\ref{corA}). 

The study of mod $\ell$ non-vanishing of Hecke $L$-values goes back to the 80's. 
The first results are independently due to Gillard \cite{Gi} and Schneps \cite{Scn},  who 
proved the non-vanishing for $p$-adic deformation of a CM elliptic curve arising from the Coates--Wiles $\BZ_p$-extension of the CM field $K$ for primes $p$ {\it split} in $K$. 
(Throughout the introduction, $\ell$ and $p$ denote primes, not necessarily distinct.)
These results build on the pioneering work of Ferrero--Washington \cite{FW}, Washington \cite{Was} and Sinnott \cite{Si} on Dirichlet $L$-values, and are based on Zariski density of torsion points on self-products of the CM elliptic curve modulo $\ell$. 
A couple of decades later, Hida initiated and extensively studied \cite{Hi1,Hi2,Hi3,Hi4}  
the mod $\ell$ non-vanishing  of Hecke $L$-values for anticyclotomic deformation of Hecke characters and primes $\ell$ {\it split} in $K$.
 In contrast to the prior work it relies on the arithmetic of $\GL_2$-Eisenstein series, studied via 
 geometry of modular curves and mod $\ell$ analogue of the Andr\'e--Oort conjecture \cite{Ch} (Chai--Oort rigidity principle). 
A few years later, Finis established \cite{Fin1,Fin2} the mod $\ell$ non-vanishing of central Hecke $L$-values in self-dual families arising from the $\BZ_p$-anticyclotomic deformation of a self-dual Hecke character\footnote{This includes the case of a CM elliptic curve.} for {\it split } primes $p$. 
His notably different approach relies on the arithmetic of $U(1)$-theta functions with complex multiplication, being rooted in Mumford's theory of theta functions and a variant of the Manin--Mumford conjecture. 

The aim of this paper is to treat self-dual cases excluded by methods of Hida and Finis, indicated by $*$ below 
(cf.~Theorems~\ref{thmA} and~\ref{thmB}). 

\begingroup
\noindent
\newcommand{\tabincell}[2]{\begin{tabular}{@{}#1@{}}#2\end{tabular}}
\renewcommand{\arraystretch}{1.5}
	\begin{center}
\begin{table}[htbp]
	\caption{Mod $\ell$ non-vanishing of $p$-anticyclotomic Hecke $L$-values}  
	\label{table1}  
	\begin{tabular}{| >{\centering\arraybackslash}m{3cm} | >{\centering\arraybackslash}m{3cm} | >{\centering\arraybackslash}m{3cm} |}  
		\hline  
   $(\ell,p)$
		&$p$ split in $K$& $p$ inert in $K$\\ 
		\hline
		$\ell$ split in $K$&Hida, Finis&* \\
    \hline
  $\ell$ inert in $K$&Finis&*\\
		\hline
	\end{tabular}
\end{table}
 \end{center}
 \endgroup
Our results have application to CM Iwasawa theory: vanishing of the 
$\mu$-invariant of Rubin's $p$-adic $L$-function and that of the imaginary quadratic field $K$ at inert primes $p$   (cf.~Theorem~\ref{thmC},~Corollary~\ref{corA} and Remark~\ref{rm-cIw}). 

This paper introduces a new approach to mod $\ell$ non-vanishing of Hecke $L$-values
via arithmetic of CM modular form on a 
definite Shimura set. It relies on Ratner's fundamental ergodicity of unipotent flows ~\cite{Ra}. 
To link arithmetic of the definite Shimura set with that of the imaginary quadratic field, a key is 
an $\ell$-integral comparison of quaterionic and CM periods (cf.~Theorem~\ref{thmE}). 
We approach the period comparison indirectly via a synthesis of local and global tools, the primary local tool being an explicit construction of $\ell$-optimal test vectors for supercuspidal representations (cf.~Theorem~\ref{thmG}), which may be of independent interest. 
In the $\ell=p$ split case our approach gives a different proof of the results of Hida and Finis. 

The vanishing of the $\mu$-invariant of Rubin's $p$-adic $L$-function is a continuation of our study of Rubin's integral Iwasawa theory of CM elliptic curves at supersingular primes initiated by our resolution of of Rubin's conjecture in 2021 (cf.~\cite{BKO,BKOb,BKOd,BKOe}). 
While Rubin's construction of his $p$-adic $L$-function is motivic, relying on elliptic units and his local conjecture encoding epsilon factors, we first construct its automorphic counterpart via the definite Shimura set. This link allows access to automorphic and representation theoretic tools. 
The automorphic perspective on Rubin's theory is at the heart of this paper. 

\subsection{Main results}

\subsubsection{Setting}
Let $K$ be an imaginary quadratic field.  
Let $\eta_K$ be the associated quadratic character over $\BQ$ and $h_K$ the class number.
Let $\lambda$ be a (conjugate) self-dual Hecke character over $K$ of infinity type $(1,0)$, that is, 
  \begin{equation}\label{theta}
  \lambda_\infty(z)=z^{-1} \text{ and $\lambda^*:=\lambda |\cdot|_{\BA_K^\times}^{1/2}$ satisfies {$\lambda^*|_{\BA_{\BQ}^\times}=\eta_{K}$. }}
 \end{equation}

For a prime $p$, let $K_{\infty}$ be the anticyclotomic $\BZ_p$-extension of $K$. Put $\Gamma=\Gal(K_{\infty}/K)$ and let $\Xi_p$ denote the set of finite order characters of $\Gamma$. 
 For $\nu\in\Xi_p$, note that $\lambda\nu$ is also self-dual. Put 
$$
\Xi_{\lambda,p}^{\pm}=\{ \nu \in \Xi_{p}| \, \epsilon(\lambda\nu)=\pm 1\},
$$
where $\epsilon(\lambda\nu)$ denotes the root number of the Hecke $L$-function $L(s,\lambda\nu)$. We normalise the latter so that $s=1$ is the center of the functional equation. If $\nu\in\Xi_{\lambda,p}^{-}$, note that $L(1,\lambda\nu)=0$.
 
 In this paper we consider divisibility properties of algebraic part of the central $L$-values $L(1,\lambda\nu)$ for  
$\nu \in \Xi_{\lambda,p}^+$ and $p\nmid D_K$. 
If $p$ splits in $K$, then $\epsilon(\lambda\nu)=\epsilon(\lambda)$. On the other hand, for inert $p\nmid 2N_{K/\BQ}(\cond{\lambda})$, Greenberg observed an interesting variation 
$$
\epsilon(\lambda\nu)=(-1)^{t_{p}+1}\epsilon(\lambda),
$$
where the associated local character $\nu_p$ is of conductor $p^{t_{p}+1}>1$ (cf.~\cite{Gr85}).
(If $p\nmid h_K$, then $\cond{\nu_{p}}=p^{t_{p}+1}$ if and only if ${\rm ord}(\nu)=p^{t_{p}}$.)

Let $\ell$ be a prime and $v_{\ell}$ the $\ell$-adic valuation on $\BC_\ell$ so that $v_{\ell}(\ell)=1$. 
Fix embeddings $\iota_{\infty}: \ov{\BQ}\hookrightarrow \BC$ and $\iota_{\ell}: \ov{\BQ}\hookrightarrow \BC_\ell$. 

To introduce CM period, consider an elliptic curve $E$ with complex multiplication by $O_K$, defined over a number field $M\subset \ov{\BQ}$, and a non-vanishing invariant differential $\omega$ on $E$. 
We may extend the field of definition to $\BC$ via $\iota_\infty$, and possibly replacing $E$ by a Galois conjugate, 
obtain $$\Omega_K \in \BC^{\times},$$ uniquely determined up to units in $K$, such that the period lattice of $\omega$ on $E$ is given by $\Omega_{K}O_{K}$. 
For a given prime $\ell$, we normalise the pair $(E,\omega)$ so that $E$ has good reduction at the $\ell$-adic place $\fl$ of $M$ determined via $\iota_\ell$,  
and $\omega$ reduces modulo $\fl$ to a non-vanishing invariant differential on the reduced curve $\tilde{E}$. Fix the pair $(E,\omega)$ and the resulting period $\Omega_K$.

Hurwitz proved that $$\frac{L(1,\lambda\nu)}{\Omega_{K}} \in \ov{\BQ}.$$
 A basic problem: 
\begin{equation}\label{Q}\tag{Q}
\text{How does $v_{\ell}\left(\frac{L(1,\lambda\nu)}{\Omega_{K}}\right)$ vary with $\nu\in\Xi_{\lambda,p}^{+}$?}
\end{equation}

The following local invariants appear in our non-vanishing results. For a prime $q$, put $\mu_{\ell}(\lambda_q)=\min_{x\in O_{K_q}^\times}v_{\ell}(\lambda_{q}(x)-1)$ and 
\begin{equation}\label{Tam}
\mu_{\ell}(\lambda)=\sum_{\text{$q|{\rm{N}}_{K/\BQ}(\cond{\lambda})$ inert in $K$}} \mu_{\ell}(\lambda_q). 
\end{equation}
The latter invariant 
is closely related to the $\ell$-part of the Tamagawa number associated to $\lambda$. 
Its relevance to the problem \eqref{Q} is due to the lower bound
\begin{equation}\label{H_lb}
v_{\ell}\left( \frac{L(1,\lambda\nu)}{\Omega_{K}}\right)\geq \mu_{\ell}(\lambda), 
\end{equation}
as predicted by the Bloch--Kato conjecture (cf.~\cite{Fin1}).  

\subsubsection{$(\ell,p)$ non-vanishing} Our first main result is the following. 
\begin{thm}\label{thmA}
Let $\lambda$ be a self-dual Hecke character over an imaginary quadratic field $K$ of infinity type $(1,0)$. 
Let $\ell$ and $p$ be two different primes which are coprime to $2{\rm{N}}_{K/\BQ}(\cond{\lambda})$.
If $\epsilon(\lambda)=-1$, suppose that $\ell \geq 5$.
Then for all but finitely many $\nu \in \Xi_{\lambda,p}^+$ we have
$$
v_{\ell}\left( \frac{L(1,\lambda\nu)}{\Omega_{K}}\right)=\mu_{\ell}(\lambda).
$$
\end{thm}

In view of the Birch and Swinnerton-Dyer formula for rank zero CM abelian varieties \cite{Ru91,BF}, Theorem~\ref{thmA} has the following application. 
\begin{cor}\label{cor-A}
Let $\lambda$ be a self-dual Hecke character over an imaginary quadratic field $K$ of infinity type $(1,0)$
 and $A_{\lambda}$ an associated CM abelian variety over $K$. 
Let $\ell$ and $p$ be two different primes which are coprime to $2{\rm{N}}_{K/\BQ}(\cond{\lambda})$. 
  For the $\BZ_p$-anticyclotomic extension $K_\infty$ and a non-negative integer $n$, let $K_n$ denote its $n$-th layer. 
  \begin{itemize}
  \item[(a)] Suppose that $p$ splits in $K$ and $\epsilon(\lambda)=+1$. If the Tate--Shafarevich groups $\Sha(A_{\lambda}/K_n)[\ell^\infty]$ are finite for each $n$, then there exists a constant $c$ such for any sufficiently large $n$, we have 
  $$\# \Sha(A_{\lambda}/K_{n})[\ell^\infty] < c.
  $$
  \item[(b)] Suppose that $p$ is inert in $K$ and\footnote{The hypothesis $p\nmid h_K$ is inessential, it only leads to a cleaner formulation of the result.} $p\nmid h_{K}$. 
  Then there exists a constant $c$ such that for any sufficiently large $n$ with $(-1)^{n+1}=\epsilon(\lambda)$, we have
  $$\#\coker(\Sha(A_{\lambda}/K_{n-1})[\ell^\infty]\rightarrow \Sha(A_{\lambda}/K_{n})[\ell^{\infty}])<c.$$
  \end{itemize}
\end{cor}

\subsubsection{} 
In the $\ell=p$ case our main result:  
\begin{thm}\label{thmB}
Let $\lambda$ be a self-dual Hecke character over an imaginary quadratic field $K$ of infinity type $(1,0)$. 
Let  $p\nmid 2{\rm{N}}_{K/\BQ}(\cond{\lambda})$ be a prime and $\mu_{p}(\lambda)$ be as in \eqref{Tam}. 
\begin{itemize}
\item[(a)] Suppose that $p$ splits in $K$ and $\fp$ the prime of $K$ above $p$ determined via the embedding $\iota_p$. Suppose 
 that $\epsilon(\lambda)=+1$. 
Then there exists an integer $c_{\lambda}\geq 0$ such that for any $\nu \in \Xi_{p}$  of order $p^{t} \gg 1$ and the local character $\nu_{\fp}$ of conductor $p^{t_{\fp}+1}$, we have 
$$
v_{p}\left( \frac{L(1,\lambda\nu)}{\Omega_{K}}\right)=\mu_{p}(\lambda)
+ 
\frac{c_{\lambda}}{p^{t-1}(p-1)}-\frac{t_{\fp}+1}{2}. 
$$
\item[(b)] Suppose that $p$ is inert in $K$. 
If $\epsilon(\lambda)=-1$, suppose also that $p\geq 5$. 
Then there exists an integer $c_{\lambda}\geq 0$ such that for any $\nu \in \Xi_{\lambda,p}^+$  of order $p^{t} \gg 1$ and the local character $\nu_{p}$ of conductor $p^{t_{p}+1}$, we have 
$$
v_{p}\left( \frac{L(1,\lambda\nu)}{\Omega_{K}}\right)=
\mu_{p}(\lambda) + \frac{c_{\lambda}}{p^{t-1}(p-1)} - \frac{t_{p}+1}{2}
+ 
\frac{1}{p^{t-1}(p-1)}\left(\frac{1-\epsilon(\lambda)}{2}+\sum_{k\equiv t-1 \mod{2}} (p^k-p^{k-1}) \right), 
$$
where $1\leq k \leq t-1$.
\end{itemize}
\end{thm}

Theorem~\ref{thmB} is a consequence of the existence of certain $p$-adic $L$-functions, and determination of their $\mu$-invariants. Our principal result is that the latter equals 
$\mu_p(\lambda)$. 

The above asymptotic formulas for $p$-adic valuation of Hecke $L$-values have an application to the variation of Tate--Shafarevich groups analogous to Corollary~\ref{cor-A} (see Corollary~\ref{cor-D}).
The variation reflects the underlying Iwasawa theory. While the split case  illustrates a  typical phenomenon at an ordinary prime, the inert case  echoes a peculiar non-ordinary phenomenon. 
While the shape of the asymptotic formula in the former case goes back to Katz \cite{Kz}, the latter case recently appeared in \cite{BKOd}. (In these works an abstract invariant $\mu\in\BZ_{\geq 0}$ instead of the above $\mu_{p}(\lambda)$ appears.)

\begin{remark}
For Theorem~\ref{thmB}(a), the hypothesis $\epsilon(\lambda)=+1$ is essential as otherwise
$L(1,\lambda\nu)=0$ for any $\nu \in \Xi_{p}$.
\end{remark}

\subsubsection{Application to CM Iwasawa theory at inert primes}
We describe an application to Rubin's supersingular CM Iwasawa theory \cite{Ru}.

Let $p\nmid 6h_K$ be a prime inert in the imaginary quadratic field $K$. 
Let $\Phi$ denote the completion of $K$ at the prime ideal generated by $p$. 
Let $H$ denote the Hilbert class field of $K$. 
 For $\lambda$ as above, suppose that 
 the Hecke character $\lambda\circ {\rm{N}}_{H/K}$ 
is associated to a $\BQ$-curve $E$ over $H$ with good reduction at primes of $H$ above $p$ (cf.~\cite{Gro80}). 
Without loss of generality, we assume that $E$ has CM by $O_K$.
Let $T$ denote the $p$-adic Tate module of $E$, which is naturally endowed with an $O_{\Phi}$-action.

Let  $\Psi_\infty$ be 
the anticyclotomic $\BZ_p$-extension  of $\Phi$ 
and $\Psi_n$ the $n$-th layer. 
Denote the Iwasawa cohomology  $\varprojlim_n H^1(\Psi_n, T^{\otimes -1}(1))$ 
by $\mathcal{H}^1$, where $T^{\otimes -1}$ is the $O_{\Phi}$-dual of $T$. 
Since $p\nmid h_K$, we may identify $\Xi_p$ with the set of finite order characters of $\Gal(\Psi_{\infty}/\Phi)$.
For $\nu \in \Xi_p$ factoring through $\Gal(\Psi_m/\Phi)$, the dual exponential map for $\nu \otimes T^{\otimes -1}(1)$ 
defines
a map 
$$\delta_\nu: \mathcal{H}^1 
\longrightarrow \Psi_m(\mathrm{Im}\,\nu),$$ dependent on a choice of N\'eron differential.
Following Rubin \cite{Ru}, define 
$$
\mathcal{H}^1_\pm := {\{} h \in \mathcal{H}^1 \text{ $|$} \text{ }
\delta_{\nu}(h)=0 \quad \text{for any }\nu \in \Xi_p \text{ of order $p^{t}$ with $t$ odd/even} {\}}.
$$
Rubin \cite{Ru} showed that $\mathcal{H}^1_\pm$ is a free $\Lambda$-module of rank one for $\Lambda:=O_{\Psi}[\![\Gal(\Psi_{\infty}/\Phi)]\!]$.

Fix a generator  $h_\pm$ of the $\Lambda$-module 
$\mathcal{H}^1_\pm$. 
Let $\varepsilon \in \{+, -\}$ denote the sign of $\epsilon(\lambda)$ 
and  
\begin{equation}\label{RupL}
\mathscr{L}_{p}(\lambda):=\mathscr{L}_p(\lambda, \Omega_{K}, h_\varepsilon)  \in \Lambda
\end{equation}
the associated Rubin $p$-adic $L$-function \cite[\S10]{Ru}. 
Let $\nu(\mathscr{L}_{p}(\lambda))$ denote its evaluation at an anticyclotomic character $\nu$.
An interpolation property of the Rubin $p$-adic $L$-function is given by 
\begin{equation}\label{Rpi}
\text{
$\nu(\mathscr{L}_{p}(\lambda))
=\frac{1}{\delta_{\nu^{-1}}(h_\varepsilon)}\cdot \frac{L(1,\overline{\lambda\nu})}{\Omega_K}$ 
 \quad ($\nu \in \Xi_{\lambda,p}^{+}\setminus\{1\}$),
}
\end{equation}
where 
the non-vanishing of $\delta_{\nu^{-1}}(h_\varepsilon)$ 
is a consequence of Rubin's conjecture \cite{BKO}.

The Iwasawa $\mu$-invariant of Rubin's $p$-adic $L$-function is given by the following. 
\begin{thm}\label{thmC}
Let $\lambda$ be a Hecke character over an imaginary quadratic field $K$ of infinity type $(1,0)$ 
 such that 
 $\lambda \circ {\rm{N}}_{H/K}$ 
is associated to a $\BQ$-curve $E$ over $H$ with good reduction at a prime $p\nmid 6h_K$ inert in $K$. 
Let $\mathscr{L}_{p}(\lambda)$ be an associated Rubin $p$-adic $L$-function. 
Then 
$$
\mu(\mathscr{L}_{p}(\lambda))=
0. 
$$
\end{thm}
The above mod $p$ non-vanishing 
is based on Theorem~\ref{thmB} and the main result of \cite{BKOd}.

\begin{remark}\label{rm:lv}
In the setting of Theorem~\ref{thmC} the invariant $\mu_p(\lambda)$ vanishes. 
\end{remark}
\begin{cor}\label{corA}
 For $\lambda$ as in Theorem~\ref{thmC}, let $\mathscr{X}_{\varepsilon}(\lambda)$ denote the associated signed anticyclotomic Selmer group \cite{BKO} and 
 $\mathscr{X}_{\rm st}(\lambda)$ the two-variable strict Selmer group. 
 Then
$$
\mu(\mathscr{X}_{\varepsilon}(\lambda))=
\mu(\mathscr{X}_{\rm st}(\lambda))=0. 
$$
\end{cor}
\begin{proof}
The assertion for $\mathscr{X}_{\varepsilon}(\lambda)$ just follows from Theorem \ref{thmC} and signed Iwasawa main conjecture \cite[Thm.~6.1]{BKO}.
In particular, the $\mu$-invariant $
\mu(\mathscr{X}_{\rm st}^{\rm ac}(\lambda)) 
$ of the anticyclotomic strict Selmer group 
vanishes. Therefore, control theorem implies the same for $\mathscr{X}_{\rm st}(\lambda)$. 
\end{proof}

\begin{remark}\label{rm-cIw}
The above vanishing of $\mu$-invariants has an application to classical Iwasawa theory. Indeed, 
the Selmer group $\mathscr{X}_{\rm st}(\lambda)$ is a classical Iwasawa module over imaginary quadratic fields: the projective limit of $p$-part of class groups along the $\BZ_p^2$-extension of $K$ with tame action arising from $\lambda$. It is extensively studied since the 1970's by Coates--Wiles \cite{CoWi} and Rubin \cite{Ru91}, among others, yet determination of its $\mu$-invariant at inert primes remained elusive. 
\end{remark}

In combination with \cite[Thm.~1.1]{BKOe} we obtain the following. 

\begin{cor}\label{cor-D}
Let $E$ be a CM elliptic curve over $\Q$, and 
$K$ the associated CM field. 
Let $p\geq 5$ be a prime of good supersingular reduction for $E$, and $K_\infty$ the anticyclotomic $\Z_p$-extension of $K$. 
Then 
 there exists an integer $\lambda_{E,p} \in \Z_{\geq 0}$ such that  for any sufficiently large $n\ge 0$ with $(-1)^{n+1}=\epsilon(E)$,
the cokernel of the restriction map $$\Sha(E/K_{n-1})[p^{\infty}]\to \Sha(E/K_{n})[p^{\infty}]$$ is finite, and
$$
{\rm length}_{\Z_p}\left(\Coker\left(\Sha(E/K_{n-1})[p^{\infty}]\to \Sha(E/K_{n})[p^{\infty}]\right)\right)
=\lambda_{E,p}. 
$$
\end{cor}

\subsubsection{Relation to prior work}\label{ss:pr}
The characteristic zero non-vanishing of Hecke $L$-values as in Theorem~\ref{thmA} goes back to Greenberg \cite{Gr85} and Rohrlich \cite{Ro}.

As for {$(\ell,p)$ non-vanishing} in the case $\ell\neq p$,
Theorem~\ref{thmA} complements the existing results if $p$ remains { inert} in $K$. 
It is due to Finis \cite{Fin1} if $p$ splits in $K$. In the $\ell=p$ case, 
Theorem~\ref{thmB} is a new result for inert primes $p$. 
The split case is again due to Finis \cite{Fin2}, of which our method gives a different proof. 

If $\ell$ and $p$  are both split in $K$, then  
the problem has been studied by Hsieh \cite{Hs1,Hs2}, Ohta \cite{Oh} and the second-named author \cite{He}, based on Hida's idea \cite{Hi1,Hi2,Hi3}. However, in the $\ell \neq p$ case, the non-vanishing is established\footnote{The results were originally announced  for all but finitely many $\nu\in\Xi_{\lambda,p}^{+}$. However, a few years back, an issue was found in Hida's strategy \cite{Hi1,Hi2}.  His present fix \cite{Hi4} only allows infinitely many $\nu\in\Xi_{\lambda,p}^{+}$.} only  for {\it infinitely} many $\nu \in \Xi_{\lambda,p}^+$ (cf.~\cite{Hi4}).

\subsection{Strategy}
The mod $\ell$ non-vanishing \eqref{Q} concerns the arithmetic of $U(1)$ over $K$. 
Via theta correspondence, we recast it as a problem on a definite unitary group $U(2)$. Since the associated Shimura set is finite, the framework turns out to be  amenable to various tools, as outlined below. 

Some of the following notation differs from the rest of the paper.
\subsubsection{Definite Shimura set} 
Given an imaginary quadratic field $K$, we have a naturally associated quaternion algebra $B$ over $\BQ$ such that  \[\epsilon(B_q)=\eta_{K_q}(-1)\] 
for any prime $q$. Note that $K$ embeds into $B$ as a $\BQ$-algebra, and we often fix such an embedding.

In our study, $B$ is foundational to the arithmetic of $K$ 
in the guise of definite Shimura set. 
Though the association is natural, the arithmetic seems largely unexplored. As far as we know, the only other examples are Tian's work \cite{Ti,CST2,TYZ} on the congruent number and cube sum problems, jointly with Cai--Shu and Yuan--Zhang respectively (see also \cite{HSY}).

\subsubsection{Ancillary results} 
The mod $\ell$ non-vanishing is based on Theorems~\ref{thmD} and~\ref{thmE} below, which concern the arithmetic of the definite Shimura set.

For a self-dual Hecke character $\lambda$ of infinity type $(1,0)$, let 
$\phi_\lambda$ be the associated $\GL_2$-theta series of weight two. 

Let $\pi_\lambda$ be the cuspidal automorphic representation of $B^\times_{\BA}$ arising from Jacquet--Langlands transfer of 
$\phi_\lambda$ to $B^\times$. 
Let $\ell \nmid 2N_{K/\BQ}(\cond{\lambda})$ be a prime. 
Let $$\varphi_{\lambda}\in \pi_\lambda$$ be 
a CM form as in Definition \ref{D:1.W}. 
It is a test vector in the sense of Gross and Prasad \cite{GP}, which is $\ell$-primitive and $K_q^\times$-invariant for all primes $q|{\rm{N}}_{K/\BQ}(\cond{\lambda})$ non-split in $K$.  
We emphasise that $\varphi_\lambda$ is {not} a newform\footnote{In this setting the finite part of the discriminant of $B$ divides $D_K$. In turn, at primes dividing $D_K$, newform is not a test vector under an optimal embedding $K\hookrightarrow B$.}.

The period \[\Omega_{\lambda}=\frac{8\pi^2(\phi_\lambda,\phi_\lambda)}{\langle \varphi_{\lambda}, \varphi_{\lambda} \rangle}\]
naturally arises while studying Rankin--Selberg $L$-values $L(1/2,\pi_{\lambda,K}\otimes \nu)$ for $\nu\in\Xi_p$, 
Here $(\ ,\ )$ and $\langle \ , \ \rangle$ are Hermitian pairings on the space of $\GL_2$ and $B^\times$-modular forms, and $L(s,\pi_{\lambda,K}\otimes\nu)$ denotes the Rankin--Selberg $L$-function associated to the self-dual pair $(\pi_{\lambda},\nu)$. 
We normalise the latter so that $s=1/2$ is the center of the functional equation. 
In view of explicit Waldspurger formula due to Cai--Shu--Tian \cite{CST} the normalised $L$-value $L(1/2,\pi_{\lambda,K}\otimes \nu)/\Omega_{\lambda}$ is $\ell$-integral.  

We first prove the following $(\ell,p)$ non-vanishing.
\begin{thm}\label{thmD} 
Let $\lambda$ be a self-dual Hecke character over $K$ of infinity type $(1,0)$ and $\pi_\lambda$ the associated cuspidal automorphic representation. 
Let $\ell$ and $p$ be two different primes which are coprime to  $2{\rm{N}}_{K/\BQ}(\cond{\lambda})$.
Then for all but finitely many $\nu\in\Xi_{\lambda,p}^{+}$, we have
\[v_{\ell}\left(\frac{L(1/2, \pi_{\lambda, K}\otimes\nu)}{\Omega_{\lambda}}\right)=0.\]

\end{thm}
As for the period $\Omega_\lambda$, note that 
$$
\Omega_{\lambda}/\Omega_{K}^{2}\in\ov{\BQ}^\times 
$$
since
the normalised $L$-values $L(1,\lambda/\lambda^{c})/\Omega_{\lambda}$ and $L(1,\lambda/\lambda^{c})/\Omega_{K}^2$ are algebraic and non-zero, where $\lambda^{c}:=\lambda\circ c$ for $c\in\Gal(K/\BQ)$ the non-trivial element. 

\begin{thm}\label{thmE}
Let $\lambda$ be a self-dual Hecke character over an imaginary quadratic field $K$ of infinity type $(1,0)$. Let $\ell\nmid 2{\rm{N}}_{K/\BQ}(\cond{\lambda})$ be a prime. If $\epsilon(\lambda)=-1$, suppose that {$\ell\geq 5$}.
Then
   \[v_{\ell}\left(\frac{\Omega_\lambda}{\Omega_K^2}\right)= 2\mu_{\ell}(\lambda).\]
\end{thm}

The above results yield Theorem~\ref{thmA} 
in light of the factorisation
\begin{equation}\label{L-fac'}
L(1/2,\pi_{\lambda,K}\otimes\nu)=L(1,\lambda\nu)L(1,\lambda\nu^{-1})
\end{equation}
of $L$-values, and the lower bound \eqref{H_lb}.
\subsubsection{About Theorem~\ref{thmD}}\label{ss:lp-nvs}
The non-vanishing is based on explicit Waldspurger formula and an equidistribution of special points.

Let $X_U$ be the definite Shimura set associated to $B^\times$ and an open subgroup $U\subset B^\times$ corresponding to the test vector $\varphi_\lambda$. An apt choice of embedding $\iota: K\rightarrow B$ leads to special points $$x_{n}(a)\in X_U$$ for $[a]\in {{\rm G}_n}:=\Gal(H_{p^n}/K)$ and $H_{m}$ the ring class field of conductor {$m \in \BZ_{\geq 1}$}. By the explicit Waldspurger formula of Cai--Shu--Tian \cite{CST}, the mod $\ell$ non-vanishing of $L$-values  in Theorem~\ref{thmD} is equivalent to $\ell$-indivisibility of toric periods 
 \[P_{\varphi_\lambda}(\nu):=\sum_{[a]\in {{\rm G}_n}}\nu(a)\varphi_\lambda(x_n(a)),\]
where $\nu$ factors through {${\rm G}_n$}. The generality of \cite{CST} is essential in our study since the classical Heegner hypothesis is not satisfied, more precisely $D_K$ divides the conductor of $\pi_\lambda$ and $B$ is ramified at primes dividing $D_K$.

An idea of Vatsal \cite{Vatsal:nonvanishing} posits to study the $\ell$-indivisibility of toric periods $P_{\varphi_{\lambda}}(\nu)$ via equidistribution of images of special points $x_n(a)$ in a self-product of the Shimura set $X_U$ as $n$ varies. The equidistribution is a consequence of Ratner's seminal ergodicity of unipotent flows (cf.~\cite{Ra}). In turn it suffices to show that $\varphi_{\lambda} \mod{\ell}$ is non-Eisenstein on certain components of $X_U$. 
The prior non-Eisenstein argument \cite{Vatsal:nonvanishing} does not apply. 

In fact, a new phenomenon happens: there is a partition 
$
X_U=X_U^+ \sqcup X_U^-
$
where $$X_{U}^+:=\{[h]\in X_{U}\ \Big|\ {\rm{N}}(h)\in \BQ^\times_+\bs \BQ_+^\times {\rm{N}}(\wh{K}^\times)/{\rm{N}}(U)\},$$
and for $p$ inert, $\varphi_\lambda \mod{\ell}$ is non-Eisenstein on {\it exactly one} of the subsets $X_U^\pm$  depending on $\epsilon(\lambda)$. 
Our non-Eisenstein argument is indirect: $\varphi_\lambda \mod{\ell}$ is non-zero by definition, and consequently non-Eisenstein on $X_U$. (See Lemma~\ref{ee} which is specific to the CM setting.) So it is non-Eisenstein on at least one of the subsets $X_U^{\tilde{\varepsilon}}\in\{X_{U}^{+},X_{U}^{-}\}$. As the above non-vanishing strategy applies on $X_U^{\tilde{\varepsilon}}$, it follows that
$\tilde{\varepsilon}$ has the desired parity since $L(1/2,\pi_{\lambda,K}\otimes \nu)=0$ for $\nu\in\Xi_{\lambda,p}^-$.

\subsubsection{About Theorem~\ref{thmE}}\label{ss:per} The $\ell$-integral period relation in Theorem~\ref{thmE} - a comparison of automorphic and motivic periods - is a basic problem (cf.~\cite{Ha,PW,Pr1,IP}). 

For weight two newforms with square-free conductor the comparison is a consequence of Ribet's level raising (cf.~\cite{PW,Pr1}). It may also be approached via R=T theorems under the square-free-ness or a Gorenstein hypothesis (cf.~\cite{CH,KO,BKM}). These methods pertain to Hecke eigenforms. 
However, our CM setting\footnote{It is also excluded by conjectures of Prasanna \cite{Pr1} and Ichino--Prasanna \cite{IP} which concern the arithmetic of ratios of Petersson norms under Jacquet--Langlands correspondence. } is neither semistable nor does it involve an eigenform on the definite Shimura set.

Our roundabout strategy is based on a tenuous link of the $\ell$-adic valuation of $\Omega_{\lambda}/\Omega_K^2$ with mod $\ell$ non-vanishing as in Theorem~\ref{thmD}. 
A key observation: if there exists an auxiliary
 prime $p\nmid 2\ell{\rm N}_{K/\BQ}(\cond{\lambda})$ and 
a character $\nu\in\Xi_{\lambda,p}^+$ such that 
\begin{equation}\label{nv-cmp}
v_{\ell}\left(\frac{L(1,\lambda\nu)L(1,\lambda\nu^{-1})}{\Omega_K^2}\right)=2\mu_{\ell}(\lambda),
\end{equation}
then Theorem~\ref{thmD} implies\footnote{This relies on a key property of $\varphi_\lambda$: it 
is a universal test vector for primes $p\nmid {\rm{N}}_{K/\BQ}(\cond{\lambda})$ i.e. a test vector for self-dual pairs 
$(\pi_{\lambda},\nu)$ for any such $p$ and $\nu\in\Xi_p$.}
Theorem~\ref{thmE}! 

If $\epsilon(\lambda)=+1$, then we check the above criterion using Finis' mod $\ell$ non-vanishing  \cite{Fin1}: for any prime 
$p\nmid 2\ell{\rm N}_{K/\BQ}(\cond{\lambda})$ split in $K$, the main result of \cite{Fin1} provides the existence of $\nu$ as in \eqref{nv-cmp}.

Now suppose that $\epsilon(\lambda)=-1$. Proceeding as above, one may seek to choose a prime $p\nmid 2\ell{\rm N}_{K/\BQ}(\cond{\lambda})$ {\it inert} in $K$. 
(If $p$ splits, then $L(1,\lambda\nu)=0$.)
However, Finis' work \cite{Fin1} excludes  inert primes. 

The decisive idea is to bootstrap the problem by choosing an auxiliary $\nu_{0}\in\Xi_{p}$ such that $\epsilon(\lambda\nu_{0})=+1$ and apply the prior non-vanishing strategy to $\lambda\nu_0$.

We begin by showing a variant of Theorem~\ref{thmD}, which allows $p$ to divide the conductor (Note that $p|N_{K/\BQ}(\cond{\lambda\nu_0})$), and obtain
\begin{equation}\label{nv-new}
v_{\ell}\left(\frac{L(1/2,\pi_{\lambda\nu_{0},K}\otimes\nu)}{\Omega_{\lambda\nu_{0}}^{\{p\}}}\right)=0
\end{equation}
for inert $p$ and all but finitely many $\nu\in\Xi_{\lambda,p}^+$. 
This non-vanishing involves a different period $\Omega_{\lambda\nu_0}^{\{p\}}$, arising from an $\ell$-primitive 
 test vector $\varphi_{\lambda\nu_0}^{\{p\}}$ which is new at $p$ and the same as $\varphi_{\lambda\nu_0}$ at other primes. We emphasise that $\varphi_{\lambda\nu_{0}}$ itself does not work in the strategy as $K_p^\times$-invariant vectors in $\pi_{\lambda\nu_0}$ are not test vectors for self-dual pairs $(\pi_{\lambda\nu_{0}},\nu)$.

Now, to utilise \eqref{nv-new}, it suffices to determine {$v_{\ell}({\Omega_{\lambda\nu_{0}}^{\{p\}}}/{\Omega_K^2})$}. Since $\ell\geq 5$ and $p$ is auxiliary in regards to Theorem~\ref{thmE},  it maybe assumed that $\ell \nmid p(p^2-1)$. Then our principle result: 
\begin{equation}\label{per-ax}
v_{\ell}\left(\frac{\Omega_{\lambda\nu_{0}}^{\{p\}}}{\Omega_K^2}\right)=v_{\ell}\left(\frac{\Omega_{\lambda\nu_{0}}}{\Omega_K^2}\right)=2\mu_{\ell}(\lambda).
\end{equation}
In light of \eqref{per-ax} and \eqref{nv-new} it follows that 
$$
v_{\ell}\left(\frac{L(1,\lambda\nu)}{\Omega_{K}}\right)=\mu_{\ell}(\lambda)
$$
for all but finitely many $\nu\in\Xi_{\lambda,p}^+$. Therefore \eqref{nv-cmp} holds, concluding the proof of Theorem~\ref{thmE}. 

We now outline 
the period relation \eqref{per-ax}.
Its' second equality just follows from the aforementioned root number $+1$ case since $\epsilon(\lambda\nu_{0})=+1$. 
As for the first, 
 the strategy is based on yet another variant of Theorem~\ref{thmD} for a different prime $q$! 
 
Let $q\nmid 2p\ell {\rm N}_{K/\BQ}(\cond{\lambda})$ be a prime inert in $K$.  
We show that 
\begin{equation}\label{nv-new'}
v_{\ell}\left(\frac{L(1/2,\pi_{\lambda\nu_{0},K}\otimes\chi)}{\Omega_{\lambda\nu_{0}}^{\{p\}}}\right)
=0
\end{equation}
for all but finitely many $\chi \in \Xi_{\lambda,q}^+$. 
Since $$v_{\ell}\left(\frac{L(1/2,\pi_{\lambda\nu_{0},K}\otimes\chi)}{\Omega_{\lambda\nu_{0}}}\right)=0$$ 
by Theorem~\ref{thmD},
  the desired period relation \eqref{per-ax} follows. 
  
  Besides equidistribution of special points, the non-vanishing \eqref{nv-new'} rests on a new result on explicit construction of $\ell$-optimal test vectors for supercuspidal representations, which is the content of the next subsection.

\subsubsection{Newforms as $\ell$-optimal test vectors for supercuspidal representations}\label{ss:ntv'}
Given an irreducible admissible representation $\pi$ of $\PGL_2(\BQ_q)$ and a 
separable quadratic extension $K/\BQ_q$, a natural question: whether newform is a test vector for 
$\Hom_{K^\times}(\pi,\BC)$.
This local problem is linked with global arithmetic in view of Waldspurger and Gross--Zagier formulas
(cf.~\cite{Ti,HSY17}).

Now suppose that $q$ is odd and $K/\BQ_q$ the unramified quadratic extension. 
Let $\lambda$ be a character of $K^\times$ of {exponential} conductor $m\geq 2$ such that $\lambda|_{\BQ_q^\times}=\eta_{K},$ where $\eta_K$ is the quadratic character\footnote{In this local setting we still use the prior global notation such as $K$ and $\lambda$.} of $\BQ_q^\times$ corresponding to $K$. 
Let $\pi=\pi_\lambda$ be the associated supercuspidal representation of $\PGL_2(\BQ_q)$ and $$f\in \pi^{R^\times}$$ a newform for $R$ the standard Eichler order of discriminant $q^{2m}$. Note that $(\pi, 1)$ is a self-dual pair.

We consider $K^\times$-toric period of $f$ under a family of optimal embeddings $\iota: O_{K,q^m}\hookrightarrow R$, parameterized by a trace zero unit $\theta\in K$ and $u\in \BZ_q^\times$ such that $u^2\theta^2-1\in \BZ_q^{\times 2}$ (see~\S\ref{ss:lt-set}). We emphasise that the embedding of the unramified
torus in $\PGL_2(\Q_q)$ depends only on the conductor of the representation. 
Let $(\ ,\ )$ be a $\PGL_2(\BQ_q)$-invariant non-degenerate Hermitian pairing on $\pi$.
Define the toric period
$$
\gamma_{\theta,u}:=\frac{1}{\vol(K^\times/\BQ_q^\times)(f,f)}\int_{K^\times/\BQ_q^\times}(\pi(\iota(t))f,f)d^\times t.
  $$
  \begin{thm}\label{thmG}
   Let the setting be as above.
   \begin{itemize}
     \item [(a)] For a given $\theta$, there exists $u\in \BZ_q^\times$ as above such that the newform $f$ is a test vector for the self-dual pair $(\pi,1)$, that is,
      $$\gamma_{\theta,u}\neq 0.$$
     \item[(b)]Let $\ell\neq q$ be a prime. Then for a given $\theta$, there exists $u\in\BZ_q^\times$ as above such that 
     \[v_\ell((q^2-1)\gamma_{\theta,u})=0.\]
     \end{itemize} 
   \end{thm}

   In fact, we obtain an explicit formula for $\gamma_{\theta,u}$ in terms of $\lambda$ (see Theorem~\ref{ml}).
   
      The Kirillov model is central to the proof. We first interpret the toric period as a linear combination of epsilon factors of twists of $\pi_\lambda$. This relies on harmonic analysis in the framework of Kirillov model and action of Atkin--Lehner operators on twists of newforms. {Then, being in the supercuspidal case, the epsilon factor of a $\GL_2(\BQ_q)$-representation equals that of the associated character of $K^\times$. In turn, we relate the linear of combination of epsilon factors with values of $\lambda$ via an explicit epsilon factor formula and analysis of Jacobi sums.} This builds on the work of Murase and Sugano \cite{MS} on local primitive theta functions. 

      The test vector problem has been studied in the literature in special cases. For instance, the $m=1$ case of Theorem~\ref{thmG}(a) is due to Vastal  
       (cf.~\cite[Thm.~7.2]{Vatsal:23}). His notably different approach is based on representation theory of $\GL_2(\BF_q)$
       and ideas from Deligne--Lusztig theory (cf.~\cite{PS}). For $m\geq 2$, the test vector problem eluded prior methods. The reader may refer to \cite[Ch.~7]{Vatsal:23} for an overview (see also \cite{Vat}).

\subsubsection{About Theorems~\ref{thmB} and~\ref{thmC}}
Suppose that $p\nmid 6{\rm N}_{K/\BQ}(\cond{\lambda})$ is inert in $K$. Then $p$ is a prime of non-ordinary reduction for $\pi_\lambda$.

Based on the principle of non-ordinary Iwasawa theory \cite{Po,Ko},  we introduce a plus/minus $p$-adic $L$-function 
$$
\mathscr{L}_{p}(\pi_\lambda):=\mathscr{L}_{p}^{-\epsilon}(\pi_{\lambda}) \in \Lambda_{O}
$$
for a finite extension $O$ of $O_\Psi$ with $\Lambda_O:=O[\![\Gamma]\!]\simeq O
[\![T]\!]$ and $\epsilon$ the sign of $\epsilon(\lambda)$, such that 
$$
\nu(\mathscr{L}_{p}(\pi_\lambda)) \doteq 
p^{t+1} \nu (\Phi_{p}^{\epsilon}) \cdot 
\frac{L(1/2,\pi_{\lambda,K}\otimes \nu)}{\Omega_{\lambda}} .
$$
Here $\nu\in\Xi_{\lambda,p}^+$ is of order $p^t$, $\Phi_p^\epsilon$ denotes a half cyclotomic polynomial and $\doteq$ an equality up to $p$-units independent of $\nu$. 
The construction of $\mathscr{L}_{p}(\pi_\lambda)$ relies on an explicit Waldspurger formula on the definite Shimura set $X_U$ and the recipe in \cite{Po,Ko0}. In the case $\epsilon(\lambda)=-1$ it slightly differs from {\it loc. cit.}  
(cf.~Definition~\ref{def:prim}). 

Recall that the work of Pollack \cite{Po}, the third-named author \cite{Ko0} and Lei \cite{Lei} concerns Iwasawa theory of $\BZ_p$-cyclotomic deformation of an elliptic newform $f$ at supersingular primes $p$ for which $a_{p}(f)=0$. Its relevance to anticyclotomic deformations over an imaginary quadratic field $K$ was first noticed by Darmon and Iovita \cite{DI}. While they assume $p$ to be split in $K$, the inert case appears in recent study of the first-named author with Buyukboduk and Lei \cite{BBL,BBL2} (see also \cite{BLV}). 
The prior anticyclotomic studies exclude the case that $f$ has CM by $K$, basically due  to complications with explicit Waldspurger formula, which we analyse via the work of Cai--Shu--Tian \cite{CST}.

The $(\ell,p)$ non-vanishing strategy outlined in \S\ref{ss:lp-nvs} also applies to the $p$-adic $L$-function $\mathscr{L}_p(\pi_\lambda)$, leading to 
\begin{equation}\label{mu-van}
\mu(\mathscr{L}_p(\pi_\lambda))=0.
\end{equation}
Then Theorem~\ref{thmB}(b) just follows from the comparison of periods as in~Theorem~\ref{thmE}. 

As for Rubin's $p$-adic $L$-function, 
in light of $p$-adic Artin formalism and \eqref{L-fac'} one may expect a factorisation 
\begin{equation}\label{pL-fac}
\mathscr{L}_{p}(\pi_\lambda) \frac{\Omega_{\lambda}}{\Omega_{K}^2}= \mathscr{L}_{p}(\lambda)\mathscr{L}_{p}^{\iota}(\lambda),
\end{equation}
of $p$-adic $L$-functions 
up to an element in $\Lambda_O^\times$, 
where $\iota$ denotes the involution of $\Lambda_O$ arising from inversion on $\Gamma$. These $p$-adic $L$-functions live in distant worlds: $\mathscr{L}_p(\pi_\lambda)$ being automorphic and 
$\mathscr{L}_p(\lambda)$ an incarnation of a zeta element (elliptic unit). Moreover, local invariants in their interpolation formulas also differ, primarily due to which the factorisation remains open.  

We still prove an identity of $\mu$-invariants 
\begin{equation}\label{mu-fac}
\mu(\mathscr{L}_p(\pi_\lambda)) + v_{p}\left(\frac{\Omega_\lambda}{\Omega_K^2}\right) =2\mu(\mathscr{L}_{p}(\lambda))
\end{equation}
mirroring \eqref{pL-fac}. 
It is based on the main result of \cite{BKOd}, 
which determines the $p$-adic valuation of generalised Gauss sum $\delta_{\chi}(v_{\pm})$ appearing in the interpolation formula \eqref{Rpi} of Rubin's $p$-adic $L$-function. The latter employs ramification theory and builds on the proof of Rubin's conjecture. 
Finally, in light of \eqref{mu-van}, \eqref{mu-fac}, and Remark~\ref{rm:lv}, Theorem~\ref{thmE} concludes the proof of Theorem~\ref{thmC}. 

While the latter involves only a given prime $p$, a salient feature of our method is that it relies on $(\ell,p)$ non-vanishing for auxiliary primes $\ell\neq p$. 

\subsubsection{Further remarks on the work of Hida and Finis} This paper is independent from Hida's approach  \cite{Hi1,Hi2,Hi3,Hi4,Hs1,He}. 
The latter begins with  translation of the non-vanishing problem \eqref{Q} in terms of $\ell$-indivisibility of toric periods of a $\GL_2$-Eisenstein series, and studies it via Chai's theory of Hecke-stable subvarieties of a mod $\ell$ Shimura variety (cf.~\cite{Ch}). 
A crucial hypothesis is that $\ell$ splits in $K$ so that elliptic curve with CM by $O_K$ have ordinary reduction at $\ell$. The approach is flexible and applies to non self-dual Hecke characters $\lambda$. 

Finis' method \cite{Fin1,Fin2}  connects the non-vanishing \eqref{Q} to $\ell$-indivisibility of a linear functional on the space of $U(1)$-theta functions with complex multiplication, and  studies it via a Manin--Mumford conjecture in the context of Mumford's theory of theta functions. A key hypothesis is that $p$ splits in $K$. In contrast, our definite $U(2)$-setting first relates the non-anishing to $\ell$-indivisibility of toric periods of a CM form on a Shimura set, and an $\ell$-integral comparison of periods. 
While the former is independent from Finis' work, the latter relies on his $(\ell,p)$ non-vanishing \cite{Fin1} for a particular prime $p=p_0$ split in $K$. As a byproduct of our method, the non-vanishing holds for all other primes $p\nmid 2p_{0}N_{K/\BQ}(\cond{\lambda})$, including the missing case of inert primes $p$!

\subsection{Organisation} 
We begin with the framework of definite Shimura sets in section \ref{S:Gross_modular_forms}. 
Then sections \ref{S:specialvalue} and \ref{S:ThetaElment} present explicit Waldspurger formulas and construction of Rankin--Selberg $p$-adic $L$-functions, respectively. These sections treat general self-dual pairs, following which the CM hypothesis ensues. 
Section \ref{s:nv} constitutes the technical core of the paper, showing mod $\ell$ non-vanishing of Rankin--Selberg $L$-values in the CM case. Then section \ref{s:nv-Hecke} establishes the desired mod $\ell$ non-vanishing of Hecke $L$-values. It rests on an explicit construction of $\ell$-optimal test vectors for supercuspidal representations, which is the content of section~\ref{s:ntv}. This last section is purely local and may be of independent interest. 

\subsection{Vistas} Our $(\ell,p)$-non-vanishing method generalises to self-dual Hecke characters over $\ell$-ordinary CM fields. Such a non-vanishing is essential for completion of Hsieh's proof of Eisenstein congruence divisibility \cite{Hs3} towards the CM Iwasawa main conjecture\footnote{Due to Hida's weakening \cite{Hi4} of the $(\ell,p)$ non-vanishing results on Hecke $L$-values \cite{Hi1,Hi2}, Hsieh's work \cite{Hs3} on the CM main conjecture is incomplete. In \cite{Hs3} the non-vanishing \cite{Hs1} was crucially used for primitivity of an Eisenstein series on $U(2,1)$ over the CM field: the non-vanishing  holds for {\it all but finitely many} twists was decisive, which remained an open problem due to Hida's weakening until the very recent sequel \cite{BHT} by a subset of authors.} and will be presented in the very recent sequel \cite{BHT}. Our result on $\ell$-optimal test vectors for supercuspidal representations suggests new phenomena regarding correlation coefficients of representations of reductive groups over finite fields \cite{Gro91} and has led to new investigations (cf.~\cite{A}).

For the $(\ell,p)$-non-vanishing problem over imaginary quadratic fields, we hope to consider the mysterious case of primes $\ell$ ramified in the imaginary quadratic field in the near future. 
Another natural problem is to generalise main results of this paper to self-dual Hecke characters of general infinity type. The sought after factorisation \eqref{pL-fac} of $p$-adic $L$-functions will also be investigated. 
Our approach to $\ell$-optimal test vectors for supercuspidal representations appears to generalise to other settings.

The automorphic view on Rubin's supersingular theory presented in this paper may manifest to symplectic self-dual Galois deformations (cf.~\cite{BKNO}).

\subsection{Notation}\label{S:notation}

\subsubsection{Global fields}
Let {$F$} be a number field. Let $O_F$ denote its integer ring, $\BA_F$ the ad$\grave{ \text{e}}$les and $\BA_{F,f}\subset \BA_{F}$
 the finite part. In the text we fix an imaginary quadratic field $K$. Let $\eta_K$ denote the associated quadratic character over $\BQ$.

  Put $\BA=\BA_\Q$.

Fix an embedding
$\iota_\infty:\ov{\Q}\hookrightarrow\BC$ and an isomorphism
$\BC\simeq\BC_q$
for a prime $q$.
Let
$\iota_q:\ov{\Q}\hookrightarrow\BC_q$ be their composition. Let $v_q:\BC_q\to\Q\cup\{\infty\}$ be the $q$-adic valuation 
so that $v_q(q)=1$. We regard $L$ as a subfield of $\BC$ (resp. $\BC_q$) via
$\iota_\infty$ (resp. $\iota_q$) and $\Hom(L,\ov{\BQ})=\Hom(L,\BC_q)$.

Denote by $\wh{\BZ}$ the finite completion of $\BZ$. For an abelian group $G$, put 
$\wh{G}=G\otimes_\Z\wh{\BZ}$.

Write $$M=M^{+}M^{-}$$ for $M^{+}$ and $M^-$ divisible only by the split and non-split primes in $K$ respectively, and further  $$M^{-}=M_{\rm sf}^{-}M_{\rm{a}}^{-}$$ for $M_{\rm sf}^-$ the square-free part of $M^-$ so that $(M_{\rm sf}^{-}, M_{\rm a}^{-})=1$. 

For an algebraic group $G$ over  $\BQ$ and $q$ a prime, denote by $G_q$ the group of its 
$\BQ_q$-points.

\subsubsection{$L$-functions}
\subsubsection*{Local}Let $F=\BQ_q$ or $\BR$. 

Let $\sigma$ be an irreducible admissible representation of $\GL_2(F)$. Let $L(s,\sigma)$ and $\ep(s,\sigma,\psi_F)$ be the associated $L$-function and epsilon factor respectively (cf.~\cite[Thm.\,2.18 (iv)]{Jacquet_Langlands:GLtwo}), where $\psi_F$ is a non-trivial additive character of $F$.

Let $K$ be a quadratic extension of $F$, and $\sigma_K$ the base change of $\sigma$. 
{Let $\chi:K^\x\to\C^\x$ be a character.} 
Let $L(s,\sigma_{K}\ot\chi)$ be the Rankin--Selberg $L$-function as in ~\cite[\S 20]{Jacquet:GLtwoPartII}.

For an irreducible admissible representation of {$\GL_2(F)$} with trivial central character, we will simply denote {$\epsilon(1/2,\sigma,\psi_F)$ by $\epsilon(\sigma)$}, since it does not depend on the choice of 
$\psi_F$.
We adopt similar convention for representations of 
{$\GL_2(K)$} with trivial central character. 
\subsubsection*{Global}We consider Rankin--Selberg $L$-functions over $\BQ$. 

Let $\sigma$ be an irreducible cuspidal automorphic representation of $\GL_2(\BA)$. 
Let $K$ be a separable quadratic extension of $\BQ$ and 
$\chi:\BA_{K}^\times \ra \BC^\times$ an algebraic Hecke character. 
For $\sigma_K$ the base change of $\sigma$ to $K$, 
we have the associated Rankin--Selberg $L$-function defined by 
\[L(s,\sigma_{K}\otimes \chi)=\prod_{q\leq \infty}L(s,\sigma_{K_q}\otimes \chi_q).\] 
The automorphic $L$-function $L(s,\sigma_K\ot\chi)$ 
satisfies the functional equation
\[L(s,\sigma_K\ot\chi)=\ep(s,\sigma_K\ot\chi)L(1-s,\wt{\sigma}_K\ot\chi^{-1}),\]
where $\ep(s,\sigma_K\ot\chi)=\prod_\pme\ep(s,\sigma_{K_\pme}\ot\chi_\pme,\psi_{K_\pme}).$  If $(\sigma,\chi)$ is self-dual in the sense that $\omega_{\sigma}\chi|_{\BA^\times}=1$, where $\omega_{\sigma}$ is the central character, then the functional equation becomes 
\[L(s,\sigma_K\ot\chi)=\ep(s,\sigma_K\ot\chi)L(1-s,\sigma_K\ot\chi). \]
{Let $\epsilon(\sigma_K\otimes \chi):=\epsilon(1/2,\sigma_{K}\otimes \chi)$ denote the associated root number.

{ For $\Sigma$ a finite set of finite places of $\BQ$, let $L^{(\Sigma)}(s, \sigma_K \otimes \chi)$ denote the
incomplete $L$-function with Euler factors at $\Sigma \cup \{\infty\}$ removed. In this article we simply denote $L^{(\infty)}(s, \pi_K \otimes \chi)$ by 
$L(s, \pi_K \otimes \chi)$.
We use the same convention for Hecke $L$-functions.}

\subsection*{Acknowledgements} We thank Matthias Flach, Haruzo Hida, Christopher Skinner, Ye Tian and Wei Zhang for helpful discussions. We also thank U. K. Anandavardhanan, Li Cai, Dipendra Prasad and Mingrui Leng for instructive comments on section \ref{s:ntv}.

We are grateful to Karl Rubin for his inspiring theory and $p$-adic $L$-function ~\cite{Ru}, which sparked this voyage. 

This work was partially supported by the NSF grant DMS 2302064 and the JSPS
KAKENHI grants JP17H02836, JP18H05233, JP22H00096, JP17K14173, JP21K13774. During the preparation of the preprint, the first and the third-named authors were in residence at 
the Max Planck Institute for Mathematics, Bonn, and acknowledge its support. 

\section{Definite Shimura sets}\label{S:Gross_modular_forms}
The section introduces definite quaternion algebras, associated Shimura sets and
modular forms. 
\subsection{Set-up}\label{SS:setup}
\subsubsection{Imaginary quadratic field} Let $K$ be an imaginary quadratic field and $-D_K<0$ the discriminant. Put $\Diff=\sqrt{-D_K}$.

Write $z\mapsto \ov{z}$ for the complex conjugation on $K$. Define $\CMP\in K$ by
\[\CMP=\frac{D'+\delta}{2},\,D'=\begin{cases}D_{K}&\text{ if }2\nmid D_{K},\\
D_K/2&\text{ if } 2\mid D_K.
\end{cases}\]
Then $O_K=\BZ+\BZ\cdot\CMP$ and $\CMP\ov{\CMP}$ is a local uniformizer of primes that are ramified in $K$. 

Let $c\in\Gal(K/\BQ)$ be the non-trivial element. For a Hecke character $\lambda$ over $K$, put $\lambda^{c}=\lambda \circ c$.

\subsubsection{Definite quaternion algebra}\label{ss:def_set} Let $B$ be a definite quaternion algebra over $\Q$ and $S_B$ the set of finite places at which it ramifies.
Let $D_{B}=\prod_{q\in S_B} q$ be the discriminant of $B$. 

Write $\rm T$ and $\rm N$ for the reduced trace and norm of $B$ respectively. Let $G=B^\x$ be an  algebraic group over $\Q$. Note that $Z=\Q^\x$ is the center of $G$.

Fix a prime $p\notin S_B$. Let $S$ be a finite set of places containing $S_B$ and $\ell\notin S$ be a prime, which may equal $p$.
Let $\mathfrak{p}$ be the prime of $K$ above $p$ induced by the embedding 
$\iota_p:K\hookrightarrow\BC_p$.

Suppose that
\begin{itemize}
\item[\tiny{$\bullet$}] $K$ can be embedded into $B$, 
\item[\tiny{$\bullet$}] $-1\in \BQ_q^\times$ is a  norm of $K_q^\times$ for any $q\notin S_B$.
 \end{itemize}
 We often fix an embedding $\iota: K\hookrightarrow B$, and then 
 choose a basis $\{1,J\}$ of $B$ 
as a $K$-algebra such that
\begin{itemize}
\item[\tiny{$\bullet$}] $\cmJ^2=\beta\in\Q^\x$ with $\beta<0$ and $\cmJ t=\ov{t}\cmJ$ for all $t\in K$,
\item[\tiny{$\bullet$}] $\beta\in(\Z_q^\x)^2$ for all $q\in \{p,\ell\}\cup S\bs S_{B}$. 
\end{itemize}

{The existence of $J$ may be seen as follows:  recall that $B_q$ splits if and only if $\beta\in {\rm N}(K_q^\times)$. Take $k\in K^\times$ such that $-N(k)N(J)\in \BZ_q^{\times 2}$ for all $q\in \{p,\ell\}\cup S\bs S_{B}$. Then replacing $J$ by $Jk$ the second property holds.}

Let $\sqrt{\beta}\in\ov{\BQ}$ be a square root of $\beta$.

Fix an isomorphism 
\begin{equation}\label{eq:sp}
i=\prod i_q: \wh B^{(S_B)}\simeq \wh{M_2}(\BA_{f}^{(S_B)}) 
\end{equation}
as follows. For $q\in \{p,\ell\}\cup S\bs S_{B}$,
 define $i_q:B_q\simeq M_2(\Q_q)$  by
\begin{equation}\label{E:embedding.W}i_q(\theta)=
\begin{pmatrix}
{\rm T(\theta)} & {-\rm {\rm{N}}(\theta)}\\
1 & 0\\
\end{pmatrix};\quad i_q(\cmJ)=
\sqrt{\beta}\cdot 
\begin{pmatrix}
-1 &\rm T(\theta)\\
0 & 1\\
\end{pmatrix}\quad(\sqrt{\beta}\in\Z_q^\x).
\end{equation}
For a finite place $q\notin \{p,\ell\}\cup S$, choose $i_q:B_q\simeq M_2(\Q_q)$ such that
 \begin{equation}\label{E:21.W}i_{q}(O_K\ot\BZ_q)\subset M_2(\Z_q).
 \end{equation}
 We further choose $i_q$ so that $i_q(J)\in i_q(K_q^\times) \GL_2(\BZ_q)$ for all $q\notin S$.

 From now, we often identify $B_q$ with $M_2(\Q_q)$ via $i_q$ for finite $q\notin S_B$, and in turn $G_q$ with $\GL_2(\BQ_q)$.

\subsection{Modular forms}\label{SS:modular_forms}
\subsubsection{Classical modular forms} 
Let $A\subset \BC$ be a $\BZ$-algebra and $U\subset \wh{B}^\times$ an open compact subgroup. 

Let $M_{2}(U,A)$ be the space of modular forms of weight $2$, trivial central character defined over $A$, which consists of functions $f:\wh{B}^\times \to A$ such that
\[f(z\gamma g u)=f(g)\text{ for }\gamma\in G(\Q),\,u\in U,\ z\in \wh{\BQ}^\times.\]
Via right translation, $M_{2}(A):=\varinjlim_UM_{2}(U,A)$ is an admissible $G(\BA_f)$-representation. 
The space $M_{2}(\BC)$ can be identified with automorphic forms on $G(\BA)$ on which $\wh{\BQ}^\times G_\infty$ acts trivially.

\subsubsection{$\ell$-adic modular forms}\label{SS:l_adic_modular_forms}
\def\pvformB{\wh f_{\pi'}}
Let $\ell\nmid S$ be a prime
as in \S\ref{SS:setup}. 

Let $A\subset \ov{\BQ}$ be a $\BZ_{(\ell)}$-algebra. We regard it as a subalgebra of $\ov{\BQ}_\ell$ via the fixed embedding $\iota_\ell:\ov{\BQ}\ra \ov{\BQ}_\ell$.
For $f\in M_2(U,A)$, we refer to $\wh{f}:=\iota_\ell\circ f\in M_2(U,\ov{\BQ}_\ell)$ as the {$\ell$-adic avatar of $f\in M_{2}(U,A)$, often simply denoted by $f$}.

\subsection{Special points}\label{gpt}
This subsection describes an analogue of CM points in the definite setting, which arises from the 
$\BZ_p$-anticyclotomic extension of $K$.

Let $S^+\subset S$ (resp. $S^-\subset S$) be a subset consisting of primes that are split (resp.~non-split) in $K$. Note that $S^{+}\cap S_{B}=\emptyset$ since $K\hookrightarrow B$. 
For $q\in S^+$, choose a prime $w$ of $K$ above $q$. 

\subsubsection{Toric embedding}
\label{toremb}
We identify $G(\BA_{f}^{ (S_B)})$ with $\GL_2(\BA_f^{(S_B)})$ as in \eqref{eq:sp}.
For ${g,h\in \wh{B}^{\times}}$, put
\[\iota_{h}(g)=h^{-1}g h.\]

For a finite place $q\nmid p$, define $\cmptv_{q}\in G(\Q_q)$ by
\begin{equation}\label{E:cmptv.W}\begin{aligned}
\cmptv_{q}=&1\text{ if $q\nmid pS^+$,}\\
\cmptv_{q}=&\begin{pmatrix}
\CMP & \ov{\CMP}\\
1 & 1\\
\end{pmatrix}
\in\GL_2(K_w)=\GL_2(\Q_q)\text{ if $q\in S^+$.}
\end{aligned}\end{equation}
If $q\in S^+\setminus\{p\}$ and $t=(t_1,t_2)\in K_q:=K\ot_\Q\Q_q=K_w\oplus K_{\ov{w}}$, note that
\begin{equation}\label{E:cm2.W}\iota_{\cmptv_{q}}(t)=\begin{pmatrix}
t_1 & 0\\
0 & t_2\\
\end{pmatrix}.
\end{equation}

For a non-negative integer $n$, define $\cmptv_p^{(n)}\in G(\Q_p)$ as follows.
If $p$ splits in $K$ as $(p)=\mathfrak{p}\ov{\mathfrak{p}}$, then
\begin{align}\label{E:op1.W}
\cmptv_p^{(n)}=&\begin{pmatrix}
\CMP & -1\\
1 & 0\\
\end{pmatrix}
\begin{pmatrix}
p^n & 0\\
0 & 1\\
\end{pmatrix}
\in\GL_2(K_\mathfrak{p})=\GL_2(\Q_p).
\intertext{If $p$ is non-split in $K$, then}
\label{E:op2.B}\cmptv_p^{(n)}=&
\begin{pmatrix}
0 & 1\\
-1 & 0\\
\end{pmatrix}
\begin{pmatrix}
p^n & 0\\
0 & 1\\
\end{pmatrix}.
\end{align}

The above local embeddings lead to a family of embeddings $\iota_{\varsigma^{(n)}}: \wh{K}\ra \wh{B}$.
\subsubsection{Quaternionic order}\label{ss:qo} We introduce an order in the definite quaternion algebra, with respect to which special points will be introduced in the next subsection. 

Let $R\subset B$ be an order such that 
\begin{itemize}
  \item[\tiny{$\bullet$}]  For $q\notin S\cup\{p\}$, $R_q=M_2(\BZ_q)$;
  \item[\tiny{$\bullet$}]  For $q\in S_B$, $R_q$ contains $\iota_{\varsigma_q} O_{K_q}$;
  \item[\tiny{$\bullet$}] For $q\in S\bs (\{p\}\cup S_B)$, $R_q\cap \iota_{\varsigma_q} O_{K_q}$ is a fixed order in $\iota_{\varsigma_q} O_{K_q}$; 
  \item[\tiny{$\bullet$}]  $R_p$ is the standard Eichler order $$M_0(p^s)_p=\left\{\begin{pmatrix}
      a& b \\ c & d
  \end{pmatrix}\in M_2(\BZ_p)\ |\ p^s|c\right\}$$ of discriminant $p^s$ for an integer $s\geq 0$.   
  \end{itemize}
    In view of the last property and the choice of $\varsigma_p^{(n)}$ in \eqref{E:op1.W} and \eqref{E:op2.B}, we have $$R_p\cap \iota_{\varsigma_{p}^{(n)}}K_p=\iota_{\varsigma_p^{(n)}} O_{K_p,p^n}$$ for $n\geq s$, where $O_{K_p,p^n}$ is the order of $O_{K_p}$ of conductor $p^n$.

  Suppose that $$\wh{R}\cap \iota_{\varsigma^{(n)}} \wh{K}=\iota_{\varsigma^{(n)}} \wh{O}_{K,p^nc}$$ for $n\geq s$, the latter being order of conductor $p^nc$.

\subsubsection{Special points}\label{sppt}
Define $x_{n,c}:\BA_K^{\times }\to G(\BA)$ by
\begin{equation}\label{E:op3.W}x_{n,c}(a):=a\cdot \cmptv^{(n)}\quad(\cmptv^{(n)}:=\cmptv_{p}^{(n)}\prod_{q\neq p}\cmptv_{q}).
\end{equation}
This gives a family of special points 
$\{x_{n,c}(a)\}_{a\in\BA_{K}^{\times}}$.

For $c=1$, we denote $x_{n,c}(a)$ just by $x_{n}(a)$.

\section{Explicit Waldspurger formula}\label{S:specialvalue}
This section presents explicit Waldspurger formulas in a general context. 
It is based on the work of Cai--Shu--Tian \cite{CST} to which we refer for an introduction. {The main results are anticyclotomic twist family versions of the formula which involve a fixed test vector (cf.~Theorems~\ref{T:central.W} and~\ref{T:central.v1}).}

\subsection{Backdrop}
\subsubsection{Setting}\label{ss:set}
\noindent Let $K$ be an imaginary quadratic field. 

Let $\sigma$ be a unitary irreducible cuspidal automorphic representation of $\GL_2(\BA)$ with trivial central character and conductor $N$ such that 
\begin{itemize}
\item[(H1)] The archimedean component $\sigma_\infty$ is the discrete series of weight $2$;
\item[(H2)] $\epsilon(\sigma_{K})=+1$.
\end{itemize}
Here $\epsilon(\sigma_K)$ denotes the root number of the base change $\sigma_K$. 

The following example will be of particular interest for the paper. 
\begin{example}\label{ex1}
  Let $\lambda$ be a self-dual Hecke character over $K$ of infinity type $(1,0)$ in the sense of \eqref{theta}.
 Then the automorphic representation generated by the associated theta series {$\phi_{\lambda}$} satisfies the hypotheses (H1) and (H2).
  \end{example}

Let $B$ be the quaternion algebra over $\BQ$ such that the Tunnell--Saito condition
\begin{equation}\label{TS}
\epsilon(\sigma_{K,q})=\epsilon(B_q)
\end{equation}
holds for any place $q$, where $\epsilon(\sigma_{K,q})$ denotes the local base change root number and $\epsilon(B_q)$ is the Hasse invariant of $B_q$. It is a definite quaternion algebra.

\begin{lem}
 Suppose that
 \begin{itemize}
\item[(H3)] $\epsilon(\sigma_q)=-1$ for $q|(D_{K},N^{-}_{\rm sf})$.
\end{itemize}
  Then we have
  $N_{\rm sf}^- \mid D_{B} \mid N^-.
  $
  \end{lem}
  \begin{proof} 
  Let $q$ be a prime divisor of $N_{\rm sf}^-$.
It follows from \cite[Prop.~3.1.2]{Schmidt:newform} that the condition (H3) is equivalent to $\sigma_q=\chi_{0}\otimes\St$ for $\chi_0$ quadratic so that $\chi_{0}\circ {\rm{N}}_{K_q/\BQ_q}$ is trivial.
Thus {the result} is a consequence of \cite[Lem.~3.1]{CST}.

 \end{proof}
 
 \begin{lem}\label{lm:disc}
 {Suppose that $\pi$ has CM by $K$. Then}
  $$
  D_{B}=\prod_{\eta_{K_{q}}(-1)=-1}q.
  $$
  \end{lem}
  \begin{proof} 

   {Suppose that} $\sigma$ is generated by {$\phi_{\lambda}$} as in Example~\ref{ex1}. For each $q$, let $\psi_q$ be a non-trivial character of $\BQ_q$ and $\psi_{K_q}=\psi_q\circ\tr_{K_q/\BQ_q}$. Then 
    \[\begin{aligned}
\epsilon(\sigma_{K_q},\psi_{K_q})
=&\epsilon(\sigma_{q},\psi_{q})\epsilon(\sigma_{q}\otimes\eta_{K_q},\psi_{q})\eta_{K_q}(-1)\\
=&\epsilon(\lambda_q^*,\psi_{K_q})^2(\lambda_{K_q}(\psi_q)^2\eta_{K_q}(-1))\\
=&\epsilon(\lambda_q^*,\psi_{K_q})^2\\
=&\eta_{K_q}(-1). 
    \end{aligned},\] 
      Here $\lambda_q^*=\lambda_q\cdot |\cdot |_{K_q}^{1/2}$ denotes unitarisation of $\lambda_q$, the second equality follows from  \cite[Thm.~4.7]{Jacquet_Langlands:GLtwo}, and 
     the last  from 
    \[\begin{aligned}
\overline{\epsilon(\lambda_q^*,\psi_{K_q})}
=&\epsilon((\lambda_q^*)^{-1},\psi_{K_q})\cdot (\lambda_q^*)^{-1}(-1)\\
=&\epsilon(\lambda_q^{*,c},\psi_{K_q})\cdot \eta_{K_q}(-1)\\
=&\epsilon(\lambda_q^*,\psi_{K_q})\cdot \eta_{K_q}(-1).
    \end{aligned}\]
   
   In view of  \eqref{TS}  the proof concludes.
 \end{proof}

Let $G=B^\times$ be the algebraic group over $\BQ$.
Let $\pi=\ot\pi_q$ denote the Jacquet--Langlands transfer of $\sigma$ to $G(\BA)$, 
{which exists by Tunnell--Saito theorem \cite{Tu,Sa} and the condition \eqref{TS}.} It is a unitary irreducible cuspidal automorphic representation with trivial central character such that
\begin{itemize}
\item[\tiny{$\bullet$}] $\pi_\infty$ is the trivial representation of $G_\infty$, 
\item[\tiny{$\bullet$}] For $q\mid N_{\rm sf}^-$, $\pi_q$ is an unramified one dimensional representation of $G_q$.
\item[\tiny{$\bullet$}] If $q\nmid N$, then $\pi_q=\sigma_q$ is an unramified principal series $\pi(\mu_q,\mu_q^{-1})$ of 
$G_q=\GL_{2}(\BQ_q)$.
\end{itemize}

Now let $S$ be the set of prime divisors of $N$. 

Let $p\nmid D_B$ be a prime. 
Let $K_\infty$ be the anticyclotomic $\BZ_p$-extension of $K$ and $\Gamma=\Gal(K_{\infty}/K)$. 
Let $\Xi_p$ be the set of 
characters $\chi:\Gamma\to\BC^\x \text{ of finite order.}$

\begin{lem}\label{lm:eps}
For any $\chi\in\Xi_p$ with $\cond{\chi_{p}}\geq p^{v_p(N)}$, 
we have 
$\epsilon(\pi_K\otimes \chi)=+1. 
$
\end{lem}
\begin{proof}

{By the Tunnell--Saito theorem \cite{Tu,Sa}, we have
\begin{equation}\label{eq:mult}
\dim_{\BC} \Hom_{\BA_{K}^\times}(\pi,\BC)\leq 1,\end{equation}
and the equality is equivalent to the condition \eqref{TS}. }

Let $q$ be a prime.
If $q=p$, then $B_p$ is split, and so 
  $\Hom_{K_p^\times}(\pi_p,\chi_p^{-1})\neq 0$
   by \cite[Lem.~3.1]{CST}  since $\cond{\chi_p}\geq p^{v_p(N)}$. 
   Now consider the case $q\neq p$.
  If $q$ is  split in $K$, then $B_q$ is split and so
  $\Hom_{K_q^\times}(\pi_q,\chi_q^{-1})\neq 0$ by Tunnell--Saito.  
 Lastly, if $q$ is non-split in $K$, then $\chi_q$ is trivial and hence 
    $\Hom_{K_q^\times}(\pi_q,\chi_q^{-1})=\Hom_{K_q^\times}(\pi_q,\BC)\neq 0$
    by \eqref{eq:mult}.

It follows that $\dim_{\BC}\Hom_{\BA_{K}^\times}(\pi,\chi^{-1})=1$, concluding the proof. 

\end{proof}

\subsection{Test vectors}\label{SS:localtoric}  
For a self-dual pair $(\pi,\chi)$, 
the Waldspurger formula \cite{Wa} links the Rankin--Selberg $L$-value $L(\frac{1}{2},\pi_K\otimes \chi)$ with $K^\times$-toric period of a test vector on $G$.
 
 Recall that, following Gross and Prasad \cite{GP}, a form in $\pi$ is called a test vector if its image under a basis of $\Hom_{K_{\BA}^\times}(\pi,\chi^{-1})$ is non-zero {with respect to a suitable embedding $\BA_{K}\hookrightarrow B_{\BA}$}. 
  Let $p\nmid D_B$ be a prime. This subsection describes a choice of the test vector which is uniform 
 for any $\chi\in \Xi_p$ with conductor at least $p^{v_p(N)}$.

In the rest of this section we suppose that the hypotheses (H1)-(H3) hold. 

\subsubsection{Test vector} \label{test}
Fix a prime $p\nmid D_B$.
For $M_0(q^n)_q=\left\{\begin{pmatrix}
      a& b \\ c & d
  \end{pmatrix}\in M_2(\BZ_q)\ |\ q^n|c\right\}$ 
 the Eichler of discriminant $q^n$, let $U_0(q^n)_q=M_0(q^n)_q\cap \GL_2(\BZ_q)$. 
\begin{defn}\label{D:1.W} For a place $q$, define a non-zero vector $\lnew\in\pi_\pme$ as follows. 
\begin{itemize}
\item[(a)]For $\pme\mid N_{\rm sf}^-$, $\lnew$ is a basis of the one dimensional representation $\pi_\pme$ of $G_q$.
\item[(b)] For $\pme \mid N_{\rm a}^-$, $\lnew$ is invariant under the action of $K
_q^\times$ if $q\neq p$ and $\varphi_p$ is fixed by $ U_0(p^{v_p(N)})_p$.
\item[(c)] For $\pme\nmid N^-$, $\lnew$ is fixed by $ U_0(q^{v_q(N)})_\pme$.
\item[(d)]For $\pme=\infty$,  $\lnew$ is a non-zero element of the trivial representation $\pi_\infty$.
\end{itemize}
\end{defn}
\begin{remark}
 The above choice of $\varphi_q$ is as in \cite[\S3]{CST}.
\end{remark}

{
\begin{defn}\label{D:2.W}
Let $R\subset B$ be an order of discriminant $N$ satisfying the following.
\begin{itemize}
\item[(a)] For $q|N_{\rm sf}^-$, $R_q \subset B_q$ is maximal.
\item[(b)] For $q|N_{\rm a}^-$, $R_q \subset B_q$ so that $R_q \cap\iota_{\varsigma_q} K_{q}=\iota_{\varsigma_q} O_{K_q}$ if $q\neq p$, and $R_p$ is the Eichler order $M_0(p^{v_p(N)})_p$ if $q=p$. 
\item[(c)] For $q|N^{+}$, $R_q\subset B_q$ is the Eichler order $M_0(q^{v_q(N)})_q$.
\item[(d)] For $q\nmid N$, $R_q=M_2(\BZ_q)$.
\end{itemize}
\end{defn}
}
{
\begin{lem}\noindent
\begin{itemize}
\item [(i)]For any $q$, an order $R_q \subset B_q$ with discriminant $q^{v_q(N)}$ 
    as in Definition \ref{D:2.W}    
    exists and is unique up to $K_{q}^\times$-conjugation. Moreover, if $q|N^-$ and $q\nmid p$, then it is unique.
   \item [(ii)]If $q\nmid p$, then $R_q \cap\iota_{\varsigma_q} K_{q}=\iota_{\varsigma_q} O_{K_q}$ and if $q=p$, $$R_p\cap \iota_{\varsigma_p^{(n)}} O_{K_p}=\iota_{\varsigma_p^{(n)}} O_{K_{p},p^n}$$ for $n\geq  v_{p}(N)$. 
    \end{itemize}
\end{lem}
\begin{proof}
 For the existence and uniqueness of $R_q$, see \cite[Lem.~3.3]{CST} and \cite[Lem.~3.4]{CST} respectively. The second part just follows from the definition of $\varsigma_q$ and the choice of $R_q$.
 \end{proof}
}
Note that $R$ satisfies the properties in \S\ref{ss:qo} for $S$ consisting of prime factors of $N$.
\begin{lem}\label{lm:t-mult}\
\begin{itemize}
\item[(i)] For any $q$, there exists $\lnew\in\pi_q$ as in Definition \ref{D:1.W}.
 Moreover, it is unique up to scalars. 
\item[(ii)] For any $q$, we have $\varphi_q \in (\pi_{q})^{R_q^\times}$. 
\end{itemize}
\end{lem}
\begin{proof} \
\begin{itemize}
\item[(i)] It suffices to consider the cases (b) and (c). 
For (c) or (b) with $q=p$, the assertion is a simple consequence of the newform theory. As for the remaining case, in view of \eqref{TS} it follows from the Tunnell--Saito theorem \cite{Tu,Sa}. 
\item[(ii)] This is a special case of \cite[Prop.~3.8]{CST}.
\end{itemize}
\end{proof}
\begin{remark}
Note that $\lnew$ is a newform besides the case (b). 
  \end{remark}
\begin{defn}[Test vector]\label{D:tv}

Define $\varphi\in\pi$ by $$\varphi=\otimes_{q}\varphi_q$$ for $\varphi_q$ as in Definition \ref{D:1.W}. 
\end{defn}
Note that for any prime $p\nmid N^-$ the test vector $\varphi$ does not depend on $p$.
On the other hand, if $p|N_{\rm a}^-$, the test vector is new at $p$, and we use the notation {$\varphi^{\{p\}}$} to emphasise the dependence.

  The following preliminary will be used in our later arguments. 
  \begin{lem}\label{lm:lev}Suppose that $v_{q}(N)>  v_{q}(D_{K})$ for any prime $q|D_{K}$.
Then ${\rm{N}}(R_r^\times)= {\rm{N}}(O_{K_r}^\times)$ for any prime $r$. 
\end{lem}
\begin{proof}
Let $r\nmid pD_K$ be a prime. Then we have 
$O_{K_r}\subset R_r$. 
Since ${\rm{N}}: O_{K_r}^\times\rightarrow \BZ_r^\times$ is surjective for $r\nmid D_K$, it follows that 
 ${\rm{N}}(R_r^\times)={\rm{N}}(O_{K_r}^\times)$.

For $r=p$, note that $B_p$ split and $R_p=M_0(p^{v_p(N)})_p$ is an Eichler order, and so ${\rm{N}}(R_p^\times)\ra \BZ_p^\times$ is surjective.

Now consider the remaining case, namely $r$ is ramified in $K$.
Note that $R_r$ is of the form $O_{K_r}+\varpi^{n-1}O_{B_r}$, where $\varpi\in K_r$ is a uniformiser, $n=v_{r}(N)$ and $O_{B_r}$ the maximal order of $B_r$. 
We have \[O_{B_r}=O_{K_r}+\varpi^{v_{r}(D_{B})-v_{r}(D_{K})}O_{K_r}(1+J),\] where ${\rm{N}}(J)\in \BZ_r^\times$ and for $r=2$, further choose $$J^2\equiv 1\pmod{D_{K_r}/r}$$ to not lie in the norm of $K_{r}^\times$.
In this case,  $v_r(N)> v_{r}(D_{K_r})$, thus $R_r=O_{K_r}+\varpi^{n-v_{r}(D_{K_r})}O_{K_r}J$. 

 For $a\in O_{K_r}$ and $b\in \varpi O_{K_r}$, we have
 \[\begin{aligned}
  {\rm{N}}(a+bJ)=&(a+bJ)(\ov{a}+\ov{J}\ov{b})\\
=& {\rm{N}}(a)+{\rm{N}}(b){\rm{N}}(J)\\
\equiv& {\rm{N}}(a)\pmod{D_{K_r}}.\\
\end{aligned}\]
It follows that ${\rm{N}}(R_r^\times)={\rm{N}}(O_{K_r}^\times)$.
\end{proof}

\subsubsection{$p$-stabilization}\label{SS:pstabilization} 

Let $R\subset B$ be an order as in \S\ref{ss:qo} that is maximal outside $N$. 

We first suppose that $p\nmid N$.
Note that $G$ is split at $p$ and $\pi_p$ is an unramified principal series $\pi(\mu_p,\mu_p^{-1})$. 
Put 
\begin{equation}\label{Hk_p}
\alpha_p:=\mu_p(p)|p|_p^{-1/2},\ 
\beta_p:=\mu_p^{-1}(p)|p|_p^{-1/2},\
a_{p}=\alpha_p +\beta_p.
\end{equation}

The {{$p$-stabilization}} $f^\dagger := f_{\alpha_p}^\dagger$ of $f \in M_2(R,\BC)[\pi]$ with respect to $\alpha_p$ is defined by 
\begin{equation}\label{eq:st}
f^\dagger=f-\frac{1}{\alpha_p}\cdot \pi(\begin{pmatrix}
1 & 0\\
0&p\\
\end{pmatrix})f.
\end{equation}

Since the $U_p$-operator is given by
\[ U_ph(g)=\sum_{x\in\Z/p\Z} h(g\begin{pmatrix}
p & x\\
0 & 1\\
\end{pmatrix}
),\]
note that $f^\dagger$ is an $U_p$-eigenform with eigenvalue $\alpha_p$. 

If $p|N$ and $p\nmid D_B$, then we simply put $f^\dagger=f$.

In either case we have
\begin{equation}\label{E:8.W}\begin{aligned}
f^\dagger(x_{n,c}(\gamma au))=&f^\dagger(x_{n,c}(a_f))\\
(\gamma\in K^\x,\ &a=(a_\infty,a_f)\in\C^\x\x\widehat{K},\,u\in {\widehat{O}_{K,p^nc}^\x}).\end{aligned}
\end{equation}
\subsubsection{Normalised test vectors}\label{ss:ntv}

The following normalisation of test vectors will appear throughout the paper.

Let $\ell\nmid N$ be a prime
as in \S\ref{SS:setup}. 
Let $\Q(\pi)$ be the Hecke field of $\pi$, and 
$O_{\pi,\ell}\subset\C_\ell$ the completion 
of the ring of integers with respect to the prime $\fl|\ell$ determined via the embedding $\iota_\ell$.

Let $\varphi\in M_2(\wh{R}^\times,\BQ(\pi))$ be an {$\ell$-optimally normalised} test vector, i.e. $\varphi$ is a test vector as in Definition \ref{D:tv} such that
$$\varphi \in M_2(\wh{R}^\times, O_{\pi,\ell})$$ and 
$$\varphi\nequiv 0\pmod{\fl}.$$

The above normalisation is crucial in $\ell$-integrality of certain Rankin--Selberg $L$-values as well as construction of $\ell$-adic $L$-functions.

\subsection{Explicit Waldspurger formulas} The aim of this subsection is to explicitly link Rankin--Selberg $L$-values with toric periods of 
{the test vector introduced in Section \s\ref{SS:localtoric} or its variants.} 
\subsubsection{Setting}\label{ss:set-W} 
We begin with generalities regarding Waldspurger formula, and then specialise to the prior setting. 

Let $B$ be a quaternion algebra over $\BQ$. Let $\pi$ be an irreducible cuspidal automorphic representation of $B_{\BA}^\times$ and $\sigma$ its base change to $\GL_2(\BA)$.  
Let $K$ be an quadratic field with an embedding $K\hookrightarrow B$ and $\chi$ a Hecke character over $K$ such that \[\chi|_{\BA^\times}\omega_{\pi}=1\] for $\omega_{\pi}$ the central character of $\pi$. 
Suppose that \[\Hom_{K_{\BA}^\times}(\pi,\chi^{-1})\neq 0.\] 
Then the Waldspurger formula \cite{Wa,YZZ} connects the toric periods 
\[P_{\chi}(f)=\int_{K^\times \BA^\times\bs \BA_{K}^\times}\chi(t)f(t)dt,\quad f\in \pi \]
with the Rankin--Selberg $L$-value $L(1/2,\pi_K\otimes \chi)$.

Now let the setting, and in particular $B$ and $\pi$, be as in \S\ref{ss:set}. Let $\chi\in \Xi_p$ be a finite order anticyclotomic Hecke character. 
Let $R\subset B$ be an order as in Definition \ref{D:2.W}, and 
pick a test vector $\varphi$ as in Definition \ref{D:tv}. 

 Let {$X_{\wh{R}^\times}$} be the Shimura set $B^\x\bksl \widehat{B}^\times/\widehat{R}^\x$, whose elements may be chosen as a set of representatives in $\widehat{B}^\times=G(\BA_f)$. For $\varphi\in \pi^{\wh{R}^\times}$, define the inner product of $\varphi$ by
\begin{equation}\label{E:normvf.W}
\langle \varphi, \varphi \rangle:=\sum_{[g_i]\in X_{\wh{R}^\times}}\frac{1}{w_i}\cdot\varphi(g_i)^{2},\quad w_i:=[B^\x\cap  g_i\widehat{R}^\x g_i^{-1}:\BZ^\times].
\end{equation}
For $\phi$ the newform of level $\Gamma_0(N)$ associated to $\pi$, the Petersson norm is defined by \[(\phi,\phi)=\int_{\Gamma_0(N)\bs \CH} |\phi(z)|^2\frac{dxdy}{y^2},\] with $z=x+iy$.

For $\chi \in \Xi_p$ of conductor $p^s$ so that $s\geq v_p(N)$, {we have $\wh{R}\cap\iota_{\varsigma^{(s)}}\wh{O}_K=\iota_{\varsigma^{(s)}}\wh{O}_{K,p^s}$. For $n=\max\{1,s\}$, the toric period $P_{\chi}(\pi(\varsigma^{(n)})\varphi^\dagger)$} with respect to embedding $\iota$ is essentially given by
{ \[P(\varsigma^{(n)},\varphi^\dagger,\chi):=\sum_{[a]\in {\rm G}_n}\chi(a)\varphi^\dagger(x_n(a)),\]} where ${\rm G}_n=K^\times\bs\wh{K}^\times/\wh{O}_{K,p^n}^\times$ (cf.~\cite[Lem.~2.3]{CST}).

The following $p$-adic multiplier will also appear in  Waldspurger formulas:
\[e_p(\pi,\chi)=\begin{cases}1&\text{if $\chi_p$ is ramified};\\
(1-\alpha_{p}^{-1}\chi(\mathfrak{p}))(1-\alpha_{p}^{-1}\chi(\ov{\mathfrak{p}}))&\text{if $\chi_p$ is unramified, $p=\mathfrak{p}\ov{\mathfrak{p}}$ is split};\\
1-\alpha_{p}^{-2}&\text{if $\chi_p$ is unramified, $p=\mathfrak{p}$ is inert.}\\
\end{cases}
\]

\subsubsection{Explicit Waldspurger formula I}
The main result of this subsection is the following Waldspurger formula.

\begin{thm}\label{T:central.W}
  Let $(\pi,\chi)$ be a self-dual pair as in \S\ref{ss:set-W} with 
  $\pi$ of conductor $N$ and 
  $\chi\in \Xi_p$. Let $\varphi$ be an associated test vector as in Definition \ref{D:tv}, and $\varphi^\dagger$ its $\alpha_p$-stabilization if $p\nmid N$, and else put $\varphi^{\dagger}=\varphi$. Suppose that the $p$-local character associated to $\chi$ is of conductor $p^s$ with $s\geq v_p(N)$, and put $n=\max\{1,s\}$. 
  Then we have 
    \[\begin{aligned}      p^{-s}\cdot P(\varsigma^{(n)},\varphi^\dagger,\chi)^2=&\frac{\pair{\varphi,\varphi}}{8\pi^2(\phi,\phi)}\sqrt{|D_K|}\cdot L^{( p^{s} N_{\rm{r}})}(\frac{1}{2},\pi_K\otimes \chi) \\
       &\cdot \frac{\epsilon(\pi)}{\epsilon(\pi_p)}{2^{\#\Sigma_D}}\chi_{S^+\bs\{p\}}(\fN^+)\begin{cases}
      1,\quad &\text{$v_p(N)\geq 1$ or $s\geq 1$},\\
  e_p(\pi,\chi)^2\alpha_p^{2} ,\quad& v_p(N)= 0, s=0.\\
    \end{cases}\end{aligned}\] 
    Here $N_{\rm{r}}$ is the factor of $N_{\rm{a}}^-$ precisely divisible by the ramified primes and 
    $\Sigma_D$  the set of prime divisors of $(D_K,N)$ coprime to $p$, {$S^+=\{q\ |\ q|N^+\}$, $N^+=\fN^+\ov{\fN^+}$ with $w|\fN^+$ for $q=w\ov{w}$}, {and $\chi_{T}=\prod_{q\in T}\chi_q$.}
    \end{thm}

The above result is a consequence of a general explicit Waldspurger formula \cite[Thm.~1.8]{CST} as we now describe. 
In {\it{loc. cit.}} absolute value square of toric period appears, and the following analysis relates it to square of the toric period.
\begin{lem}\label{sc}
For $\chi\in\Xi_p$ of conductor $p^s$ with $s\geq v_p(N)$ and $n=\max\{1,s\}$, we have
\[P(J\varsigma^{(n)}\tau,\varphi^{\dagger},\chi)=P(\varsigma^{(n)},\varphi^{\dagger},\chi) \frac{\epsilon(\pi)}{\epsilon(\pi_p)}.\] 
\end{lem}
\begin{proof}
Let $S$ be a set of primes as in \S\ref{ss:def_set} given by prime factors of $N$. 

If $q\notin S$, note that $J\in K_q^\times\GL_2(\BZ_q)$ and $\varsigma_{q}=1$, and so the Hecke action of $J$ at $q$ does not change the toric period. 
For $q\in S^+\bs \{p\}$,  $$\varsigma_q^{-1}J\varsigma_q\tau\varphi_q=w_{\pi_q}\varphi_q=\epsilon(\pi_q)\varphi_{q},$$ where $w_{\pi_q}=\begin{pmatrix}
  &1 \\ q^{v_q(N)}
\end{pmatrix}\in \GL_2(\BQ_q)$ is the Atkin--Lehner operator (cf.~\cite[Thm.~3.2.2]{Schmidt:newform}).
If $q| N^-$ and $q\nmid p$, we have $\chi_q=1$ and $\Hom_{K_q^\times}(\pi_q,\BC)\neq 0$. Let $P_q$ be a basis. Then $J$ acts on $P_q$ by $\epsilon(\pi_q)\epsilon(B_q)$ \cite[Thm.~4]{DP}. Thus $$P_q(J\varsigma_q \varphi_q)=\epsilon(\pi_q)\epsilon(B_q) P_q(\varsigma_q\varphi_q).$$ 
Now consider the case $q=p$. 
Then $\varsigma_p^{(n),-1}J\varsigma_p^{(n)}$ stabilises $\varphi^\dagger$.

Therefore 
\[P(J\varsigma^{(n)}\tau,\varphi^{\dagger},\chi)=\prod_{q\in S,q\neq p}\epsilon(\pi_q)\epsilon(B_q)P(\varsigma^{(n)},\varphi^{\dagger},\chi),\]
and the result then follows from the fact that $\prod_{q|S,q\neq p}\epsilon(\pi_q)\epsilon(B_q)=\frac{\epsilon(\pi)}{\epsilon(\pi_p)}$. 
\end{proof}

\begin{proof}[Proof of Theorem \ref{T:central.W}] 
The following is based on \cite[Thm.~1.8]{CST}.

Taking $f$ in {\it ibid.} to be $\varsigma^{(s)}\varphi$, we have
\begin{equation}\label{eq1}
\left|\frac{P(\varsigma^{(s)},{\varphi},\chi)}{p^{s}[O_{K,p^s}^\times:\BZ^\times]}\right|^2={2^{\# \Sigma_D}}{}\frac{\pair{\varphi,\varphi}_{H,\wh{R}^\times}}{8\pi^2(\phi,\phi)_{U_0(N)}}\sqrt{|D_K|} 
\cdot L^{(p^sN_{\rm{a}})}(\frac{1}{2},\pi_K\otimes\chi),
\end{equation}where $\pair{\ ,\ }_{H,\wh{R}^\times}$ is Hermitian invariant pairing of level $\wh{R}^\times$.
Indeed, this follows from \cite[Thm.~1.8]{CST} since 
\[\pair{\varphi_1,\varphi_2}_{H,\wh{R}^\times}=\pair{\varphi_{1},\ov{\varphi}_{2}},\]
\[P_{\chi}^0(\varsigma^{(s)}\varphi)=P(\varsigma^{(s)},\varphi,\chi),\] $C_\infty=4\pi^3$, $2\pi \pair{\phi^0,\phi^0}_{U_0(N)}=(\phi,\phi)_{U_0(N)}$, and $\nu_{p^s}=[O_{K,p^s}^\times :\BZ^\times]$ in our setting, the notation being as in \cite{CST}.

Put \[\tau=\begin{pmatrix}
  N^+ & \\ &1
\end{pmatrix}\in \prod_{q|N^+, q\nmid p}\GL_2(\BQ_q)\subset \GL_2(\BA_{f}).\]
Then in view of the $S$-version of Waldspurger formula \cite[Thm.~1.9]{CST} we have  
\[\frac{P(\varsigma^{(s)},\varphi,\chi) P(\varsigma^{(s)}\tau,\ov{\varphi},\chi^{-1})}{p^{s}[O_{K,p^s}^\times:\BZ^\times]^2}=\frac{\pair{\varphi,\varphi}_{H,\wh{R}^\times}}{8\pi^2(\phi,\phi)_{U_0(N)}}\cdot \sqrt{|D_K|}\chi_{S^+\bs\{p\}}(\fN^+){2^{\# \Sigma_D}}{} \cdot L^{(p^sN_{\rm{a}})}(\frac{1}{2},\pi_K\otimes\chi).\]

We now analyse the left hand side. 
Since $\pi=\ov{\pi}$, by multiplicity one of $\varphi$ (cf.~Lemma~\ref{lm:t-mult}), note that 
$\varphi=C\ov{\varphi}$ for a non-zero constant $C\in \BC$. 
\vskip2mm

\underline{Case I}. Suppose that $v_{p}(N)\geq 1$. 

Then by definition, $\varphi^\dagger=\varphi$. 
So we have
\[\begin{aligned}
  \frac{P(\varsigma^{(s)},\varphi,\chi) P(\varsigma^{(s)}\tau,\ov{\varphi},\chi^{-1})}{\pair{\varphi,\varphi}_{H,\wh{R}^\times}}
  =&\frac{P(\varsigma^{(s)},\varphi,\chi) P(\varsigma^{(s)}\tau,{\varphi},\chi^{-1})}{\pair{\varphi,\varphi}_{R}}\\
  =&\frac{P(\varsigma^{(s)},\varphi,\chi)P(J\varsigma^{(s)}\tau,{\varphi},\chi)}{\pair{\varphi,{\varphi}}_{R}}\\
  =&\frac{\epsilon(\pi)}{\epsilon(\pi_p)}\frac{P(\varsigma^{(s)},\varphi,\chi)P(\varsigma^{(s)},{\varphi},\chi)}{\pair{\varphi,{\varphi}}_{R}}\quad (\text{Lemma \ref{sc}})\\
  =&\frac{\epsilon(\pi)}{\epsilon(\pi_p)}\frac{P(\varsigma^{(s)},\varphi,\chi)^2}{\pair{\varphi,{\varphi}}_{R}}.
\end{aligned}\]
{Here the first equality follows from the multiplicity one of $\varphi$ (cf.~Lemma~\ref{lm:t-mult}), and the second from {automorphy of $\varphi$} and the fact $\ov{t}=J^{-1}t J$.}

Therefore, noting that  $[O_{K,p^n}^\times:\BZ^\times]=1$ for $n\geq 1$, 
 the result is a consequence of \eqref{eq1}. 

\vskip2mm 
\underline{Case II}. {Suppose that $v_{p}(N)=0$.}

Then we similarly have 
\begin{equation}\label{eq2}
\frac{P(\varsigma^{(n)},\varphi^\dagger,\chi) P(\varsigma^{(n)}\tau,\ov{\varphi_{\beta_p}^\dagger},\chi^{-1})}{\pair{\varphi,\varphi}_{H,\wh{R}^\times}}=\frac{\epsilon(\pi)}{\epsilon(\pi_p)}\frac{P(\varsigma^{(n)},\varphi^\dagger,\chi)^2}{\pair{\varphi,{\varphi}}_{R}},
\end{equation}
 where the only difference in this case is that $\ov{\varphi_{\beta_p}^\dagger}=C\varphi^\dagger$ (recall that $\varphi^{\dagger}:=\varphi^{\dagger}_{\alpha_p}$). In the following we consider these toric periods. 

Henceforth, without loss of generality, we suppose that $L(1/2,\pi_K\otimes\chi)\neq 0$. 
By the Waldspurger formula and multiplicity one of $\Hom_{\BA_K^\times}(\pi,\chi^{-1})$, we then have
\begin{equation}\label{eq3}
\begin{aligned}
  &\frac{P(\varsigma^{(n)},\varphi^\dagger,\chi) P(\varsigma^{(n)}\tau,\ov{\varphi_{\beta_p}^\dagger},\chi^{-1})}{[{\rm G}_n:{\rm G}_s]^2}
  =P(\varsigma^{(s)},\varphi,\chi) P(\varsigma^{(s)}\tau,\ov{\varphi},\chi^{-1})\\
  &\cdot \int_{K_p^\times/\BQ_p^\times}\frac{(\iota_{\varsigma^{(n)}}(t)\varphi_p^\dagger,\varphi^\dagger_{\beta_p,p})\chi_p(t)}{(\varphi_p,\varphi_{p})}d^\times t\cdot \left(\int_{K_p^\times/\BQ_p^\times}\frac{(\iota_{\varsigma^{(s)}}(t)\varphi_p,\varphi_p)\chi_p(t)}{(\varphi_p,\varphi_p)} d^\times t\right)^{-1},
\end{aligned}
\end{equation}
 where the local invariant pairings are Hermitian. {Here 
 $$P(\varsigma^{(s)},\varphi,\chi)=\sum_{[a]\in {\rm G}_s}\chi(a)\varphi(x_s(a))$$ with ${\rm G}_s=K^\times\bs\wh{K}^\times/\wh{O}_{K,p^s}^\times$.}

By \cite[Prop.~3.12]{CH}, 
\[\int_{K_p^\times/\BQ_p^\times}\frac{(\iota_{\varsigma^{(s)}}(t)\varphi_p,\varphi_p)}{(\varphi_p,\varphi_p)}\chi_p(t)d^\times t=\begin{cases}
|D_K|_p^{1/2}c_p,\quad s=0,\\
|D_K|_p^{1/2}c_pL(1,\eta_{K_p})^2p^{-s},\quad s>0,\\
\end{cases}\]
where \[c_p=\frac{L(2,1_{\BQ_p})L(1/2,\pi_{K_p}\otimes\chi_p)}{L(1,\eta_{K_p})L(1,\pi_{p}, \ad)}.\]
Recall that for normalised spherical Whittaker functional $W_{\pi_p}$ {so that $W_{\pi_p}(1)=1$}, we have \[(W_{\pi_p},W_{\pi_p})=\frac{L(1,\pi,\ad)L(1,1_{\BQ_p})}{L(2,1_{\BQ_p})}\] for $(\ ,\ )$ the standard Hermitian pairing on the Whittaker model (cf.~\cite[Prop.~3.11]{CST}). In combination with \cite[Prop.~3.10]{CH}, {which is an explicit toric period formula for stabilized newforms with respect to the Hermitian pairing}, we have
\[\int_{K_p^\times/\BQ_p^\times}\frac{(\iota_{\varsigma^{(n)}}(t)\varphi_{\alpha_p,p}^\dagger,\varphi^\dagger_{\beta_p,p})\chi(t)d^\times t}{(\varphi_p,\varphi_p)}=|D_K|_p^{1/2}c_p\cdot L(1,\eta_{K_{p}})^2\cdot\begin{cases}e_p(\pi,\chi)^{2}\alpha_p^{2}p^{-2}&\text{ if $s=0$},\\
  p^{-s}&\text{ if $s>0$}.\end{cases}\]
Note that $$[{\rm G}_1:{\rm G}_0]L(1,\eta_{K_p})p^{-1}=\frac{1}{[O_{K}^\times:O_{K,p}^\times]}.$$ 
Therefore, in view of the previous paragraph and \eqref{eq1},~\eqref{eq2},~\eqref{eq3}, the proof concludes.

\end{proof}

\subsubsection{Explicit Waldspurger formula II}\label{var:test}

In this subsection we consider a choice of test vector for self-dual pairs $(\pi,\chi)$ 
which differs from \S\ref{test}. Specifically, newform is a test vector at certain primes $q$ so that $\pi_q$ is supercuspidal and $\chi_q=1$, as shown in section \ref{s:ntv}. This choice will be a key to subsequent applications.

\subsubsection*{Setting}
We consider self-dual pairs $(\pi,\chi)$ for $\chi\in \Xi_p$ as in \S\ref{ss:set-W}.
Let $q\neq p$ be a prime such that 
\begin{itemize}
  \item[\tiny{$\bullet$}] $q$ {is an odd prime}  inert in $K$, 
  \item[\tiny{$\bullet$}] $B_q$ split and $\pi_{q}=\pi_{\lambda}$ is the CM lifting of a character $\lambda$ of $K_{q}^\times$ with conductor ${q^m}$ for $m\geq 2$ such that $\lambda|_{\BQ_q^\times}=\eta_{K_q}$.
\end{itemize}

\begin{defn}
A test vector $\widetilde{\varphi}=\otimes_v\widetilde{\varphi}_v $ for $(\pi,\chi)$ is chosen to be the following:
\begin{itemize}
    \item[\tiny{$\bullet$}] If a prime $r\nmid q$, then {$\widetilde{\varphi}_{r}=\varphi_r$} is as in \S\ref{test}.
    \item[\tiny{$\bullet$}] If $r=q$, 
    let $R_{q}$ be the Eichler order $M_0(q^{2m})_q$ of discriminant $q^{2m}$ under the identification $i_{q}$ as in \S\ref{ss:def_set}. Let {$\widetilde{\varphi}_{q}\in 
 \pi_{q}^{R_{q}^\times}$} be a newform. 
\end{itemize}
\end{defn}
That {$\widetilde{\varphi}_{q}$} is a test vector for the pair $(\pi_q,1)$ is the main result of section \ref{s:ntv}, which is a new contribution to  explicit construction of test vectors.

For a prime $r\neq q$, let $\varsigma_r^{(n)}$ be as in \S\ref{gpt}.
If $r=q$, we choose $\varsigma_q$ so that {$$R_{q}\cap \iota_{\varsigma_q}K_{q}=\iota_{\varsigma_{q}} O_{K_q,q^{m}}.$$}
Let $\theta\in K$ be a unit so that $\ov{\theta}=-\theta$, where $\ov{\cdot}$ denotes the action of non-trivial element in $\Gal(K/\BQ_{q})$.
Let $u\in \BZ_q^\times$ be such that
$u^2\theta^2-1\in \BZ_q^{\times 2}.$ {Choose $\varsigma_q$ such that} 
\[\iota_{\varsigma_q}(\theta)=\begin{pmatrix}
  q^{-m}&\\ & 1
\end{pmatrix}\begin{pmatrix}
 1 &-u\\ &1
\end{pmatrix}\begin{pmatrix}
  &1\\
  \theta^2&
\end{pmatrix}\begin{pmatrix}
 1 &u\\ &1
\end{pmatrix}\begin{pmatrix}
 q^m &\\ & 1
\end{pmatrix}.\]

 For $s\geq v_p(N)$, we have {$$\wh{R}\cap\iota_{\varsigma^{(s)}} \wh{O}_K=\iota_{\varsigma^{(s)}}\wh{O}_{K,p^sq^{m}}.$$ }
 Define CM points {$x_{n,q^m}(a)$} as in \S\ref{gpt}.
 
 For {$n=\max\{1,s\}$, consider}
 {
  \[P(\varsigma^{(n)},\widetilde{\varphi}^\dagger,\chi):=\sum_{[a]\in {\rm G}_{n,q}}\widetilde{\varphi}^\dagger(x_{n,q^m}(a))\chi(a),\]} where ${\rm G}_{n,q}=K^\times\bs\wh{K}^\times/\wh{O}_{K,p^nq^m}^\times$.

 \subsubsection*{Result}

\begin{thm}\label{T:central.v1} 
  Let $(\pi,\chi)$ be as in \S\ref{ss:set-W} with $\chi\in \Xi_p$. Suppose that the $p$-local character associated to $\chi$ is of conductor $p^s$ with $s\geq v_p(N)$, and put $n=\max\{1,s\}$. Then we have 
    \[\begin{aligned}
     p^{-s}\cdot P(\varsigma^{(n)},\wt{\varphi}^\dagger,\chi)^2=&\frac{\pair{\varphi,\varphi}}{8\pi^2(\phi,\phi)}\sqrt{|D_K|}\cdot L^{(p^s{N_{\rm{r}}})}(\frac{1}{2},\pi_{K}\otimes \chi)\\
       &\cdot \frac{\epsilon(\pi)}{\epsilon(\pi_p)}{2^{\#\Sigma_D}}\chi_{S^+\bs\{p\}}(\fN^+)\cdot {[{\rm G}_{n,q}:{\rm G}_{n}]^2}\gamma_q \begin{cases}
      1,\quad &\text{$v_p(N)\geq 1$ or $s\geq 1$}\\
  e_p(\pi,\chi)^2\alpha_p^{2} ,\quad& v_p(N)= 0, s=0,\\
    \end{cases},\end{aligned},\] 
    where $N_{\rm{r}}$ is the factor of $N_{\rm{a}}^-$ precisely divisible by the ramified primes, 
    $\Sigma_D$ the set of prime divisors of $(D_K,N)$ coprime to $p$, $S^+=\{q\ |\ q|N^+\}$, $N^+=\fN^+\ov{\fN^+}$ with $w|\fN^+$ for $q=w\ov{w}$, {$\chi_{T}=\prod_{t\in T}\chi_t$}, and
 $\gamma_q:=\gamma_{\theta,u}$ is as in Theorem \ref{ml}. Moreover, the following holds. 
 \begin{itemize}
  \item [(a)] For given $\lambda$ and $\theta$, there exists  $u$ such that $\gamma_{\theta,u}\neq 0$.
  \item[(b)]Let $\ell\nmid q$ be a prime. Then for a 
  given $\theta$, there exists $u$ such that 
  \[v_\ell((q^2-1)\gamma_{\theta,u})=0.\]
  \end{itemize}

    \end{thm}
   \begin{proof}
The assertion is a consequence of Theorem~\ref{T:central.W}, $S$-version of explicit Waldspurger formula  \cite[Thm.~1.9]{CST}, local toric period formula for newform at $q$ as in ~Theorem~\ref{ml}, and Corollary~\ref{cml}. 
  \end{proof}
\section{$p$-adic $L$-functions}\label{S:ThetaElment}
We introduce 
$p$-adic $L$-functions interpolating  Rankin--Selberg $L$-values. {The main results are their constructions for general self-dual pairs (cf.~ Theorems~\ref{T:Thetaevaluation.W} and~\ref{p-adic:RS}). In the supersingular CM inert, we also compare the associated Iwasawa invariants  with that of Rubin's $p$-adic $L$-function (cf.~Proposition~\ref{prop:rel}).}
\subsection{Theta elements}\label{SS:thetalets}
\def\whmForm{\wh\vp^{[m]}_{\pi'}}

\subsubsection{Definition}\label{ss:theta-def}
Let the setting be as in  \S\ref{ss:set}.

Let $n\geq 0$ be an integer. {Recall that $${\rm G}_n=K^\x\backslash \wh{K}^\times/\wh{O}_{K,p^n}^\x$$ is the associated Picard group of $O_{K,p^n}$.} We identify ${\rm G}_n$ with the Galois group of the ring class field of conductor $p^n$ over $K$ via geometrically normalised reciprocity law. Denote by \[ [\cdot ]_n:\wh K^\x\to {\rm G}_n,\,a\mapsto [a]_n \] the natural projection map.

Let $\varphi\in \pi $ be the $\ell$-optimally normalised test vector as in \S\ref{ss:ntv}. For $p\nmid N$, recall that $\varphi^\dagger$ is the stabilization of $\varphi$ with respect to a root $\alpha_p$ of the Hecke polynomial at $p$ as in \eqref{Hk_p}. We occasionally adopt the convention that $\alpha_{p}=1$ if $p|N$.

\begin{defn}\label{D:Theta.W} 
The $n$-th theta element $$\Thetam_n(\pi)\in O_{\pi,\ell}[\alpha_{p}^{-1}][{\rm G}_n]$$ is defined as 
\begin{align*}\Thetam_n(\pi):=\begin{cases}
\displaystyle \alpha_p^{-n}\cdot \sum_{a\in {\rm G}_n}\varphi^\dagger (x_n(a))\cdot [a]_n,&\quad p\nmid N,\\ 
 \displaystyle \sum_{a\in {\rm G}_n}\varphi(x_n(a))\cdot [a]_n,&\quad p|N.
\end{cases}
\end{align*}
\end{defn}
We have the following compatibility. 
\begin{lem}\label{L:7.W} Suppose that $p\nmid N$, let $\pi_{n+1,n}:{\rm G}_{n+1}\to {\rm G}_n$ be the natural quotient map. For $n\geq 1$, we have
\[\pi_{n+1,n}(\Theta_{n+1}(\pi))=\Theta_n(\pi).\]
\end{lem}
\begin{proof}
The assertion follows from \[\alpha_p\varphi^\dagger (x_{n}(a))=U_p\varphi^\dagger (x_{n}(a))=\sum_{[u]_{n+1}\in \ker {\rm G}_{n+1}\ra {\rm G}_{n}}\varphi^\dagger (x_{n+1}(ua)).\]
 \end{proof}
\subsubsection{Interpolation}
Let $\phi\in S_2(\Gamma_0(N))$ be the normalised elliptic newform corresponding to $\sigma$. 

 Define a {period} $\Omega_{\pi}$ by
\begin{equation}\label{E:periodV.W}\Omega_{\pi}:=\frac{8\pi^2(\phi,\phi)}{\pair{\varphi,\varphi}},
\end{equation} {where $\varphi\in \pi$ denotes the $\ell$-optimally normalised test vector as before.}

Let $\chi\in \Xi_p$ be with conductor $p^s$.  
For $n\geq \max\{s,1\}$, note that 
\begin{align*}\chi(\Thetam_n(\pi))=
&
\begin{cases}
  \displaystyle \alpha_p^{-n}\cdot \sum_{a\in {\rm G}_n}\varphi^\dagger (x_n(a))\cdot \chi(a),&\quad p\nmid N,\\ 
   \displaystyle \sum_{a\in {\rm G}_n}\varphi(x_n(a))\cdot \chi(a),&\quad p|N.
  \end{cases}
\end{align*}

\begin{prop}\label{P:evaluation.W}
Let $(\pi,\chi)$ be as in \S\ref{ss:set}. 
Suppose that $\chi\in \Xi_p$ is of conductor $p^s$ with $s\geq v_p(N)$. 
 Then for $n\geq \max\{s,1\}$, we have

 \[\begin{aligned}      \chi(\Thetam_n(\pi)^2)=&\sqrt{|D_K|}\cdot \frac{L^{( p^{s} N_{\rm{r}})}(\frac{1}{2},\pi_K\otimes \chi)}{\Omega_{\pi}} \\
       &\cdot p^{s}\chi_{S^+\bs\{p\}}(\fN^+)\frac{\epsilon(\pi)}{\epsilon(\pi_p)}{2^{\#\Sigma_D}}\begin{cases}
      1,\quad &v_p(N)\geq 1\\
  \alpha_p^{-2s},\quad & s\geq 1\\
  e_p(\pi,\chi)^2 ,\quad& v_p(N)= 0, s=0.\\
    \end{cases}\end{aligned}\]

\end{prop}
\begin{proof}
This is a simple consequence of Theorem \ref{T:central.W} .
\end{proof}

\subsection{$p$-adic $L$-functions: ordinary case}\label{pordfun} Theta elements readily lead to an integral $p$-adic $L$-function in the ordinary case. 

In this subsection we suppose that $\ell=p\nmid 2N$. Fix embeddings $\iota_{\infty}:\ov{\BQ}\hookrightarrow \BC$ and $\iota_{p}:\ov{\BQ}\hookrightarrow\ov{\BC}_p$. 

Let ${\rm G}_\infty:=\varprojlim_n {\rm G}_n$. Let $\Gamma\simeq\Z_p$ be the maximal $\Z_p$-free quotient group of ${\rm G}_\infty$ and $\Delta$ the torsion subgroup of ${\rm G}_\infty$. 
Fix a non-canonical isomorphism $${\rm G}_\infty\simeq\Delta\x\Gamma.$$ If $n\geq 1$, then 
\[{\rm G}_n\simeq \Delta\x \Gamma_n,\,\Gamma\twoheadrightarrow\Gamma_n:={\rm G}_n/\Delta.\] 

Let $1_{\Delta}:\Delta\to\ov{\Q}^\x$ be the trivial branch character. Put 
$$O=O_{\pi,p}[\alpha_{p}] ,\ \Lambda=O[\![\Gamma]\!]$$
for $O_{\pi,p}$ the completion of integer ring of the Hecke field at the prime above $p$ determined via the embedding $\iota_p$.

Put 
\[\Theta_n(\pi,1)=1_{\Delta}(\Theta_n(\pi))\in O[\alpha_{p}^{-1}][\Gamma_n]\,\]
and 
\[\Theta_\infty(\pi)=\{\Theta_n(\pi)\}_n\in O[\alpha_{p}^{-1}][\![{{\rm G}_\infty}]\!];\,\quad \Theta_\infty(\pi,1)=\{\Theta_n(\pi,1)\}_n=1_{\Delta}(\Theta_\infty(\pi))\in O[\alpha_{p}^{-1}][\![\Gamma]\!].\]
The latter are well-defined by Lemma \ref{L:7.W}. In some applications we extend $O$ to contain $O_{K_\fp}$ for $\fp$ the prime of $K$ above $p$ determined via the embedding 
$\iota_p$.

If the Hecke eigenvalue $\alpha_p$ as in \eqref{Hk_p} satisfies
\begin{equation}\label{ord}\tag{ord}
v_{p}(\alpha_p)=0, 
\end{equation}
where $\alpha_p$ is viewed as an element in $\BC_p$ via $\iota_p$, then 
$$
\Theta_{\infty}(\pi,1)\in \Lambda.
$$
If the condition \eqref{ord} holds, define the $p$-adic $L$-function
$$
\mathscr{L}_{p}(\pi)=\Theta_{\infty}(\pi,1)^2\in \Lambda.
$$

To describe an interpolation property of the theta elements,  
put  
\begin{equation}\label{eq:con}
C(\pi,K)= \frac{\epsilon(\pi)}{\epsilon(\pi_p)} 2^{\#\Sigma_D}\sqrt{|D_K|}. 
\end{equation}
\begin{thm}\label{T:Thetaevaluation.W} Let $\chi\in \Xi_p$ be of conductor $p^s$. 
We have 
\begin{align*}\chi(\Theta_{\infty}(\pi,1)^{2})= 
\frac{L^{(p^{s} N_{\rm{r}})}(\frac{1}{2},\pi_K\otimes \chi)}{\Omega_{\pi}}  \cdot e_p(\pi,\chi)^{2} p^s\alpha_p^{-2s}
\cdot \chi_{S^+}(\fN^+) C(\pi,K). 
\end{align*}
In particular, under the condition \eqref{ord}, the same interpolation formula holds for the $p$-adic $L$-function $\mathscr{L}_{p}(\pi)\in\Lambda$.
\end{thm}
The result just follows from Theorem \ref{T:central.W}.

\subsection{$p$-adic $L$-functions: supersingular case} 
The section describes a construction of integral plus/minus $p$-adic $L$-functions in the supersingular case.  
It builds on an idea of Pollack \cite{Po}.
\subsubsection{Setting}
Recall that $p$ is an odd prime split or inert in $K$. 

Suppose that $a_p$ as in \eqref{Hk_p} satisfies
\begin{equation}\label{ss}\tag{ss}
a_p=0.
\end{equation}
One then has $\alpha_{p}=-\beta_{p}$.

For $\tiny{\bullet}\in \{\alpha_p,\beta_p\}$, 
recall that $\varphi^
\dagger_{\tiny{\bullet}}$ denotes the $p$-stabilization of $p$-optimally normalised test vector $\varphi$ with respect to $\bullet$  as in \S\ref{ss:theta-def}. 
 Let $\Theta_{\tiny{\bullet}}(\pi)$ be the theta element 
$$\Theta_{\tiny{\bullet}}(\pi)=\{\Theta_{n}(\varphi^\dagger_{\tiny{\bullet}},1)\}_{n}
$$ associated to the pair $(\pi,1)$.
\begin{lem}\label{lm:am}
For $\tiny{\bullet}\in\{\alpha_{p},\beta_{p}\}$, the theta element
$\Theta_{\tiny{\bullet}}(\pi)$ is a $(1/2)$-admissible distribution on $\Gamma$.
 \end{lem}
 \begin{proof} 
 Recall that $\varphi$ is $p$-integral as in \S\ref{ss:ntv}
 and 
  $\tiny{\bullet}$ is a square-root of $-p$.  
 Hence the assertion just follows from the definition of $\Theta_{\tiny{\bullet}}(\pi)$. 
 \end{proof}

Fix an isomorphism 
$$\Lambda=O[\![\Gamma]\!] \simeq O[\![T]\!] ,\ \gamma \mapsto 1+T$$ 
for $\gamma$ a topological generator of $\Gamma$. For a $p$-th power root of unity $\zeta \in \ov{\BQ}_p^\times$, let 
$$
\psi_{\zeta}:\Gamma\ra \ov{\BQ}_{p}^\times ,\ \gamma \mapsto \zeta
$$
be a character, and $\psi_{\zeta}: \Lambda \ra O[\zeta]$ also denote the associated homomorphism. 
Let $\Xi^{+}_{p}\subset \Xi_{p}$ and $\Xi^{-}_{p}\subset \Xi_{p}$ be subsets of characters corresponding to $\zeta$ of order $p^{t}$ with t even and odd respectively.

\begin{defn}
Let
$$
\log_{p}^{+}(1+T)=\frac{1}{p}\prod_{n=1}^{\infty}\frac{\Phi_{p^{2n}}(1+T)}{p} ,\
\log_{p}^{-}(1+T)=\frac{1}{p}\prod_{n=1}^{\infty}\frac{\Phi_{p^{2n-1}}(1+T)}{p}
$$
be the half $p$-adic logarithms of Pollack \cite{Po}, where $\Phi_{p^{m}}(X)$ denotes the $p^m$-th cyclotomic polynomial. 
\end{defn}
\subsubsection{Plus/minus $p$-adic $L$-functions}

\begin{prop}\label{prop:pLss}
Let $\pi$ be as in \S\ref{ss:set}. Suppose that the condition \eqref{ss} holds. Then 
there exist $$\Theta^{\pm}(\pi) \in \Lambda$$ such that 
$$
\Theta_{\pm \alpha_{p}}(\pi)=\log_{p}^{+}(1+T) \cdot \Theta^{-}(\pi) \pm \alpha_{p} \log_{p}^{-}(1+T)\cdot \Theta^{+}(\pi). 
$$
\end{prop}
\begin{proof}
In the following we proceed as in the proof of \cite[Thm.~5.6]{Po} and \cite[\S8]{Ko0}
(see also~\cite[\S2]{DI} and \cite[\S3]{BBL}).

Consider theta elements $\{\wt{\Theta}_{n}(\pi)\}_{n\geq 0}$
given by 
\begin{align*}\wt{\Thetam}_n(\pi)= \sum_{a\in {\rm G}_n}\varphi(x_n(a))\cdot [a]_{n} \in O[{\rm G}_n].
\end{align*}
If $n\geq 2$, we have
$$
\pi^{n}_{n-1}(\wt{\Theta}_{n}(\pi)) = -\xi_{n-1} \wt{\Theta}_{n-2}(\pi)
$$
for $\pi^{n}_{n-1}:O[{\rm G}_{n}] \rightarrow O[{\rm G}_{n-1}]$ the projection map and 
$\xi_{n-1}:=\sum_{\sigma \in {\rm G}_{n-1}/{\rm G}_{n-2}} \sigma$. 

Consequently, $\wt{\Theta}_{n}(\pi)$ is divisible by the half cyclotomic polynomial $\omega_n^{\epsilon}$ as defined in \cite[p.~10]{Ko0} for $\epsilon$ the parity of 
$(-1)^{n-1}$. (These are denoted by ${\rm G}_n^{-\epsilon}$ in \cite[p.~544]{Po}.)
For a fixed parity $\epsilon$ of $n$, factoring out these extra zeros yields a $p$-integral norm compatible sequence 
$$\{\wt{\Theta}_{n}^{\epsilon}(\pi)\in O[\Delta][\![T]\!]/(\omega_{n}^{\epsilon})\}.$$ 
Let $\Theta_{n}(\pi)\in O[\![T]\!]/(\omega_{n}^{\epsilon})$ denote the image of 
$\{\wt{\Theta}_{n}^{\epsilon}(\pi)\}$ under the projection ${\rm G}_{n}\twoheadrightarrow \Gamma_{n}$.

Define $$\Theta_{}^{\epsilon}(\pi)=\lim \Theta^{\epsilon}_{n}(\pi) \in O[\![T]\!] \simeq \Lambda.$$ 
In view of the construction the proof concludes. 
\end{proof}
\begin{remark}
While the sign labelling of $\Theta_{\pm}(\pi)$ is opposite to \cite{Po}, it is compatible with \cite{Ko0}.
\end{remark}
Define 
$$
\CL_{p}^{\pm}(\pi)=\Theta^{\pm}(\pi)^{2}.
$$
An interpolation property: 

\begin{thm} \label{p-adic:RS}
Let $\pi$ be as in \S\ref{ss:set}. Suppose that the condition \eqref{ss} holds.
\begin{itemize}
\item[(a)] For {$\chi=\psi_{\zeta}\in\Xi_p^+$} of order $p^{t}>1$ and conductor $p^s$, we have 
\begin{align*}\chi(\CL_{p}^{+}(\pi))=& 
\frac{L^{( p^{s} N_{\rm{r}})}(\frac{1}{2},\pi_K\otimes \chi)}{-\Omega_{\pi}} \cdot p^{t+1} 
\prod_{\text{odd $m=1$}}^{t-1}\Phi_{p^{m}}(\zeta)^{-2} \cdot \chi_{S^+}(\fN^+)C(\pi,K)
\end{align*}
for $C(\pi,K)$ as in \eqref{eq:con}. 

{Moreover, if $p$ is inert in $K$, then
\[
1(\Theta^{+}(\pi))=0.
\]}
\item[(b)] For $\chi=\psi_{\zeta}\in\Xi_p^-$ of order $p^{t}$ and conductor $p^s$, we have 
\begin{align*}\chi(\CL_{p}^{-}(\pi))=& 
\frac{L^{( p^{s} N_{\rm{r}})}(\frac{1}{2},\pi_K\otimes \chi)}{\Omega_{\pi}} \cdot p^{t+1} 
\prod_{\text{even $m=2$}}^{t-1}\Phi_{p^{m}}(\zeta)^{-2}
\cdot\chi_{S^+}(\fN^+)C(\pi,K).
\end{align*}

\end{itemize}

\end{thm}
\begin{proof}
\begin{itemize}
\item[(a)]  For {$\chi=\psi_{\zeta}\in\Xi_p^+$} of order $p^{t}>1$, 
note that 
$$
\begin{aligned}
\psi_{\zeta}(\Theta_{\alpha_{p}}(\pi))&=\alpha_{p} \cdot \phi_{\zeta}(\log_{p}^{-}(1+T))\phi_{\zeta}(\Theta^{+}(\pi))\\
&=\alpha_{p} \cdot 
\frac{1}{p}\prod_{\text{odd $m=1$}}^{t-1}\frac{\Phi_{p^{m}}(\zeta)}{p} \cdot
\phi_{\zeta}(\Theta^{+}(\pi))
\end{aligned}
$$
by \cite[Lem.~4.7]{Po}.
Hence, 
$$
\psi_{\zeta}(\Theta_{\alpha_{p}}(\pi)^{2})= \frac{-1}{p^{t+1}}\cdot 
\prod_{\text{odd $m=1$}}^{t-1}\Phi_{p^{m}}(\zeta)^{2} \cdot 
\phi_{\zeta}(\Theta^{+}(\pi)^{2}). 
$$
Now the assertion just follows from Theorem~\ref{T:Thetaevaluation.W}. 

{As for $\psi_1$, if $p$ is inert in $K$, we have
 $$\begin{aligned}
1(\Theta^{+}(\pi))
=&-\sum_{a\in {\rm G}_1}\varphi (x_1(a))\\
=&-\sum_{a\in {\rm G}_0}T_p\varphi (x_0(a))\\
=&0. 
\end{aligned}
$$
Here the first equality follows from definitions of $\Theta(\pi)$, $\varphi_{\pm\alpha_p}^\dagger$ (cf.~\S\ref{SS:pstabilization}) and $\Theta^{+}(\pi)$ (cf.~Proposition \ref{prop:pLss}),
the second from: if $p$ is inert in $K$, then we have an identity  
$$\sum_{u\in O_{K_p}^\times/O_{K_p,p}^\times}au\varsigma^{(1)}\equiv a\varsigma^{(0)}T_p\pmod{\GL_2(\BZ_p)}$$ 
of Hecke operators 
for $a\in \wh{K}^\times$ (see~\S\ref{gpt} for the definition of $x_n(a)$), and the third just follows from \eqref{ss}.}

\item[(b)]  One may proceed as in part (a). 

 \end{itemize}
\end{proof}
\begin{remark}\noindent
\begin{itemize}
\item[(i)] The evaluation $1(\CL^{-}_{p}(\pi))$ basically equals $L(\frac{1}{2},\pi_K)$.
Indeed, we have
{$$1(\Theta^{-}(\pi))=[G_1:G_0]\sum_{a\in {\rm G}_0}\varphi (x_0(a))$$ whose square}  equals algebraic part of the central $L$-value $L(\frac{1}{2},\pi_K)$ up to explicit factors by {\eqref{eq1}}. 
\item[(ii)] The vanishing of $1(\Theta^{+}(\pi))$ in Theorem~\ref{p-adic:RS}(a) is intertwined with direct sum decomposition of local Iwasawa cohomology groups in Rubin's conjecture (cf.~\cite{BKO,BKOb}). This phenomenon does not occur in the cyclotomic setting \cite{Ko0}.
\end{itemize}
\end{remark}

\begin{cor}\label{cor:val}
Let $\pi$ be as in \S\ref{ss:set}. Suppose that the condition \eqref{ss} holds.
\begin{itemize}
\item[(a)] For $\psi_{\zeta}\in\Xi_p^+$ of order $p^{t}\gg 1$, we have 
$$
v_{p}\left(\frac{L^{( N_{\rm{r}})}(\frac{1}{2},\pi_K\otimes \chi)}{\Omega_{\pi}}\right)=
\mu^{+}+ \frac{2(p^{t-1}-p^{t-2}+ \cdots + p-1)+\lambda^{+}}{p^{t-1}(p-1)}-(t+1).
$$
for $\mu^{+}$ and $\lambda^{+}$ the Iwasawa invariants of $\CL_{p}^{+}(\pi)$. 
\item[(b)] For $\psi_{\zeta}\in\Xi_p^-$ of order $p^{t}\gg 1$, we have 
$$
v_{p}\left(\frac{L^{( N_{\rm{r}})}(\frac{1}{2},\pi_K\otimes \chi)}{\Omega_{\pi}}\right)=
\mu^{-}+ \frac{2(p^{t-1}-p^{t-2}+ \cdots + p^{2}-p)+\lambda^{-}}{p^{t-1}(p-1)}-(t+1).
$$
for $\mu^{-}$ and $\lambda^{-}$ the Iwasawa invariants of $\CL_{p}^{-}(\pi)$. 

\end{itemize}
\end{cor}

\subsubsection*{Primitive $p$-adic $L$-functions}
In view of Theorem~\ref{p-adic:RS} we are led to the following. 
\begin{defn}\label{def:prim}
For $\circ\in\{+,-\}$, define
$$
\mathscr{L}_{p}^{\circ}(\pi)=
\begin{cases}(\frac{\Theta^{+}(\pi)}{T})^{2}&\text{if $p$ is inert and $\circ=+$};\\
\CL_{p}^{\circ}(\pi)&\text{else}.\\
\end{cases}
$$
\end{defn}
We expect $\mathscr{L}_{p}^{\circ}(\pi)$ to appear in signed Iwasawa main conjectures (cf.~\cite{DI,BKO,BBL,BBL2}). 
\begin{remark}
For $p$ inert in $K$, an interesting problem: to formulate a conjecture predicting the value of $\mathscr{L}_{p}^{+}(\pi)$ at the trivial character (cf.~\cite{MTT}). In the CM case, it encodes $p$-adic logarithm of rational points on the associated CM abelian variety (cf.~\cite{BKOb}).  
\end{remark}
\subsection{CM case} This section considers $p$-adic $L$-functions associated to a Hecke character  over an imaginary quadratic field, and links among them. 

\subsubsection{Rubin's $p$-adic $L$-function}
The following is a resume of results of \cite{Ru,BKO,BKOd}. 

Let $K$ be an imaginary quadratic field with $p$ inert and 
$H$ the Hilbert class field of $K$. 
Assume that 
\begin{equation}\label{clp}
p \nmid 6h_{K}.
\end{equation}
Let $\Phi$ denote the localisation of $K$ at the prime ideal above $p$. 

Let $K_\infty$ be the anticyclotomic $\BZ_p$-extension of $K$ 
and $K_n$ the $n$-th layer. In view of \eqref{clp} we often regard the set $\Xi_p$ of anticyclotomic $p$-power order  characters of $\Phi$ as that of $K$.

Let $\lambda$ be a self-dual Hecke character of $K$ of infinity type $(1,0)$. Let $E$ be a $\BQ$-curve in the sense of Gross {\cite{Gro80}}
such that the Hecke character $\lambda\circ {\rm{N}}_{H/K}$  is associated to $E$, and $E$ has good reduction at each prime of $H$ above $p$. 
Let $\fp$ be the prime of $H$ above $p$ compatible with the embedding $\iota_p$.
Fix a Weierstrass model  of $E$ over {$H \cap 
O_{H_\fp}$} which is smooth at $\fp$.
By considering a Galois conjugate of $E$ over $H$, 
we may assume 
the existence of a complex period $\Omega_K \in\BC^\times$
such that 
$$L= O_K\Omega_{K},$$ where  
$L$ is the period lattice associated to the model. 

Rubin's $p$-adic $L$-function also involves the following local setting. 

Let  $O_\Phi$ be the integer ring of $\Phi$.
Let $\mathscr{F}$ be a Lubin-Tate formal group  over $O_\Phi$ for   
the uniformizing parameter $\pi:=-p$. 
For $n\ge 0$,
write $\Phi_n=\Phi(\SF[\pi^{n+1}])$, the extension of $\Phi$ in $\BC_{p}$ generated by the $\pi^{n+1}$-torsion points of $\mathscr{F}$. 
Put $\Phi_{\infty}=\cup_{n\ge 0}\Phi_n$
and $T=T_{\pi}\SF$. 

Let $\Theta_\infty \subset \Phi_\infty$ be the $\BZ_p^2$-extension of $\Phi$, 
$\Psi_\infty$ the anticyclotomic $\BZ_p$-extension 
and $\Psi_n$ the $n$-th layer.
Put 
$\Gamma=\Gal(\Psi_{\infty}/\Phi)\cong\BZ_{p}$, $\Lambda_{O_\Phi}=O_\Phi[\![\Gamma]\!]$ 
and fix  a topological generator $\gamma$ of $\Gamma$. 
Let $U_n$ be the group of principal units in $\Phi_n$, that is, 
the group of elements in $O_{\Phi_n}^{\times}$ congruent to one modulo the maximal ideal. 

Put
 \[
T^{\otimes -1}=\Hom_{O_{\Phi}}(T, O_{\Phi}),\quad   V_{\infty}^{*} = \left(\varprojlim_nU_{n}\otimes_{\Z_p} T^{\otimes -1}\right)^{\Delta}\otimes_{O_{\Phi}[\![\Gal(\Phi_{\infty}/\Phi)]\!]}\Lambda_{O_{\Phi}},
 \]
 where $\Delta:=\Gal(\Phi_\infty/\Theta_\infty)$ 
 and the superscript $\Delta$ refers to $\Delta$-invariants.
For a finite character $\chi$ of $\Gal(\Psi_{\infty}/\Phi)$, let 
$\delta_\chi$ be the associated Coates--Wiles logarithmic derivative on $V_\infty^*$.

Let $\Xi_p$ be the set of finite characters of $\Gamma$ and put
\[
\begin{split}
\Xi_p^+&={\{}\chi \in \Xi_p \text{ $|$} \text{ order $\chi$} \text{ is an even power of } p{\}},\\
\Xi_p^-&={\{}\chi \in \Xi_p \text{ $|$} \text{ order $\chi$} \text{ is an odd power of } p{\}}.
\end{split}
\]
Define
\begin{equation}\label{def:loc}
V^{*,\pm}_{\infty} := {\{} v \in V_{\infty}^* \text{ $|$} \text{ }
\delta_{\chi}(v)=0 \quad \text{for every }\chi \in \Xi_p^{\mp} {\}}.
\end{equation}
Rubin showed that $V^{*,\pm}_{\infty}$ is a free $\Lambda_{O_\Phi}$-module of rank one (cf.\ \cite[Prop.\ 8.1]{Ru}).

An insight of Rubin is the following existence of 
a $p$-adic $L$-function (cf.~\cite[\S10]{Ru},~\cite[\S6]{BKO}). 
\begin{thm}\label{thm:Rubin-pL}
{Let $\varepsilon \in \{+, -\}$ be the sign of the functional equation of the Hecke $L$-function 
$L(\lambda, s)$.}
Let $v_\varepsilon$ be a generator of the $\Lambda$-module $V_\infty^{*, \varepsilon}$. 
Then there exists $$\mathscr{L}_p(\lambda, \Omega, v_\varepsilon)=:\mathscr{L}_{p}(\lambda)  \in \Lambda$$ such that 
\[
\chi(\mathscr{L}_{p}(\lambda))=\frac{1}{\delta_\chi(v_\varepsilon)}\cdot \frac{L(1, \overline{\lambda\chi})}{\Omega_K}
\]
for $\chi\in\Xi_p^{-\varepsilon}\setminus\{1\}$. 
\end{thm}
\noindent Here the non-vanishing of $\delta_{\chi}(v_{\varepsilon})$ is a consequence of Rubin's conjecture (cf.~\cite[Lem.~10.1]{Ru}). 

The main result of \cite{BKOd} is the following. 

\begin{thm}\label{thm:delta}
Let $\chi$ be a finite character of $\mathrm{Gal}(\Psi_n/\Phi)$ of order $p^t >1$, and put $\varepsilon=(-1)^{t-1}$.  
Let $v_{\varepsilon}$ 
be a generator of $V_{\infty}^{*, \varepsilon}$.
Then we have 
\[
v_{p}(\delta_{\chi}(v_{\varepsilon}))=-\frac{t+1}{2}+
\frac{1}{p^{t-1}(p-1)}\left(\frac{1-\varepsilon}{2}+\sum_{(-1)^{k}=\varepsilon} (p^k-p^{k-1})\right)
\] 
where $1\leq k \leq t-1$ such that $(-1)^k=\varepsilon$. 
\end{thm}

\subsubsection{A link with Rankin--Selberg $p$-adic $L$-function}
Let the setting be as before.
In particular, $\pi_{\lambda}$ denotes the irreducible cuspidal automorphic representation associated to $\lambda$.

We have a factorisation
$$L(1/2,\pi_{\lambda,K}\otimes \chi)=L(1,\lambda\chi)L(1, \lambda\chi^{-1})$$
of $L$-values. 
In light of $p$-adic Artin formalism, 
one may expect a factorisation
\begin{equation}\label{fac}
\mathscr{L}_{p}(\pi_{\lambda}) =
\mathscr{L}_{p}(\lambda) \mathscr{L}_{p}(\lambda)^{\iota}
\end{equation}
up to an element in $\Lambda^\times$.
Here 
\begin{equation}\label{pmL_par}
\mathscr{L}_{p}(\pi_{\lambda}):=\mathscr{L}_{p}^{-\varepsilon}(\pi_\lambda)
\end{equation}
for $\varepsilon$ the sign {of} $\epsilon(\lambda)$, and 
$\iota$ denotes the involution of $\Lambda$ arising from $\gamma \mapsto \gamma^{-1}$.
A difficulty in realising the factorisation is that interpolation formula for the Rubin $p$-adic $L$-function involves the local invariant $\delta_{\chi}(v_{\varepsilon})$ and the CM period 
$\Omega_K$, whereas that for $\mathscr{L}_{p}(\pi_{\lambda})$ involves a half cyclotomic polynomial and the automorphic period 
$\Omega_{\lambda}:=\Omega_{\pi_\lambda}$. 

We prove a comparison of Iwasawa invariants predicted by \eqref{fac}. We begin with the following preliminary. 
\begin{lem}\label{cor:del_asy}
 For $\chi$ of order $p^t \gg 1$ so that $(-1)^{t-1}=\varepsilon$, 
we have
$$
\begin{aligned}
v_{p}(\delta_{\chi}(v_{\varepsilon}))
&=-\frac{t+1}{2}+
\frac{1}{p^{t-1}(p-1)} 
\left(\frac{1-\varepsilon}{2}+\sum_{(-1)^{k}=\varepsilon} (p^k-p^{k-1}) \right) + \frac{\lambda(\mathscr{L}_{p}(\pi_{\lambda}))-2\lambda(\mathscr{L}_{p}(\lambda))
}{2p^{t-1}(p-1)}\\
&+ \frac{1}{2}
\left(\mu(\mathscr{L}_{p}(\pi_{\lambda}))+v_{p}(\frac{\Omega_{\lambda}}{\Omega_{K}^{2}})-2\mu(\mathscr{L}_{p}(\lambda))\right). 
\end{aligned}
$$
Here $1\leq k \leq t-1$ such that $(-1)^k=\varepsilon$, 
 $\mu(\cdot)$ and $\lambda(\cdot)$ are associated Iwasawa invariants, and  $\mathscr{L}_{p}(\pi_\lambda):=\mathscr{L}_{p}^{-\varepsilon}(\pi_{\lambda})$.
 \end{lem}
\begin{proof}
The following is based on comparison $p$-adic valuation of $L$-values
interpolated\footnote{Note that the interpolated $L$-values
 are generically non-zero (cf.~\S\ref{s:nv}).} by the $p$-adic $L$-functions $\mathscr{L}_{p}(\pi_{\lambda})$ and $\mathscr{L}_{p}(\lambda)$. 
 We consider the case $\epsilon(\lambda)=-1$, and leave the other case to the interested reader. 

In view of Corollary~\ref{cor:val} and Definition~\ref{def:prim}, 
for $\psi_{\zeta}\in\Xi_p^+$ of order $\chi=p^{t}\gg 1$, we have 
$$
v_{p}\left(\frac{L(\frac{1}{2},\pi_{\lambda,K}\otimes \chi)}{\Omega_{\lambda}}\right)=
\mu(\mathscr{L}_{p}(\pi_{\lambda}))+ \frac{2(p^{t-1}-p^{t-2}+ \cdots + p-1)+\lambda(\mathscr{L}_{p}(\pi_{\lambda}))+2}{p^{t-1}(p-1)}-(t+1).
$$

On the other hand, by Theorem~\ref{thm:Rubin-pL}, 
$$
v_{p}\left(\frac{L(1,\lambda\chi)L(1,\lambda\chi^{-1})}{\Omega_{K}^{2}}\right)=
2\mu(\mathscr{L}_{p}(\lambda))+\frac{2\lambda(\mathscr{L}_{p}(\lambda))}{p^{t-1}(p-1)}+2v_{p}(\delta_{\chi}(v_{-})).
$$
By comparing the above two, the proof concludes. 
\end{proof}
\begin{remark}
The left hand side of Corollary~\ref{cor:del_asy} is a local invariant, and Rubin asked for determination of its $p$-adic valuation in \cite{Ru}.
\end{remark}
Towards the factorisation \eqref{fac} our main result is the following. 
\begin{prop}\label{prop:rel}
Let $\lambda$ be a self-dual Hecke character over an imaginary quadratic field $K$ of infinity type $(1,0)$ and $\pi_\lambda$ the associated cuspidal automorphic representation. 
For a prime $p\nmid 6h_{K}\cond{\lambda}$ inert in $K$, let $\mathscr{L}_{p}(\lambda)$ and $\mathscr{L}_{p}(\pi_{\lambda})$ be the associated Rubin and Rankin--Selberg $p$-adic $L$-functions. 
Then we have
$$
\mu(\mathscr{L}_{p}(\pi_{\lambda}))+v_{p}\left(\frac{\Omega_{\lambda}}{\Omega_{K}^{2}}\right)=2\mu(\mathscr{L}_{p}(\lambda)),\,
\lambda(\mathscr{L}_{p}(\pi_{\lambda}))=2\lambda(\mathscr{L}_{p}(\lambda)).
$$
\end{prop}
\begin{proof}
In view of Corollary \ref{cor:del_asy} and Theorem~\ref{thm:delta}
 it  follows that 
$$
\frac{\lambda(\mathscr{L}_{p}(\pi_{\lambda}))-2\lambda(\mathscr{L}_{p}(\lambda))
}{p^{t-1}(p-1)} = 2\mu(\mathscr{L}_{p}(\lambda))-\mu(\mathscr{L}_{p}(\pi_{\lambda}))-v_{p}(\frac{\Omega_{\lambda}}{\Omega_{K}^2}).
$$
Since the numerator is a constant, it vanishes. 
\end{proof}

\section{Non-vanishing of Rankin--Selberg $L$-values modulo $\ell$: CM case}\label{s:nv}
The section presents mod $\ell$ non-vanishing of Rankin--Selberg $L$-values in the CM case.
{The main results are $(\ell,p)$ non-vanishing Theorems \ref{T:2.W}, \ref{T:b.W}, \ref{T:v.W} and \ref{T:3.W} for $\ell\neq p$, and Theorems \ref{T:1.W} and \ref{T:4.W} which concern $\mu$-invariants.}

\subsection{Key tools}

\subsubsection{Equidistribution of special points}\label{SS:Uniform_CM}We describe a special case of the main result of \cite{Vatsal_Cornut:Documenta}, which is based on Ratner's ergodicity of unipotent flows.

Let the setting be as in  \S\ref{SS:setup}, where we fix a definite quaternion algebra $B$ over $\BQ$, 
an odd prime $p$ with $B_p$ split and an embedding $\iota: K\ra B$ of an imaginary quadratic field.

We have the ring class group ${\rm G}_\infty:=\varprojlim_n {\rm G}_n$ of conductor $p^\infty$.
Let $\Delta^\alg$ be the subgroup of ${\rm G}_\infty$ generated by the image of 
$$K^\x_{\rm ram}:=\prod_{\pme|D_{K}}K^\x_\pme.$$ Note that $\Delta^\alg$ is a $(2,\cdots,2)$-subgroup of $\Delta$. Let $D_0$ be a set of representatives of $\Delta^\alg$ in $K^\x_{\ram}$, and $D_1$ that  of 
$\Delta/\Delta^\alg$ in $\wh{K}^\x$. Then $D:=D_1D_0$ is a set of representatives of $\Delta$ in $\wh{K}^\x$.

Write $\ov{K}^\x$ for the closure of $K^\x$ in $\wh{K}^\times$ and $\ov{B}^\x$ that of $B^\x$ in $\wh{B}^\times$.
Put 
$$\CMspace:=\ol{K}^\x\bksl \wh{B}^\times , \, \quad X:=\ol{B}^\x\bksl \wh{B}^\times , \, \quad Z:=\ol{\BQ}_+^\times\bksl \wh{\BQ}^{\times}.
$$
The group $\wh{B}^\times$ acts on these spaces via right multiplication on the first two spaces and via multiplication by the norm on the third space. Similarly, there is a left action of the group $\wh{K}^\times$ on these spaces. 
Let $\Red:\CMspace\to X$ be the natural quotient map and $c:X\to Z$ the one induced by the reduced norm ${\rm{N}}:B^\x\to \Q^\x$. For $g\in \wh{B}^\times$, let $[g]$ denote the image of $g$ in $\CMspace$. Let $U$ be an open compact subgroup of $\wh{B}^\times$. Put
\[X(D_1,U)=\prod_{\tau\in D_1} X/U\text{ and }Z(D_1,U)=\prod_{\tau\in D_1}Z/{\rm{N}}(U).\]
Define
\begin{align*}\Red_{D_1}:\ &\CMspace\longrightarrow X(D_1,U),\quad x\mapsto (\Red(\tau\cdot x)U)_{\tau\in D_1}\\
\intertext{ and}
c_{D_1}:\ &X(D_1,U)\longrightarrow Z(D_1,U),\quad(x_{\tau})_{\tau\in D_1}\mapsto ({\rm{N}}(x_\tau))_{\tau\in D_1}.\end{align*}

The following key result is a special case of \cite[Cor.\,2.10]{Vatsal_Cornut:Documenta}.
\begin{thm}\label{P:Vatsal_Cornut.W} Let $\CH$ be a $B^\x_p$-orbit in $\CMspace$ and $\ol{\CH}$ the image of $\CH$ in $\CMspace/U$. Then for all but finitely many $x\in\ol{\CH}$, we have
\[\Red_{D_1}(\wh{O}_K^\x\cdot x)=c_{D_1}^{-1}(\wh{O}_K^\x\cdot \ol{x}),\]
 where $\overline{x}=c_{D_1}\circ \Red_{D_1} (x)$.
\end{thm}

\subsubsection{Non-Eisenstein functions: generalities}
Let $A$ be a commutative $\Z$-algebra. 

Let $U$ be an open compact subgroup of $\wh{B}^\times$. Recall that $M_2(U,A)$ is the set of functions $h:B^\x\wh{\BQ}^\times \bksl \wh{B}^\times\to A$ such that $h$ is right invariant by $U$. Let 
$$M_2(A):=\varinjlim_{U\subset \wh{B}^\times}M_2(U,A)$$ be the space of smooth $A$-valued functions on $B^\x\wh{\BQ}^\times\bksl \wh{B}^\times$. Let $\rho:\wh{B}^\times\to\Aut(M_2(A))$ denote the right translation of $\wh{B}^\times$.
\begin{defn}Let $B^1=\{g\in B^\x\mid {\rm{N}}(g)=1\}$ be the algebraic group over $\Q$. Let
\[M_2(A)_{\Eis}:=\{h\in M_2(A)\mid \rho(g_1)h=h\text{ for all $g_1\in B^1(\BA_f)$}\}\]
and 
\[S_2(A):=M_2(A)/M_2(A)_{\Eis}.\]
Denote by $S_2(U,A)$ the image of $M_2(U,A)$ in $S_2(A)$.

A function $h\in M_2(A)$ is called Eisenstein if $h\in M_2(A)_{\Eis}$. Equivalently, $h$ is Eisenstein if and only if 
$h(g)=h_1({\rm{N}}(g))$ for a smooth function $h_1:\Q_+^\x\bksl \wh{\Q}^\times\to A$.
\end{defn}

The following properties of non-Eisenstein functions will be used in our non-vanishing arguments. 
\begin{lem}\label{L:Ihara2.W}Let $\pme$ be a finite place such that $B_q$ is split and $U_q= U_0(q^k)_q$ for some $k$. For $\beta_1,\cdots ,\beta_s\in A$, define $\CR\in\End(M_2(A))$ by
\[\CR=1+\sum_{i=1}^s\beta_i\cdot \rho(
\begin{pmatrix}
\pme^{-i} & \\
 & 1\\
\end{pmatrix}
).\]
Then $\CR:\CS_2(U,A)\to\CS_2(A)$ is injective (cf.~\cite[Lem.~5.5]{CH}).
\end{lem}

In the following lemma, let $U=\wh{R}^\times$ for an order $R$ of $B$ and $q$ a prime such that $B_q$ is a split quaternion algebra.  Let $K_q\subset B_q$ be a quadratic subalgebra.
Let $A$ be the ring of integers of a finite extension of $\BQ_\ell$ and $\varpi$ a uniformizer of $A$.

\begin{lem}\label{tnn} 
Let $\pi$ be a cuspidal automorphic representation and 
  pick a non-zero $ f\in M_2(U,A)[\pi]$.
   For a prime $q$, suppose that $K_q$ is a field and  {$B_q$ splits.}
  Moreover, for $f_q$ the newform, suppose that 
    \[\gamma:=\int_{K_q^\times/\BQ_q^\times V_q}\frac{(\pi(t)f_q,f_q)_q}{(f_q,f_q)_q}d^\times t\] is an $\ell$-adic unit, where $(\ ,\ )_q$ is a non-degenerate $\GL_2(\BQ_q)$-invariant bilinear pairing on $\pi_q$ and $V_q= K_q^\times\cap U_q$.
    
    Assume that {$\ell\nmid q(q^2-1)$.}
    Then if $f\pmod{\varpi}$ is non-zero, so is
    \[F:=\sum_{t \in K_q^\times/\BQ_q^\times V_q}\pi(t)f\pmod{\varpi}.\]  
  \end{lem}
  \begin{proof}
  
  Let $R_q'\subset R_q$ be a suborder such that $R_q^\times$ stabilizes $F$. Put \[F'=\sum_{g\in R_q^\times/R_q'^\times}\pi(g)F\in \pi^{\wh{R}^\times}.\] If $F'$ is non-zero modulo $\varpi$, then 
  so is $F$. 
  In the following, we consider $F'$. 
  
  Note that $F'\in \BC f$ and so {$F'=\kappa f$} for a constant $\kappa$.
  Let $\pair{\ ,\ }$ be a $B_{\BA}^\times$-invariant bilinear pairing on $\pi$. We have 
  \[\begin{aligned}
      \kappa=&\frac{\pair{F' ,f }}{\pair{f ,f }}\\
       =&\# R_q^\times/R_q'^\times\cdot \frac{\pair{F ,f }}{\pair{f ,f }}\\
      =&\# R_q^\times/R_q'^\times\cdot \gamma.
  \end{aligned}\]
  Here the last equality follows from the uniqueness of $B_{\BA}^\times$-invariant bilinear pairing on $\pi$ up to scalars (cf.~\cite[Lem.~2.6]{Jacquet_Langlands:GLtwo}). {Note that $\# \GL_2(\BZ_q)/(1+q^nM_2(\BZ_q))$ is an $\ell$-adic unit since $\ell\nmid q(q^2-1)$ and in turn so is $\# R_q^\times/R_q'^\times$.
   Since  $\gamma$ is an $\ell$-adic unit and $f\nequiv 0\pmod{\varpi}$, the proof concludes.}
  \end{proof}

\subsubsection{Non-Eisenstein functions: CM case} 
This endoscopic case exhibits peculiar features, which the subsection describes.

We begin with the setting. 
Let $K\subset B$ be an imaginary quadratic field. Let $U$ be an open compact subgroup of $\wh{B}^\times$ and $X_{U}:=B^\times\bs \wh{B}^\times/U$ the associated Shimura set. 

If $U$ is of the form $\wh{R}^\times$ for an order $R$ of $B$ with ${\rm{N}}(U)\subset \BQ_{+}^\times {\rm{N}}(\wh{K}^\times)$, then we may write $$X_{U}=X_{U}^+\sqcup X_{U}^-,$$ where $$X_{U}^+:=\{[h]\in X_{U}\ \Big|\ {\rm{N}}(h)\in \BQ^\times_+\bs \BQ_+^\times {\rm{N}}(\wh{K}^\times)/{\rm{N}}(U)\}.$$

Let $A$ be the ring of integers of some finite extension of $\BQ_\ell$ and $\varpi$ a uniformizer of $A$. 
For $f\in M_2(U,A)$, let $f^\epsilon$ denote its restriction to $X^{\epsilon}_{U}$ 
for $\epsilon\in \{\pm\}$.

Henceforth, we suppose that $U$ is of the form $\wh{R}^\times$ such that
\begin{equation}\label{lev}
{\rm{N}}(U)\subset \BQ_{+}^\times {\rm{N}}(\wh{K}^\times).
\end{equation} Then restricting to $X^\epsilon_{U}$, we define spaces of non-Eisenstein forms $S_2(U,A)^\epsilon$ and $S_2(A)^\epsilon$.
\begin{defn}
{We say $f\in M_2(U,A)$ has CM by $K$ if
 $$T_{q}f=a_{q}f$$ for all but finitely many primes $q$ and $a_q$ the Hecke eigenvalue of the theta series associated to a self-dual Hecke character $\lambda$ over $K$ of infinity type $(1,0)$.} 
     \end{defn}
\begin{lem}\label{ee} Suppose that 
  \begin{itemize}
       \item[\tiny{$\bullet$}] $\ell\nmid 2D_K$, 
     \item[\tiny{$\bullet$}] $f\in M_2(U,A)$ has CM by $K$,   
    \item[\tiny{$\bullet$}] $f^\epsilon \pmod{\varpi}$ is non-zero.
  \end{itemize}  Then $f^\epsilon \pmod{\varpi}$ is non-Eisenstein. 
\end{lem}
\begin{proof}
Pick a prime $q$ so that 
\begin{itemize}
\item[\tiny{$\circ$}] $\ell \nmid q+1$,
\item[\tiny{$\circ$}] $q$ is inert in $K$, 
\item[\tiny{$\circ$}] $f$ is $T_q$-eigen. 
\end{itemize}
In view of the first two hypotheses, such a $q$ exists.
Note that the $T_q$-eigenvalue  is $0$ 
since $f$ has CM by $K$.

Now assume that 
$f^\epsilon \pmod{\varpi}$ is Eisenstein.

By the hypothesis, $f^{\epsilon}(z)\nequiv 0\pmod{\varpi}$ for some $z\in X_U^{\epsilon}$. Then we consider $T_q f\left(z\begin{pmatrix}1&\\ & q^{-1}\end{pmatrix}\right)$. Note that 
\[0\equiv T_q f(z\begin{pmatrix}1&\\&q^{-1}\end{pmatrix})\equiv (q+1)f^{\epsilon}(z)\nequiv 0\pmod{\varpi}, \]
where the congruence $T_qf(z\begin{pmatrix}1&\\&q^{-1}\end{pmatrix})\equiv (q+1)f^{\epsilon}(z)\pmod{\varpi}$ just follows from $T_q=\sum_{i=1}^{q+1}u_q$ for $u_q\in \GL_2(\BQ)$ with ${\rm{N}}(u_q)=q$, and $f^\epsilon\pmod{\varpi}$ being Eisenstein. 
A contradiction.
\end{proof}

\begin{lem}\label{ee3} 
Suppose that ${\rm{N}}(U)\subset \BQ_{+}^\times {\rm{N}}(\wh{K}^\times)$. 
Let $q$ be a prime unramified in $K$ such that $U_q= U_0(q^k)_q$ for some $k$. For a commutative 
$\BZ$-algebra $A$, let $\beta_1,\cdots ,\beta_s\in A$ and $\CR\in\End(M_2(A))^\epsilon$ be the endomorphism defined\footnote{Since $q\nmid D_K$, it is well defined.} by
  \[\CR=1+\sum_{i=1}^s\beta_i\cdot \rho(
  \begin{pmatrix}
  \pme^{-2i} & \\
   & 1\\
  \end{pmatrix}
  ).\]
  Then $\CR:\CS_2(U,A)^\epsilon \to\CS_2(A)^\epsilon $ is injective. Moreover, if $q$ splits in $K$, then the same holds when $2i$ in the definition of $\CR$ is replaced with $i$.
  \end{lem}
  \begin{proof}
    Let $f^\epsilon \in \CS_2(U,A)^\epsilon$ be so that $\CR f^\epsilon\in M_2(A)^\epsilon_\text{Eis}$. In the following we show $f^\epsilon\in M_2(U,A)_{\text{Eis}}^\epsilon$. (See also the proof of~\cite[Lem.~5.5]{CH}.)
      
Note that $N(\BZ_q)\subset \Stab(f^\epsilon)$ for $N(\BZ_q)$ the subgroup of upper triangular matrices. Put $$u=\begin{pmatrix}
  q^{-1}&\\&1
\end{pmatrix}$$ and 
$$P(u)=\sum_{i=1}^s\beta_i\cdot \rho(u^{2i-2}
).$$ Then $P\in \End(M_2(A)_{}^\epsilon)$. 
By the assumption, we have $(1-\rho(u^{2})P(u))f^\epsilon\in M_2(A)^\epsilon_\text{Eis}$, and so 
\[(1-\rho(u^{2n})P(u)^n)f^\epsilon\in M_2(A)^\epsilon_\text{Eis}\]
for any $n\geq 1$. 

Note that $(1-\rho(u^{2n})P(u)^n)f^\epsilon$ and $\rho(u^{2n})P(u)^nf^\epsilon$ are fixed by $\begin{pmatrix}
  1&x\\&1
\end{pmatrix}$ for $x\in \BQ_q$ with $q^{2n}x\in \BZ_q$ since $u^{-2n}\begin{pmatrix}
  1&x\\&1
\end{pmatrix}u^{2n}=\begin{pmatrix}
  1& q^{2n}x\\&1
\end{pmatrix}$. Thus $f^\epsilon$ is fixed by $\begin{pmatrix}
  1&x\\&1
\end{pmatrix}$ for all $x\in \BQ_q$. By smoothness, the same holds for $\begin{pmatrix}
  1&\\y&1
\end{pmatrix}$ for some $y\in \BQ_q^\times$. Therefore, $f^\epsilon$ is fixed by $w_0=\begin{pmatrix}
  &y^{-1}\\ -y&
\end{pmatrix}$, 
and in turn by 
\[\begin{pmatrix}
  1&\\x&1
\end{pmatrix}=w_0\begin{pmatrix}
  1&-y^2x\\&1
\end{pmatrix}w_0^{-1}.\] 

Hence $\SL_2(\BQ_q)$ fixes $f^\epsilon$ and $f^\epsilon(tg)=f^\epsilon(g)$ for all $t\in B^1(\BQ_q)$. By strong approximation, $B^1$ is dense in $B^1(\BA_{f}^{(q)})$, thus $f^\epsilon(tg)=f^\epsilon(g)$ for all $t\in B^1(\BA_{f}^{(q)})$. It follows that $f^\epsilon(tg)=f^\epsilon(g)$ for all $t\in B^1(\BA_f)$ and hence $f^\epsilon(gt)=f^\epsilon(g)$ for all $t\in B^1(\BA)$, concluding the proof. 
  \end{proof}

\subsubsection{An independence of CM values}
The following consequence of equidistribution of special points will be a key to our non-vanishing results.

Let $x_n(a)$ be a family of special points for $a\in \wh{K}^\times$ as in \S\ref{sppt}. 
\begin{prop}\label{C:Vatsal_Cornut} 
Let $A$ be the ring of integers of a finite extension of $\BQ_\ell$ and $\varpi$ a uniformizer.
Let $(\beta_\tau)_{\tau\in D_1}$ be a sequence in $A$ such that $\beta_{\tau_1}\in A^\x$ for some $\tau_1$. Let $f\in M_2(U,A)$ be a CM form of level $U$ as in \eqref{lev}. Assume that
$f$ is non-zero modulo $\varpi$. 

\begin{itemize}
\item[(a)] Suppose that $p$ is inert in $K$. Then there exists an integer $n_0$ such that for every $n>n_0$ of a fixed parity,  we have
\[\sum_{\tau\in D_1}\beta_\tau\cdot f(x_n(a\tau))\nequiv 0 {\pmod{\varpi}}\text{ for some }a\in \wh{K}^\times.\]
\item[(b)] Suppose that $p$ splits in $K$, and that $f$ is non-zero on $X_U^+$.
Then there exists an integer $n_0$ such that for every $n>n_0$,  we have
\[\sum_{\tau\in D_1}\beta_\tau\cdot f(x_n(a\tau))\nequiv 0\pmod{\varpi}\text{ for some }a\in \wh{K}^\times.\]
\end{itemize}
\end{prop}
\begin{proof} Consider a special point $$P_0:=[\cmpt^{(0)}]\in\CMspace$$
for $\cmpt^{(0)}$ as in \S\ref{gpt}.
Note that ${\rm{N}}(\cmptv^{(0)})\in \BQ_{+}^\times\bs \BQ_+^\times {\rm{N}}(\wh{K}^\times)$. Let $\CH=P_0\cdot B^\x_p$ be the $B^\x_p$-orbit of $P_0$. For integers $n\geq 1$, put 
\begin{equation}\label{CM_Iw}
P_n:=P_0
\begin{pmatrix}
p^n & \\
 & 1\\
\end{pmatrix}
\in \CH.
\end{equation}

Note that  the images of $(P_{n})_{n=1,2,\cdots}$ are distinct in $X_U$.
Hence, by Theorem \ref{P:Vatsal_Cornut.W}, there exists $n_0$ such that 
\begin{equation}\label{E:20.W}
\Red_{D_1}(\wh{O}_K^\x P_n)=c_{D_1}^{-1}(\wh{O}_K^\x\ol{P_n})\text{ for every }n>n_0.
\end{equation}

 Since $f \pmod{\varpi}$  is non-Eisenstein by Lemma \ref{ee}, there exist $y,z \in X_U$ such that 
 $$
 f(y) \nequiv f(z) \pmod{\varpi}
 $$
 and $c(y)=c(z)$.
 \begin{itemize}

\item[(a)] Let $n\geq 1$ be an integer. Note that 
$ c({\rm Red}(P_{n})) = c({\rm Red}(P_{n+2})) \pmod{{\rm{N}}(\wh{K}^{\times}){\rm{N}}(U)} $, and
 $$
 c(y)=c(z) \neq c({\rm Red}(P_{n})) \pmod{{\rm{N}}(\wh{K}^{\times}){\rm{N}}(U)} \implies 
  c(y)=c(z) = c({\rm Red}(P_{n+1}))  \pmod{{\rm{N}}(\wh{K}^{\times}){\rm{N}}(U)}
$$
  since $p$ in inert in $K$. It follows that $$\text{$c(y)=c(z)=c(P_{n}) \pmod{{\rm{N}}(\wh{K}^{\times}){\rm{N}}(U)}$ for $n$ of a fixed parity.}$$ In the following we consider $n>n_{0}$ of that parity. 
  
  Replacing $D_1$ by $a'D_1$ for $a' \in \wh{K}^\times$, we may assume that 
  $$c(y)=c(z)=c(P_{n}) \pmod{{\rm{N}}(U)}.$$  
  Pick $(w_{\tau})_{\tau \in D_{1}}\in c_{D_{1}}^{-1}(\ov{P}_{n})$.
   In view of \eqref{E:20.W} there exist $a_{1},a_{2} \in \wh{O}_{K}^\times$ such that 
  $$
  {\rm Red}_{D_{1}}(a_{1}P_{n})=(y,w_{\tau_{2}}, w_{\tau_{3}},\cdots),  \, {\rm Red}_{D_{1}}(a_{2}P_{n})=(z,w_{\tau_{2}},w_{\tau_{3}},\cdots).$$
  Hence 
  $$
  \sum_{\tau \in D_{1}} \beta_{\tau}\cdot f(x_{n}(a_{1}\tau)) -   \sum_{\tau \in D_{1}} \beta_{\tau}\cdot f(x_{n}(a_{2}\tau))  {\equiv } \beta_{\tau_{1}} (f(y)-f(z)) \nequiv 0 \pmod{\varpi},$$
  and the assertion follows.

\item[(b)] Since $p$ splits in $K$, we have
$ c({\rm Red}(P_{n})) = c({\rm Red}(P_{n+1})) \pmod{{\rm{N}}(\wh{K}^{\times}){\rm{N}}(U)} $. 

In view of the assumption we may suppose that \[c(y)=c(z)=c(\Red(P_n))\pmod{{\rm{N}}(U)}\] 
and $f(y)\nequiv f(z) \pmod{\varpi}$. Then the assertion follows just as in the proof of part (a). 

\end{itemize}
\end{proof}

\subsection{Setting}
We introduce set-up for the rest of the section. 

Let $\lambda$ be a self-dual Hecke character over $K$ of infinity type $(1,0)$ and $\phi\in S_{2}(\Gamma_{0}(N))$ the associated theta series with $N=D_K{\rm N}_{K/\BQ}(\cond\lambda)$. Let $B$ the quaternion algebra over $\BQ$ so that $$D_{B}=\prod_{\eta_{K_{q}}(-1)=-1}q$$ 
(cf.~Lemma \ref{lm:disc}). 
{Let $\pi_\lambda$ be the cuspidal automorphic representation of $B_{\BA}^\times$ associated to $\phi_{\lambda}$ with conductor $N$. }

Let $\ell\nmid 2N$ and $p\nmid 2D_B$ be primes.
Let {$\varphi,\varphi^{\{p\}}\in \pi_{\lambda}$} be the $\ell$-optimally normalised test vectors in the cases when 
{$p\nmid N^-$, $p|N^-$} respectively as in \S\ref{ss:ntv}. 
We have the associated periods
{\[\begin{cases}
  \Omega_{\lambda}=\frac{8\pi ^2(\phi,\phi)}{\pair{\varphi,\varphi}},&\quad p\nmid N^-,\\
  \Omega_{\lambda}^{\{p\}}=\frac{8\pi ^2(\phi,\phi)}{\pair{\varphi^{\{p\}},\varphi^{\{p\}}}},&\quad p| N^-.\\
\end{cases}\]} The first period does not depends on the choice of $p$.

In the following subsections we study $\ell$-indivisibility of the $L$-value $L^{\text{alg}}(1/2,\pi_{\lambda,K}\otimes \chi)$
via studying $\ell$-divisibility of toric periods.

\subsection{$(\ell,p)$ non-vanishing}\label{ss:lp}
This subsection concerns the case $\ell \neq p$.

\subsubsection{Inert case}
Let $p$ be an odd prime inert in $K$. 

As before, let $K_{\infty}$ be the anticyclotomic $\BZ_p$-extension of $K$, $\Gamma=\Gal(K_{\infty}/K)$, and 
$\Xi_p$ the set of finite order characters of $\Gamma$. For $\nu\in\Xi_p$, the pairs $(\pi_\lambda,\nu)$ are self-dual with root number $+1$.

Put 
$$
\Xi_{\lambda,p}^{\pm}=\{ \nu \in \Xi_{p}| \, \epsilon(\lambda\nu)=\pm 1\},
$$
where $\epsilon(\lambda\nu)$ denotes the global root number.
One may consider non-vanishing of central $L$-values $L^{\text{alg}}(1/2,\pi_{\lambda,K}\otimes \chi)$ modulo $\ell$ for $\nu\in\Xi_{\lambda,p}^{+}$, where $\ell$ is a fixed prime.

\bigskip 
\underline{Case I. $p\nmid 2{\rm{N}}_{K/\BQ}(\cond{\lambda})$}
\bigskip

For $\nu\in \Xi_p$, we have
\begin{equation}\label{exproot}
\epsilon(\lambda\nu)=(-1)^{t}\epsilon(\lambda),
\end{equation}
where the associated local character $\nu_p$ is of conductor $p^{t}>1$ (cf.~\cite[Prop.~3.7]{MS}).

Our main result is the following. 
\begin{thm}\label{T:2.W} 
Let $\lambda$ be a self-dual Hecke character over $K$ of infinity type $(1,0)$ and $\pi_\lambda$ the associated cuspidal automorphic representation. Let $p\nmid 2{\rm{N}}_{K/\BQ}(\cond{\lambda})$ be a prime inert in $K$. 
Let $\ell \nmid 2p{\rm{N}}_{K/\BQ}(\cond{\lambda})$ be a prime. 
Then for all but finitely many $\nu\in\Xi_{\lambda,p}^{+}$, we have
\[v_{\ell}\left(\frac{L(1/2, \pi_{\lambda, K}\otimes\nu)}{\Omega_{\lambda}}\right)=0.\]

\end{thm}
\begin{proof} Choose a finite extension $O$ of $\Z_\ell$ in $\C_\ell$ so that $O$ contains $O_{\pi_{\lambda},\ell}$ and $\alpha_p$ 
and let $\uf$ be a uniformizer of $O$. 

Let $F_\ell$ be the $\ell$-adic avatar of the $p$-stabilization $\varphi^\dagger$ of $\varphi$ with respect to $\alpha_p$. For $R$ the order in the definition of $\varphi$, let $U=\wh{R}^\times$, we have $\varphi\in M_2(U,O)$. Put
$$U'=\wh{R}^{\times, (p)}{U_0(p)_p},$$ where $R$ is the order in the definition of $\varphi$.
Then $F_\ell\pmod{\varpi}\in M_2(U',{\bf k}_\ell)$ for ${\bf k}_\ell:=O/\uf {O}$.

 For each integer $n\geq 0$, put \[\Theta_n:=\sum_{[a]_n\in{\rm G}_n}F_\ell(x_n(a))\cdot [a]_n\in O[{\rm G}_n].\]

In view of explicit Waldspurger formula (cf.~Theorem~\ref{T:central.W} ), it suffices to show that $$v_{\ell}(\nu(\Theta_n))\neq 0$$ for all but finitely many $\nu \in \Xi_{\lambda,p}^+$.

Recall that ${\rm G}_n=(D_1\times D_0)\cdot \Gamma_{n}$. Note that elements in $D_0$ are represented by product of uniformizers of $K_q$ for $q|D_K$, it follows from the definition of $\varphi$ that
\begin{equation}\label{op_tv}\sum_{d\in D_0}\pi(\iota(d))F_\ell=|D_0|\cdot F_{\ell}.\end{equation}

Let $p^s$ be the order of the Sylow $p$-subgroup of ${\bf k}_\ell^\x$. Let $\nu:\Gamma_n\to\mu_{p^\infty}$ be a character of conductor $p^n$ with $n>2s$. Put
\[C_n=\{\gamma\in\Gamma_n\mid \nu(\gamma)\in {\bf k}_\ell^\x\}.\]
Note that $C_n=\Ker({\rm G}_{n}\to {\rm G}_{n-s})$. Let ${\bf k}_\ell(\nu)$ be the field generated by the values of $\nu$ over ${\bf k}_\ell$. Since ${\bf k}_\ell$ contains $\mu_p$, $d_\nu:=[{\bf k}_\ell(\nu):{\bf k}_\ell]$ is a $p$-power, and for a $p$-power root of unity $\zeta\in {\bf k}_\ell(\nu)$, we have
\[{\rm Tr}_{{\bf k}_\ell(\nu)/{\bf k}_\ell}(\zeta)=\begin{cases}0&\cdots \zeta\not\in {\bf k}_\ell,\\
d_\nu&\cdots \zeta\in {\bf k}_\ell.\end{cases}\]
It follows from the above that for each $a\in \wh{K}^\x$,
\begin{equation}\label{E:9.W}\begin{aligned}
{\rm Tr}_{{\bf k}_\ell(\nu)/{\bf k}_\ell}(\nu(a^{-1})\cdot\nu(\Theta_n)\pmod{\varpi})
{\equiv }&|D_0|d_\nu\cdot \sum_{\tau\in  D_1}\sum_{y\in\Z/p^s\Z}F_{\ell}(x_n(a\tau)
\begin{pmatrix}
1&\frac{y}{p^s}\\
0&1\\
\end{pmatrix}
\zeta_\nu^y{\pmod{\varpi}}
\end{aligned}
\end{equation}
for a primitive $p^s$-th root of unity $\zeta_\nu$. 

Define $\wtd F_\ell\in M_2({\bf k}_\ell(\nu))$ by
\begin{equation}\label{tv-tw}
\wtd F_\ell(g):=\sum_{y\in\Z/p^s\Z}\zeta_\nu^y\rho(
\begin{pmatrix}
1& \frac{y}{p^s}\\
0 &1\\
\end{pmatrix}
)
F_{\ell}(g)\pmod{\varpi}.
\end{equation}
Then $\wtd F_\ell\in M_2(\wtd{U}',{\bf k}_\ell(\nu))$ for
$\wtd{U}'=\{g\in U'\ |\ g_p\equiv \begin{pmatrix}
    1 & *\\ 0&1
\end{pmatrix}\pmod{p^{2s}}\}$.

By \eqref{E:9.W}, we have
\begin{equation}\label{E:10'.W} 
{\rm Tr}_{{\bf k}_\ell(\nu)/{\bf k}_\ell}(\nu(a^{-1})\cdot\nu(\Theta_n)\pmod{\varpi})=|D_0|d_\nu\cdot\sum_{\tau\in D_1}\wtd F_{\ell}(x_n(a\tau)).
\end{equation}
We next show that $\wtd F_\ell$ is non-Eisenstein. 

{Note that \eqref{lev} holds by Lemma \ref{lm:lev}.} 
Under our assumptions, {$F_{\ell}\pmod{\varpi}$} is non-Eisenstein by Lemmas \ref{L:Ihara2.W} and \ref{ee}. A simple calculation shows that
\begin{align*}
\sum_{a\in(\Z/p^s\Z)^\x}\rho
\begin{pmatrix}
a & \\
 & 1\\
\end{pmatrix}
\wtd F_{\ell}
{\equiv }p^s\cdot (1-p^{-1}\alpha_p\cdot \rho\left(
\begin{pmatrix}
p^{-1} & \\
 & 1\\
\end{pmatrix}
\right)
F_{\ell}{\pmod{\varpi}}.
\end{align*}
Therefore $\wtd F_{\ell}$ is non-Eisenstein by Lemma \ref{L:Ihara2.W}. 

In view of \eqref{E:10'.W} and Proposition \ref{C:Vatsal_Cornut}, it thus follows that 
$$
v_{\ell}(\nu(\Theta_n)) \neq 0
$$
for all but finitely many $\nu \in \Xi_{\lambda, p}^{{\varepsilon_0}}$, where ${\varepsilon_0}$ denotes the sign of $(-1)^{n}\epsilon(\lambda)$ for $n$ the parity arising from Proposition \ref{C:Vatsal_Cornut} (cf.~\eqref{exproot}). 

If $\varepsilon_0\neq +$, then $L(1, \lambda\nu)\neq  0$ for all but finitely many 
$\nu \in \Xi_{\lambda, p}^{-}$. But the latter $L$-value vanish since $\epsilon(\lambda\nu)=-1$ for any 
$\nu\in \Xi_{\lambda,p}^{-}$. Thus we have $\varepsilon_0=+$ and the parity $n$ in Proposition \ref{C:Vatsal_Cornut} satisfies $(-1)^n\epsilon(\lambda)=+1$. Moreover, $\wt{F}_\ell^{\sgn(\epsilon(\lambda))}$ is non-Eisenstein.
\end{proof}

\bigskip 
\underline{Case II. $p|{\rm{N}}_{K/\BQ}(\cond{\lambda})$}
\bigskip

Let $p^m$ be the conductor of $\lambda_p$ with $m>0$.
By \cite[Prop.~3.7]{MS}, 
for $\nu\in \Xi_p$ with $\cond{\nu_p}=p^{n}>p^m$, we have
$$
\epsilon(\lambda\nu)=(-1)^{n-m}\epsilon(\lambda).
$$

Our main result is the following. 
\begin{thm} \label{T:b.W} 
  Let $\lambda$ be a self-dual Hecke character over $K$ of infinity type $(1,0)$ and $\pi_\lambda$ the associated cuspidal automorphic representation.
  Let $p\mid {\rm{N}}_{K/\BQ}(\cond{\lambda})$ be an odd prime inert in $K$. 
  Let $\ell \nmid 2p{\rm{N}}_{K/\BQ}(\cond{\lambda})$ be a prime.  
  Then for all but finitely many {$\nu\in\Xi_{\lambda, p}^{+}$}, we have
  \[v_\ell\left(\frac{L(1/2, \pi_{\lambda,K}\otimes \nu)}{\Omega_{\lambda}^{\{p\}}}\right)=0.\]
  \end{thm}

  \begin{proof}
  The proof is essentially the same as in the Case I, except we consider 
 $F_{\ell}$ the $\ell$-adic avatar of the test vector $\varphi^{\{p\}}$.

Let $\nu:\Gamma_n\to\mu_{p^\infty}$ be a character of conductor $p^n$ with $n>\max\{2s,{2m}\}$.
 
 Define $\wtd F_{\ell}$ as in the proof of Theorem \ref{T:2.W}. Then $\wtd F_{\ell}\in M_2(\wtd{U}',{\bf k}_\ell(\nu))$ for $$\wtd{U}'=\{g\in U'\ |\ g_p\equiv \begin{pmatrix}
    1 & *\\ 0&1
\end{pmatrix}\pmod{p^{\max\{{2m},2s\}}}\}.$$
It suffices to show that $\wtd F_{\ell}$ is non-Eisenstein.

Note that $F_{\ell}\pmod{\varpi}$ is non-Eisenstein by Lemma \ref{ee} and \eqref{op_tv}.
We have 
  \[\begin{aligned}
    \sum_{a\in(\Z/p^s\Z)^\x}\rho
  \begin{pmatrix}
  a & \\
   & 1\\
  \end{pmatrix}\wtd F_{\ell}{\equiv }&p^{s}F_{\ell}-{p^{s-1}}\rho\left(\begin{pmatrix}
      p^{-1}& \\ & 1
  \end{pmatrix}\right)U_p 
      F_{\ell}\pmod{\varpi}\\
      {\equiv }&p^{s}F_{\ell}\pmod{\varpi}\\
  \end{aligned}\] 
  since the $U_p$-eigenvalue of a newform is $0$ whenever $p^2$ divides the conductor {of $\pi_{\lambda}$.}
Therefore $\wtd F_{\ell}$ is non-Eisenstein.
  
  \end{proof}
  \bigskip 
  \underline{Variant}. 
  \bigskip

Let $p$ and $\ell$ as before.
  Fix an odd prime $p_0\neq p\ell$ inert in $K$ such that $p_0^2| \cond\lambda$.

  Take $\wtd{\varphi}$ to be the {$\ell$-optimally normalised} test vector as in \S\ref{var:test} 
  {with $p_0=q$ therein.}
  Recall that $\wtd{\varphi}$ is new at $*$ and $p_0$, where 
  \[*=\begin{cases}
  p,&\quad p|N^-,\\
    \emptyset,&\quad p\nmid N^-.
  \end{cases}\] 
  For {$p\nmid N^-$}, note that $\wt{\varphi}$ does not depend on $p$. We denote $\wt{\varphi}$ by {$\begin{cases}
\varphi^{\{p_0\}},&\quad p\nmid N^-\\ 
\varphi^{\{p,p_0\}},&\quad p\mid N^-
\end{cases}$} to emphasise the dependence.
Define periods {$$\begin{cases}
  \Omega_{\lambda}^{\{p_0\}}=\frac{8\pi^2(\phi,\phi)}{\pair{\varphi^{\{p_0\}},\varphi^{\{p_0\}}}},&\quad p\nmid N^-,\\ 
   \Omega_{\lambda}^{\{p, p_0\}}=\frac{8\pi^2(\phi,\phi)}{\pair{\varphi^{\{p,p_0\}},\varphi^{\{p,p_0\}}}},&\quad p\nmid N^-.
  \end{cases}$$}

 The following result will be used in section \ref{s:nv-Hecke} to connect the mod $\ell$ non-vanishing of Hecke $L$-values with an $\ell$-integral comparison of periods (cf.~Theorem \ref{thm:nv+}).
  \begin{thm}\label{T:v.W} 
    Let $\lambda$ be a self-dual Hecke character over $K$ of infinity type $(1,0)$ and $\pi_\lambda$
   the associated cuspidal automorphic representation.  
    Let $p$ be an odd prime inert in $K$. 
    Let $\ell \nmid 2 p {\rm{N}}_{K/\BQ}(\cond{\lambda})$ be a prime. 
Let $p_0 \nmid p\ell$ be an odd prime inert in $K$ 
such that $p_0^2|\cond{\lambda}$ and $\ell\nmid p_0(p_0^2-1)$. Then for all but finitely many {$\nu\in\Xi_{\lambda,p}^{+}$}, we have
   {\[\begin{cases}
    v_\ell\left(\frac{L(1/2, \pi_{\lambda,K}\otimes\nu)}{\Omega_{\lambda}^{\{p_0\}}}\right)=0,&\quad p\nmid N^-,\\
     v_\ell\left(\frac{L(1/2, \pi_{\lambda,K}\otimes\nu)}{\Omega_{\lambda}^{\{p,p_0\}}}\right)=0,&\quad p\mid N^-.
    \end{cases}\]}
    \end{thm}

  \begin{proof}
Let $F_\ell$ be defined as in Cases I and II by replacing $\varphi$ therein with $\wt{\varphi}$. Put {$p_0^m=\cond \lambda_{p_0}$} and $$F_{\ell}'=\sum_{k\in O_{K_{p_0}}^\times/O_{K_{p_0},p_0^m}^\times }\pi_{\lambda}(\iota_{\varsigma_{p_0}}(k))F_{\ell},$$ where 
{$\iota_{\varsigma_{p_0}}$} is the local embedding $K_{p_0}\hookrightarrow B_{p_0}$ arising from $\theta$ and $u$ as in the second bullet point of Theorem~\ref{T:central.v1}. 

For integers $n\geq 0$, put {\[{\Theta_n}:=\sum_{[a]_n\in{\rm G}_n}F'_\ell(x_{n,p_0^{m}}(a))\cdot [a]_n\in O[{\rm G}_n],\]} 
{
where we use the modified $\varsigma^{(n)}$ in subsection \s\ref{var:test} to define CM points.}
In view of Theorem~\ref{T:central.v1} 
the proof of the $\ell$-indivisibility of $\nu(\Theta_n)$ is essentially the same as that of  Theorems  \ref{T:2.W} and \ref{T:3.W}: it suffices to show that the similarly defined $\tilde{F}_\ell'$ is non-Eisenstein, 
which is again a consequence of  
 Lemma \ref{tnn} and Corollary \ref{cml}.
\end{proof}

\subsubsection{Restriction of test vector to components of Shimura set}
We describe some consequences of the proof of Theorem \ref{T:2.W} which will be used in the split case. 

Let $\varphi\in M_2(U,\ov{\BQ})$ be the test vector as in \S\ref{SS:localtoric} associated to $\lambda$.  
\begin{prop}\label{nv}
If $\epsilon(\lambda)=\pm1$, then $\varphi^\pm\neq 0$ and $\varphi^\mp = 0$.
\end{prop}
\begin{proof}
In the following we choose an auxiliary prime $\ell\nmid 2{\rm{N}}_{K/\BQ}(\cond{\lambda})$ and normalise $\varphi$ to be $\ell$-optimal.

As seen in the proof of Theorem \ref{T:2.W}, for $p\nmid 2\ell{\rm{N}}_{K/\BQ}(\cond{\lambda})$ inert in $K$, we have

   \begin{align*}
    \sum_{a\in(\Z/p^s\Z)^\x}\rho
    \begin{pmatrix}
    a & \\
     & 1\\
    \end{pmatrix}
    \wtd F_{\ell}
    \equiv  &p^s\cdot (1-p^{-1}\alpha_p\cdot \rho(
    \begin{pmatrix}
    p^{-1} & \\
     & 1\\
    \end{pmatrix}
    )
    F_{\ell}\pmod{\varpi}.
    \end{align*}
and\[
      F_{\ell}= \varphi-\frac{1}{\alpha_p}{\rho}\begin{pmatrix}
        p^{-1}&\\ &1
        \end{pmatrix}\varphi.\]
Therefore, 
\begin{equation}\label{eq:tw}
p^{-s}\sum_{a\in(\Z/p^s\Z)^\x}\rho
\begin{pmatrix}
a & \\
 & 1\\
\end{pmatrix}
\wtd F_{\ell}^{}{\equiv }\varphi+\frac{1}{p}\rho \begin{pmatrix}
  p^{-2} & \\
   & 1\\
  \end{pmatrix}\varphi\pmod{\varpi}.
  \end{equation}
  
  Put $\epsilon$ for the sign of $\epsilon(\lambda)$.  As seen in the proof of Theorem~\ref{T:2.W}, $\wt{F}_{\ell}^{-\epsilon}$ is Eisenstein but $\wt{F}_{\ell}^{\epsilon}$ is non-Eisenstein.
 Moreover, $\sum_{a\in(\Z/p^s\Z)^\x}\rho
 \begin{pmatrix}
 a & \\
  & 1\\
 \end{pmatrix}
 \wtd F_{\ell}^{}$ is non-Eisenstein. 
 Therefore the proof shows that 
 $\sum_{a\in(\Z/p^s\Z)^\x}\rho
 \begin{pmatrix}
 a & \\
  & 1\\
 \end{pmatrix}
 \wtd F_{\ell}^{-\epsilon}$ is Eisenstein, and in turn $\sum_{a\in(\Z/p^s\Z)^\x}\rho
 \begin{pmatrix}
 a & \\
  & 1\\
 \end{pmatrix}
 \wtd F_{\ell}^{\epsilon}$ non-Eisenstein.

In view of \eqref{eq:tw} and the preceding paragraph, it follows that $\varphi^\epsilon \pmod{\varpi}$ is non-Eisenstein and \[\varphi^{-\epsilon}+\frac{1}{p}\rho \begin{pmatrix}
  p^{-2} & \\
   & 1\\
  \end{pmatrix}\varphi^{-\epsilon}\equiv 0\pmod{\varpi}\] by Lemma \ref{ee}. 
Therefore $\varphi^{-\epsilon}\pmod{\varpi}$ is Eisenstein by Lemma \ref{ee3} and so 
$$\varphi^{-\epsilon}\equiv 0\pmod{\varpi}$$ by Lemma \ref{ee} again. 
If $\varphi^{-\epsilon}\neq 0$, then for $\ell$ sufficiently large $\varphi^{-\epsilon}\nequiv 0\pmod{\varpi}$, so contradiction\footnote{Alternatively, one may apply the same argument modulo powers of $\ell$.}. 
 \end{proof}

We now describe an application to the split case. 
Let $p\nmid 2{\rm N}_{K/\BQ}(\cond{\lambda})$ be a prime split in $K$. 
As in the inert case \eqref{tv-tw}, define $\wt{F}_{\ell}$. 
\begin{cor}\label{ep1}Let $\ell\nmid 2{\rm{N}}_{K/\BQ}(\cond{\lambda})$ be a prime. If $\epsilon(\lambda)=\pm 1$, then $\wt{F}_{\ell}^{\pm}$ is non-Eisenstein and $\wt{F}_{\ell}^{\mp}=0$.
 \end{cor}
\begin{proof} 
{In light of Proposition \ref{nv}, $\varphi^{\pm}\pmod{\varpi}$ is non-Eisenstein and $\varphi^{\mp}=0$. To see the same claim for $\wt{F}_{\ell}^\pm$, just note that they are related as in the proof of Theorem \ref{T:2.W}, and then Lemma \ref{ee3} applies.}

\end{proof}

 \subsubsection{Split case}
 Let $\lambda$ be a self-dual Hecke character over $K$ of infinity type $(1,0)$ and $\pi_\lambda$  the associated cuspidal automorphic representation. 

Let $p\nmid 2{\rm{N}}_{K/\BQ}(\cond{\lambda})$ be a prime split in $K$.  For $\nu\in \Xi$ of order $p^{n}>1$, we have
$$
\epsilon(\lambda\nu)=
 \epsilon(\lambda).
$$
So the pair $(\pi_\lambda,\nu)$ is self-dual with root number $+1$. One may consider non-vanishing of central $L$-values $L^{\rm alg}(1/2,\pi_{\lambda,K}\otimes\nu)$ modulo $\ell$ whenever $\epsilon(\lambda)=+1$.

Our main result is the following.

\begin{thm}\label{T:3.W} 
Let $\lambda$ be a self-dual Hecke character over $K$ of infinity type $(1,0)$ and $\pi_\lambda$ the associated cuspidal automorphic representation. Let $p$ be an odd prime split in $K$. 
Let $\ell$ be a prime. Suppose that
\begin{itemize}
\item[(i)] $\ell\nmid 2p{\rm{N}}_{K/\BQ}(\cond{\lambda})$,
\item[(ii)] $\epsilon(\lambda)=+1$.
\end{itemize}
Then for all but finitely many $\nu\in\Xi_p$, we have
\[v_{\ell}\left(\frac{L(1/2,\pi_{\lambda,K}\otimes \nu)}{\Omega_{\lambda}}\right)= 0.\]
\end{thm}

\begin{proof} 
First consider $p\nmid 2{\rm{N}}_{K/\BQ}(\cond{\lambda})$.
The argument is similar to the proof of Theorem \ref{T:2.W}, whose notation will appear below.

Let $F_{\ell}:=\varphi^\dagger$ be the $\ell$-adic avatar of the $p$-stabilization.
 For each integer $n\geq 0$, put \[\Theta_n:=\sum_{[a]_n\in{\rm G}_n}F_\ell(x_n(a))\cdot [a]_n\in O[{\rm G}_n].\]
It suffices to show that $v_{\ell}(\nu(\Theta_n))= 0$ for all but finitely many $\nu \in \Xi_p$.

Recall that 
\begin{equation}\label{E:10.W} 
{\rm Tr}_{{\bf k}_\ell(\nu)/{\bf k}_\ell}(\nu(a^{-1})\cdot\nu(\Theta_n)\pmod{\mathfrak{\varpi}}){\equiv }|D_0|d_\nu\cdot\sum_{\tau\in D_1}\wtd F_{\ell}(x_n(a\tau)).
\end{equation}

Under the assumption $\epsilon(\lambda)=+1$, $\wt{F}^+_{\ell}$ is non-Eisenstein by Corollary~\ref{ep1}.
Therefore, in light of \eqref{E:10.W} and Proposition \ref{C:Vatsal_Cornut}, we conclude that 
$$
\nu(\Theta_n) \neq 0
$$
for all but finitely many $\nu \in \Xi_p$. 

Now consider the case $p|{\rm{N}}_{K/\BQ}(\cond{\lambda})$. Then $F_\ell$ is the $\ell$ adic avatar of $\varphi$. 
We have \[\begin{aligned}
    \sum_{a\in(\Z/p^s\Z)^\x}\rho
  \begin{pmatrix}
  a & \\
   & 1\\
  \end{pmatrix}\wtd F_{\ell}{\equiv }&p^{s}F_{\ell}-{p^{s-1}}\rho\left(\begin{pmatrix}
      p^{-1}& \\ & 1
  \end{pmatrix}\right)U_p 
      F_{\ell}\pmod{\varpi}\\
  \end{aligned}\] and $F_\ell$ is $U_p$-eigen with eigenvalue $\pm 1$ if $p\parallel{\rm{N}}_{K/\BQ}(\cond{\lambda})$ and $0$ else.
By Lemma \ref{ee3} and Proposition \ref{nv}, $\wtd{F}_{\ell}^+$ is non-Eisenstein.

\end{proof}

\subsection{The vanishing of $\mu$-invariants}\label{SS:padicL} 
We consider the $\mu$-invariant of Rankin--Selberg $p$-adic $L$-functions in the CM case. 

\subsubsection{Split case}

 \begin{thm}\label{T:1.W} 
  Let $\lambda$ be a self-dual Hecke character over $K$ of infinity type $(1,0)$ and $\pi_\lambda$ the associated cuspidal automorphic representation. Suppose that $\epsilon(\lambda)=+1$.

  Let $p\nmid 2{\rm{N}}_{K/\BQ}(\cond{\lambda})$ be a prime split in $K$ and $\mathscr{L}_{p}(\pi_\lambda)$ the associated $p$-adic $L$-function.   
Then 
$$\mu(\mathscr{L}_p(\pi_{\lambda}))=0.$$ 
\end{thm}
\begin{proof}
Let the notation be as in \S\ref{pordfun}.
The following is a variant of the strategy used for $(\ell,p)$-non-vanishing in \S\ref{ss:lp}.

Let $F_p
$ be the $p$-adic avatar of the $p$-stabilization of the $p$-primitive test vector $\varphi$ with respect to $\alpha_p$.
Note that 
\[\Theta_n(\pi_\lambda,1) \pmod{\uf}{\equiv }|D_0|\alpha_p^{-n}\sum_{[u]_n\in\Gamma_n}\left(\sum_{\tau\in D_1}F_{p}(x_n(u\tau))\right)\cdot[u]_n{\pmod{\varpi}}.\]

For the vanishing of the $\mu$-invariant of the theta element $\Theta_\infty(\pi_{\lambda})$, it suffices to show that for $n\gg 0$, there exists $a\in\wh{K}^\times$ such that
\[\sum_{\tau\in D_1}F_p(x_n(a\tau))\nequiv 0\pmod{\uf}.\]
In turn, it suffices to verify the hypotheses
of Proposition \ref{C:Vatsal_Cornut}(b)
 for $F_{p}^+$, which may be seen as follows. 
By Proposition \ref{nv}, $\varphi^+\pmod{\varpi}$ is non-Eisenstein, and consequently so is $F_{\ell}^+\pmod{\varpi}$ by Lemma \ref{ee3}.

\end{proof}
\begin{remark}
If $\lambda$ has root number $-1$, then $\mathscr{L}_{p}(\pi_{\lambda})=0$. 
\end{remark}
\subsubsection{Inert case}

 \begin{thm}\label{T:4.W} 
  Let $\lambda$ be a self-dual Hecke character over $K$ of infinity type $(1,0)$ and $\pi_\lambda$ the associated cuspidal automorphic representation.

Let $p\nmid 2{\rm{N}}_{K/\BQ}(\cond{\lambda})$ be a prime inert in $K$.
Then
$$\mu(\mathscr{L}_p(\pi_{\lambda}))=0.$$ 
\end{thm}
\begin{proof} 
Let the notation be as in the proof of Proposition~\ref{prop:pLss}. 

Recall that
$$
\wt{\Theta}_{n}(\pi_{\lambda})=\omega_{n}^{\epsilon}\wt{\Theta}_{n}^{\epsilon}(\pi_{\lambda}),\, \Theta_{\infty}^{\epsilon}(\pi_{\lambda})=\lim \wt{\Theta}_{n}^{\epsilon}(\pi_{\lambda}).
$$
for $\epsilon$ the parity of $n$. 
Note that $\mu(\omega_{n}^\pm)=0$. 

In view of Definition \ref{def:prim} and \eqref{pmL_par}, it suffices to show the vanishing of the $\mu$-invariant of $\Theta_\infty^{-\varepsilon}(\pi_{\lambda})$. By the above discussion, this is equivalent to 
$$\mu(\wt{\Theta}_{n}(\pi_{\lambda}))=0$$ for $n \gg 0$ of the same parity as $-\varepsilon$. The latter $p$-indivisibility follows by the same argument as in the proof of  Theorem~\ref{T:2.W}. 
\end{proof}

\section{Non-vanishing of Hecke $L$-values modulo $\ell$}\label{s:nv-Hecke}

This section establishes our main results on the non-vanishing of Hecke $L$-values building on 
the Rankin--Selberg results in section \ref{s:nv}. The bridge among the two arises from  comparison of quaternionic and CM periods, which constitutes the core of the section. 

{The main results are Corollary \ref{cor:per} on the $(\ell,p)$ non-vanishing of Hecke $L$-values and Theorem \ref{thmC'} concerning $\mu$-invariant.} Along the way we prove the comparison of periods (cf.~Theorems~\ref{thm:nv+} and~\ref{thm:per}).

\subsection{Backdrop}
\subsubsection{Setting}Let $\lambda$ be a self-dual Hecke character over an imaginary quadratic field $K$ of infinity type $(1,0)$. Let $\phi_\lambda\in S_{2}(\Gamma_{0}(N))$ be the associated weight $2$ theta series associated for $$N=D_{K}{\rm{N}}_{K/\BQ}(\cond{\lambda}).$$ Note that $D_{K}|{\rm{N}}_{K/\BQ}(\cond{\lambda})$ by the self-duality.

Let $B$ be the definite quaternion algebra over $\BQ$ such that  \[\epsilon(B_q)=\eta_{K_q}(-1)\] 
for any $q$ (cf.~Lemma~\ref{lm:disc}). 
Let $\pi_\lambda$ be the cuspidal irreducible automorphic representation of $B^\times_{\BA}$ associated to {$\phi_{\lambda}$.} 

Let $\ell \nmid 2N$ be a prime. 
Let $\varphi_{\lambda}\in \pi_\lambda^{\wh{R}^\times}$ be the toric vector as in Definition \ref{D:1.W}, which is $\ell$-primitive and $K_q^\times$-invariant for all $q|{\rm{N}}_{K/\BQ}(\cond{\lambda})$ non-split in $K$.

{Note that for any finite order Hecke character $\chi$ over $K$, we have a factorisation
\begin{equation}\label{L-fac}
L(1/2,\pi_{\lambda,K}\otimes\chi)=L(1,\lambda\chi)L(1,\lambda\chi^{-1}).
\end{equation}}

{In the context of Ranin--Selberg $L$-values the period}
\[\Omega_{\lambda}=\frac{8\pi^2(\phi_\lambda,\phi_\lambda)}{\langle \varphi_{\lambda}, \varphi_{\lambda} \rangle}\]
arises naturally. 
{On the other hand, we have a CM period $\Omega_{K}$ associated to Hecke $L$-values over $K$, which is well defined up to $\ell$-adic units (cf.~\S\ref{s:Intro}). }
In light of the factorisation \eqref{L-fac} of $L$-values a basic problem is to compare the periods 
$\Omega_\lambda$ and $\Omega_K$.

For a prime $p$, recall $\Gamma=\Gal(K_{\infty}/K)$ is the Galois group of the anticyclotomic $\BZ_p$-extension of $K$, 
$\Xi_p$ the set of finite order characters of $\Gamma$ and
$$
\Xi_{\lambda,p}^{+}=\{ \nu \in \Xi_{p}| \, \epsilon(\lambda\nu)=+1 \}. 
$$
\subsubsection{}\label{lowerbd}
The subsection describe a lower bound for $\ell$-adic valuation of Hecke $L$-values and of periods in terms of certain local invariants.

For $q|{\rm{N}}_{K/\BQ}(\cond{\lambda})$ such that $q$ is non-split in $K$, put $\mu_\ell(\lambda_q)=\inf_{x\in O_{K_q}^\times}v_\ell (\lambda_q(x)-1)$.

\begin{lem}\label{lm:lb}
Let $\lambda$ be a self-dual Hecke character over an imaginary quadratic field $K$ of infinity type $(1,0)$. Let $\ell\nmid 2{\rm{N}}_{K/\BQ}(\cond{\lambda})$ be a prime. Then
  \[v_{\ell}\left(\frac{L(1,\lambda)}{\Omega_K}\right)\geq \sum_{q|{\rm{N}}_{K/\BQ}(\cond{\lambda}) \text{ inert}}\mu_{\ell}(\lambda_q). \]
  \end{lem}
  \begin{proof}
  This is due to Finis \cite[Propositions 3.6 and 3.7]{Fin1}.
  (For $\ell$ prime to $2q$, note that $\mu_\ell(\lambda_q)=0$ if $K_q$ is ramified.)
\end{proof}
\begin{remark}\label{vm}Let $q$ be an inert prime.
  \begin{itemize}
   \item[\tiny{$\bullet$}] If conductor of $\lambda_q$ is $q$ and $(\ell,p+1)=1$, then $v_\ell(\lambda_q)=0$. 
      \item[\tiny{$\bullet$}] If conductor of $\lambda_q$ is at least $q^2$, then $\mu_\ell(\lambda_q)=0$. 
        \item[\tiny{$\bullet$}] Consider anticyclotomic $p$-power order twist $\lambda\nu$ with $p$ inert {and $\nu\in\Xi_{p}$}. If $q\neq p$ is inert and divides ${\rm{N}}_{K/\BQ}(\cond{\lambda})$, then $\nu_q$ is trivial, and so {$\mu_{\ell}(\lambda_q\nu_q)=\mu_{\ell}(\lambda_q)$}. If $p=q$, then $\mu_\ell(\lambda_p\nu_p)=0$ for $\nu_p$ so that {$\cond{\nu_p}>\max\{p,\cond{\lambda_p}\}$.} 
  \end{itemize}
\end{remark}
\begin{lem}\label{lm:per}
Let $\lambda$ be a self-dual Hecke character over an imaginary quadratic field $K$ of infinity type $(1,0)$. Let $\ell\nmid 2{\rm{N}}_{K/\BQ}(\cond{\lambda})$ be a prime. 
Then
   \[v_{\ell}\left(\frac{\Omega_\lambda}{\Omega_K^2}\right)\geq 2\sum_{\text{$q|{\rm{N}}_{K/\BQ}(\cond{\lambda})$ inert}}\mu_\ell(\lambda_q).\]
\end{lem}
\begin{proof}
Let $p\nmid 2\ell{\rm{N}}_{K/\BQ}(\cond{\lambda})$ be a prime inert\footnote{If $\epsilon(\lambda)=+1$, one may use split $p$ in the proof.} in $K$. 

By Theorem \ref{T:2.W}, for all but finitely many 
$\nu\in \Xi_{\lambda,p}^{+}$, we have
\[v_{\ell}\left(\frac{L(1,\lambda\nu)L(1,\lambda\nu^{-1})}{\Omega_\lambda}\right)=0\]
Then for such a $\nu$, 
\[v_{\ell}\left(\frac{L(1,\lambda\nu)L(1,\lambda\nu^{-1})}{\Omega_{K}^{2}}\right)=
v_{\ell}\left(\frac{\Omega_{\lambda}}{\Omega_{K}^{2}}\right).
\]
Therefore Lemma~\ref{lm:lb} concludes the proof. 
\end{proof}

\subsection{Comparison of periods and $(\ell,p)$ non-vanishing}\label{prnv}
This subsection establishes the comparison and $(\ell,p)$ non-vanishing of Hecke $L$-values almost simultaneously. 

{{We begin with an outline of the strategy. 
In light of the $(\ell,p)$ non-vanishing result for Rankin--Selberg $L$-values established in \S\ref{s:nv} and the $\ell$-divisibility lower bound for the Hecke $L$-values in \S\ref{lowerbd}, the sought after non-vanishing of Hecke $L$-values and the comparison of periods are equivalent. Since the non-vanishing of Hecke $L$-value in the $p$ split case is known due to Finis \cite{Fin1} and the comparison of periods essentially does not depend on the prime $p$, the non-vanishing in the $p$ inert case follows if $\epsilon(\lambda)=+1$.
To approach the case $\epsilon(\lambda)=-1$ and $p$ inert, we find another link between the non-vanishing and comparison of periods via an anticyclotomic twist, leading to a connection between the root number $+1$ and $-1$ cases! It turns out that the variant non-vanishing - Theorem~\ref{T:v.W} - is the key to such a connection (see the proof of Theorem \ref{thm:nvb}).}}

\subsubsection{The case $\epsilon(\lambda)=+1$}
\begin{thm}\label{thm:nv+}
Let $\lambda$ be a self-dual Hecke character over an imaginary quadratic field $K$ of infinity type $(1,0)$ with $\epsilon(\lambda)=+1$. Let $\ell\nmid 2{\rm{N}}_{K/\BQ}(\cond{\lambda})$ be a prime. 
  \begin{itemize}
    \item [(i)] We have 
     \[v_{\ell}\left(\frac{\Omega_\lambda}{\Omega_K^2}\right)=2\sum_{\text{$q|{\rm{N}}_{K/\BQ}(\cond{\lambda})$ inert}}\mu_\ell(\lambda_q).\]
    \item [(ii)]Let $p\nmid 2\ell {\rm{N}}_{K/\BQ}(\cond{\lambda})$ be a prime. 
    Then for all but finitely many $\nu\in \Xi_{\lambda,p}^+$, we have
    \[
    v_\ell\left(\frac{L(1,\lambda\nu)}{\Omega_K}\right)=\sum_{\text{$q|{\rm{N}}_{K/\BQ}(\cond{\lambda})$ inert}}\mu_\ell(\lambda_q).\]
  \end{itemize}
\end{thm}
\begin{proof}
We first show that without any assumption on $\epsilon(\lambda)$, the first assertion and the second assertion for a fixed prime $p\nmid 2\ell {\rm{N}}_{K/\BQ}(\cond{\lambda})$ are equivalent.  

By Theorems \ref{T:2.W} and \ref{T:3.W}, for all but finitely many 
$\nu\in \Xi_{\lambda,p}^+$, we have 
\begin{equation}\label{cpp0}v_{\ell}\left(\frac{L(1,\lambda\nu)L(1,\lambda\nu^{-1})}{\Omega_K^2}\right)-2\sum_{\text{$q|{\rm{N}}_{K/\BQ}(\cond{\lambda})$ inert}}\mu_{\ell}(\lambda_q)=v_{\ell}\left(\frac{\Omega_\lambda}{\Omega_K^2}\right)-2\sum_{\text{$q|{\rm{N}}_{K/\BQ}(\cond{\lambda})$ inert}}\mu_\ell(\lambda_q).\end{equation}
Note that $\nu_q$ is trivial character for $q|{\rm{N}}_{K/\BQ}(\cond{\lambda})$ inert. So in view of Lemma \ref{lm:lb} and the third part of Remark \ref{vm}, we have
    \[v_{\ell}\left(\frac{L(1,\lambda\nu)}{\Omega_K}\right),\quad v_{\ell}\left(\frac{L(1,\lambda\nu^{-1})}{\Omega_K}\right)\quad \geq \sum_{\text{$q|{\rm{N}}_{K/\BQ}(\cond{\lambda})$ inert}}\mu_{\ell}(\lambda_q)\]  for all $\nu\in \Xi_{\lambda,p}^+$ such that $\cond{\nu_p}\geq p^2$. Therefore the two assertions are equivalent.

  Under the condition $\epsilon(\lambda)=+1$, the second assertion for a prime $p\nmid 2\ell {\rm{N}}_{K/\BQ}(\cond{\lambda})$ split in $K$ is a result of Finis \cite[Thm.~1.1]{Fin1}. The proof concludes.

  \end{proof}

\subsubsection{An intermediate case}
To connect the root number $+1$ and $-1$ cases, we consider anticyclotomic twist at an auxiliary inert prime as described below.

For {$\wt{\lambda}$} a self-dual Hecke character over $K$ with infinitely type $(1,0)$ and $r|{\rm{N}}_{K/\BQ}(\cond{\wt{\lambda}})$ an odd inert prime, let {$\varphi_{\wt{\lambda}}^{\{r\}}$} be the $\ell$-primitive test vector defined in \S\ref{test} which is of $U_0((\cond{\wt{\lambda})^2})_r$ level at $r$. Recall that $\BC\varphi_{\wt{\lambda}}$ and $\BC{\varphi}_{\wt{\lambda}}^{\{r\}}$ differ at most at $r$, where the former is $K_{r}^\times$-invariant and the latter a newform at $r$. Put \[{\Omega}^{\{r\}}_{\wt{\lambda}}=\frac{8\pi^2(\phi_{\wt{\lambda}},\phi_{\wt{\lambda}})}{\langle {\varphi}^{\{r\}}_{\wt{\lambda}},{\varphi}^{\{r\}}_{\wt{\lambda}} \rangle}.\]

We present the following variant of Theorem \ref{thm:nv+}.

\begin{thm}\label{thm:nvb}
  Let $\lambda$ be a self-dual Hecke character over $K$ with infinitely type $(1,0)$. Let $\ell\nmid 2{\rm{N}}_{K/\BQ}(\cond{\lambda})$ be a prime. Let $r$ be an odd inert prime so that $\ell\nmid r(r^2-1)$. 
  \begin{itemize}
  \item [(i)]For any anticyclotomic $\chi\in \Xi_{r}$, we have 
  \[\begin{cases}
\displaystyle     v_{\ell}\left(\frac{{\Omega}^{\{r\}}_{\lambda\chi}}{\Omega_K^2}\right)= v_{\ell}\left(\frac{{\Omega}_{\lambda\chi}}{\Omega_K^2}\right)=2\sum_{\substack{\text{$q|{\rm{N}}_{K/\BQ}(\cond{\lambda})$ inert}\\ q\nmid r}}\mu_\ell(\lambda_q),\quad& \substack{\text{if $\cond{\lambda_r\chi_r}\geq {r^2}$}},\\
  \displaystyle  v_{\ell}\left(\frac{{\Omega}_{\lambda\chi}}{\Omega_K^2}\right)=2\sum_{\substack{\text{$q|{\rm{N}}_{K/\BQ}(\cond{\lambda})$ inert}\\ q\nmid r}}\mu_\ell(\lambda_q),\quad& \text{if $\lambda_r\chi_r$ is unramified}.\\
  \end{cases}\]
    \item [(ii)] For all but finitely many anticyclotomic $\nu\in \Xi_{\lambda,r}^+$, we have
    \[
    v_\ell\left(\frac{L(1,\lambda\nu)}{\Omega_K}\right)=\sum_{\substack{\text{$q|{\rm{N}}_{K/\BQ}(\cond{\lambda})$ inert}\\ q\nmid r}}\mu_\ell(\lambda_q).\]
  
   \end{itemize}

\end{thm}
\begin{proof}

We first show that the assertion (i) for a given $\chi$ is equivalent to the assertion  (ii). The case $r\nmid{\rm{N}}_{K/\BQ}(\cond{\lambda\chi})$ is treated in the proof of Theorem \ref{thm:nv+} (without the assumption $\ell\nmid (r^2-1)$). 

 In the following we show the equivalence for $\lambda\chi$ ramified at $r$ if 
 $\ell\nmid r(r^2-1)$.
 
 By Theorem \ref{T:b.W} for the $\BZ_{r}$-anticyclotomic twist family of $\lambda$ for test vector 
{${\varphi}^{\{r\}}_{\lambda\chi}$}, for all but finitely many $\nu\in\Xi_{\lambda\chi,r}^+$, 
we have   
\begin{equation}\label{nv2}
  v_\ell\left(\frac{{\Omega}^{\{r\}}_{\lambda\chi}}{\Omega_K^2}\right)=v_\ell\left(\frac{L(1,\lambda\chi\nu)L(1,\lambda\chi\nu^{-1})}{\Omega_K^2}\right).
\end{equation}

Let $p'\nmid 2\ell r {\rm{N}}_{K/\BQ}(\cond{\lambda})$ be an odd prime inert in $K$. Applying the non-vanishing results for $\BZ_{p'}$-anticyclotomic twist of $\lambda$ for the test vector $\varphi_{\lambda\chi}$ as in Theorem \ref{T:2.W} and ${\varphi}_{\lambda\chi}^{\{r\}}$ as in Theorem \ref{T:v.W}, under the pertinent hypotheses on $\ell$ and $r$, we have 
{
\[v_\ell\left(\frac{L(1,\lambda\chi\nu)L(1,\lambda\chi\nu^{-1})}{{\Omega}_{\lambda\chi}}\right)=0=v_\ell\left(\frac{L(1,\lambda\chi\nu)L(1,\lambda\chi\nu^{-1})}{{\Omega}^{\{r\}}_{\lambda\chi}}\right)\]for all but finitely many $\nu\in \Xi_{\lambda\chi,p'}^+$. 
It follows that}
\begin{equation}\label{cp2}v_\ell\left(\frac{\Omega_{\lambda\chi}}{\Omega_K^2}\right)=v_\ell\left(\frac{{\Omega}^{\{r\}}_{\lambda\chi}}{\Omega_K^2}\right).\end{equation}

In light of \eqref{nv2} and \eqref{cp2}, 
the analysis in the proof Theorem \ref{thm:nv+} leads to the desired equivalence.

Since $r$ is inert, we may choose $\chi$ with {$\cond{\chi_r\lambda_r}\geq r^2$} so that $\epsilon(\lambda\chi)=+1$. Applying Theorem 
\ref{thm:nv+}, the part (i) holds for such a $\chi$, concluding the proof.
  \end{proof}
  \begin{remark}
  The result allows $r$ to divide the conductor of $\lambda$. Moreover, the first part allows $\lambda$ to vary in its $\BZ_r$-anticyclotomic twist family and the second holds without the root number condition (cf.~Theorem~\ref{thm:nv+}).  
  \end{remark}

\subsubsection{The case $\epsilon(\lambda)=-1$}
\begin{thm}\label{thm:per}
  Let $\lambda$ be a self-dual Hecke character over an imaginary quadratic field $K$
  of infinity type $(1,0)$
  with $\epsilon(\lambda)=-1$. Let $\ell\nmid 6{\rm{N}}_{K/\BQ}(\cond{\lambda})$ be a prime. 
    \begin{itemize}
      \item  [(i)] We have 
      \[v_{\ell}\left(\frac{\Omega_\lambda}{\Omega_K^2}\right)=2\sum_{\text{$q|{\rm{N}}_{K/\BQ}(\cond{\lambda})$ inert}}\mu_\ell(\lambda_q).\]
      \item[(ii)] For $p\nmid 2\ell {\rm{N}}_{K/\BQ}(\cond{\lambda})$ be a prime. 
      Then for all but finitely many $\nu\in \Xi_{\lambda,p}^+$, we have
      \[
      v_\ell\left(\frac{L(1,\lambda\nu)}{\Omega_K}\right)=\sum_{\substack{\text{$q|{\rm{N}}_{K/\BQ}(\cond{\lambda})$ inert}}}\mu_\ell(\lambda_q).\]
    \end{itemize}
  \end{thm}
\begin{proof} The equivalence between the first and the second parts has been shown in the proof of Theorem \ref{thm:nv+}.
Now it is enough to establish the second part for a particular prime $r$.

Let $r\nmid 2\ell {\rm{N}}_{K/\BQ}(\cond{\lambda})$ be a prime\footnote{It exists since $\ell>3$.}  inert in $K$ such that 
 $\ell\nmid r(r^2-1)$.
By Theorem \ref{thm:nvb}(i) for $\chi=1$, the assertion follows.

\end{proof}

\subsubsection{}
We summarise consequences for Hecke $L$-values. 
\begin{cor}\label{cor:per}
Theorem~\ref{thmA} holds.
  \end{cor}

  \begin{proof}
  For $p$ split in $K$, the result is due to Finis {(cf.~\cite[Thm.~1.1]{Fin1}).} The inert case { is the content of Theorems \ref{thm:nv+} and \ref{thm:per}}.
   \end{proof}

\begin{cor}
Theorem~\ref{thmB} holds.
\end{cor}
\begin{proof}
If $p$ splits in $K$, then the result follows from the interpolation formula for the $p$-adic $L$-function $\mathscr{L}_{p}(\pi_{\lambda})$ (cf.~Theorem~\ref{T:Thetaevaluation.W}), the vanishing of its $\mu$-invariant (cf.~Theorem~\ref{T:1.W}) and the comparison of periods (cf.~Theorems~\ref{thm:nv+} and~\ref{thm:per}). 

Likewise, the inert case follows from Theorems~\ref{p-adic:RS} and \ref{T:4.W}, and the period comparison.  \end{proof}
\subsection{Rubin's $p$-adic $L$-function}
\begin{thm}\label{thmC'}
Let $\lambda$ be a Hecke character over an imaginary quadratic field $K$ of infinity type $(1,0)$ 
 such that 
 $\lambda \circ {\rm{N}}_{H/K}$ 
is associated to a $\BQ$-curve $E$ over $H$ with good reduction at a prime $p\nmid 6h_K$ inert in $K$. 
Let $\mathscr{L}_{p}(\lambda)$ be an associated Rubin $p$-adic $L$-function. 
Then 
$$
\mu(\mathscr{L}_{p}(\lambda))=
0. 
$$
\end{thm}
\begin{proof}
By Proposition~\ref{prop:rel} and Theorem~\ref{T:4.W}, 
$$
\mu(\mathscr{L}_{p}(\lambda))=\sum_{\substack{\text{$q|{\rm N}_{K/\BQ}(\cond{\lambda})$ inert}}}\mu_{p}(\lambda_q). 
$$
As explained below, the right hand side vanishes.

We have
\[ \sum_{\substack{\text{$q|{\rm{N}}_{K/\BQ}(\cond{\lambda})$ inert}}}\mu_p (\wh{\lambda}_q)=\sum_{\substack{\text{$q|{\rm{N}}_{K/\BQ}(\cond{\lambda})$ inert}}}\mu_p(\lambda_q)\]
for $\wh{\lambda}$ the $p$-adic avatar of $\lambda$
and $\mu_p(\wh{\lambda}_q):=\inf_{x\in O_{K_q}^\times}v_p (\wh{\lambda}_{q}(x)-1)$.
{
 Since $H/K$ is unramified, for each $q|{\rm N}_{K/\BQ}(\cond\lambda)$ inert in $K$ and $\fq$ a prime of $H$ above $q$, the norm map $O_{H_\fq}^\times\ra O_{K_q}^\times$ is surjective\footnote{In fact the norm map is identity as $q$ splits completely in $H$.}. Thus $\mu_p(\wh{\lambda}_q)=\mu_p((\wh{\lambda}\circ{\rm N}_{H/K})_{\fq})$, where the latter  is similarly defined. 
 
 Note that $\wh{\lambda}\circ {\rm N}_{H/K}$ factors through $\Gal(H(E[p^\infty])/H)\subset \Aut_{O_{K_p}}(E[p^\infty])=O_{K_p}^\times$. Write $\Gal(H(E[p^\infty])/H)$ as $\BZ_p^2\times \Delta$, for which $p\nmid \# \Delta$. {Hence, we have $$(\wh{\lambda}\circ{\rm N}_{H/K})|_{O_{H_\fq}^\times}\subset (\wh{\lambda}\circ{\rm N}_{H/K})|_{\Delta}$$ with order coprime to $p$.} It follows that
 $\mu_p((\wh{\lambda}\circ{\rm N}_{H/K})_{\fq})=0$.
}
\end{proof}

 \section{Newforms as test vectors for supercuspidal representations}\label{s:ntv}
In this section we show that newform is a test vector for certain self-dual pairs $(\pi,1)$ with $\pi$ supercuspidal, and calculate the associated toric period. 
For a given prime $\ell$, it is also shown that the latter is an $\ell$-adic unit 
under some 
conditions. {The main result is Theorem \ref{ml}.}

The explicit study of such toric periods is a key to arithmetic applications, such as Theorems~\ref{thmA},~\ref{thmB} and \ref{thmC} (see also \cite{Ti,HSY,HSY17,HN}).

The setting and notation of this section are independent from the rest.

\subsection{Main result}\label{ss1}
\subsubsection{Setting}\label{ss:lt-set}
 Let $q$ be an odd prime and $K/\BQ_q$ an unramified quadratic extension. Let $\eta_K$ be the associated quadratic character of $\BQ_q^\times$.

 Let $\lambda$ be a character of $K^\times$ of {exponential} conductor $m\geq 2$ such that $\lambda|_{\BQ_q^\times}=\eta_{K}$. 

 Let $\pi=\pi_\lambda$ be the associated representation of $\PGL_2(\BQ_q)$, which has {exponential} conductor $n=2m$. Note that $(\pi, 1)$ is a self-dual pair. Moreover, the Tunnell--Saito condition is satisfied as seen in the proof of Lemma \ref{lm:disc}. The primary goal of this section is to consider $K^\times$-toric period of newforms in $\pi$ with respect to the following embedding $K\hookrightarrow M_{2}(\BQ_q)$. 

{
Put $$
\text{$M_0(q^{2m})=\left\{\begin{pmatrix}
  a& b\\ c& d
\end{pmatrix}\in M_2(\BZ_q)\ |\ \  q^{2m}|c\right\}$ and $U_0(q^{2m})=M_0(q^{2m})\cap \GL_2(\BZ_q)$}.$$ 
The following family of embeddings $\iota: K\hookrightarrow M_{2}(\BQ_{q})$ satisfy $\iota K \cap M_0(q^{2m})=\iota O_{K,q^m}$ for $O_{K,q^m}:=\BZ_q+q^{m}O_K$.}

{
Let $\theta\in K$ be a unit so that $\ov{\theta}=-\theta$, where $\ov{\cdot}$ denotes the action of non-trivial element in $\Gal(K/\BQ_{q})$.
Pick $u\in \BZ_q^\times$ such that\footnote{{It suffices to solve $u^2\theta^2-1\equiv x^2\pmod{q}$ for $u,x\in \BF_q^\times$. Consider the surjective map $\BF_{q^2}^\times\ra \BF_q^{\times}, x+u\theta\mapsto x^2-u^2\theta^2$. Then $x^2-u^2\theta^2=-1\pmod{q}$ has $q+1$ solutions {$x+u\theta$} with $u,x\in\BF_q$, and at most $\begin{cases}4,&\quad q>3,\\ 2,&\quad q=3\end{cases}$ solutions with $x$ or $u=0$ in $ \BF_q$.}} \[u^2\theta^2-1\in \BZ_q^{\times 2}.\]
Define an embedding $\iota:K\hookrightarrow M_2(\BQ_q)$ by 
\[\begin{aligned}
  \theta\mapsto &
\begin{pmatrix}
  q^{-m}&\\ & 1
\end{pmatrix}\begin{pmatrix}
 1 &-u\\ &1
\end{pmatrix}\begin{pmatrix}
  &1\\
  \theta^2&
\end{pmatrix}\begin{pmatrix}
 1 &u\\ &1
\end{pmatrix}\begin{pmatrix}
 q^m &\\ & 1
\end{pmatrix}
=&\begin{pmatrix}
  -u\theta^2& \frac{1-u^2\theta^2}{q^m} \\ q^m\theta^2& u\theta^2
\end{pmatrix}.\\
\end{aligned}\]
}
We emphasise that the embedding of the unramified
torus in $\PGL_2(\Q_q)$ depends only on the conductor of the representation. 

Let \[f\in \pi^{U_0(q^{2m})}\] be a newform in $\pi$.
Denote by $(\ ,\ )$ a $\PGL_2(\BQ_q)$-invariant non-degenerate Hermitian pairing on $\pi$. 
The primary object of this section is the toric period
$$
\gamma_{\theta,u}:=\frac{1}{\vol(K^\times/\BQ_q^\times)(f,f)}\int_{K^\times/\BQ_q^\times}(\pi(\iota(t))f,f)d^\times t .
  $$
\subsubsection{Results}
    { \begin{thm}\label{ml}
      Let the setting be as in \S\ref{ss:lt-set}. Then for $\gamma_{\theta,u}$ the toric period of a newform $f\in \pi$, we have 
     \[
          \begin{aligned}
     \gamma_{\theta,u} =&\frac{1}{(1-q^{-2})q^m}\left(2+\epsilon(\pi)(\lambda^{-1}(a_0+\theta u)+\lambda^{-1}(-a_0+\theta u))\right).\end{aligned}\]
          Here $a_0\in (\BZ/q^{m}\BZ)^\times$ is a solution of $1+(a^2-\theta^2u^2)\equiv 0\pmod{q^m}$.
        \end{thm}}
To discuss non-vanishing of the toric period, consider the decomposition $$O_K^\times=\mu_{K}(1+qO_K)$$ {with $\mu_K\subset O_K^\times$ the torsion subgroup} and let ${\rm pr}: O_K^\times\ra 1+qO_K$ be the projection map. 
{\begin{cor}\label{cml}\ {Let the setting be as in \S\ref{ss:lt-set}.}
  \begin{itemize}
    \item [(i)] Any newform $f\in \pi$ is a test vector for the pair $(\pi,1)$, i.e. 
    $$
  \gamma_{\theta,u} \neq 0, 
    $$
  except in the case that $a_0+u\theta\in \mu_{K}\cdot {\rm pr}(\ker(\lambda))$ and $\lambda((a_0+u\theta))=-\epsilon(\pi)$. Moreover, for a given $\lambda$ and $\theta$, there exists $u$ such that $\gamma_{\theta,u}\neq 0$.
    \item[(ii)]Let $\ell\nmid q$ be a prime.  
    Then for a given $\theta$, there exists $u$ such that 
    \[v_\ell((q^2-1)\gamma_{\theta,u})=0.\]
    \end{itemize}
  \end{cor}}
\begin{proof}\
\begin{itemize}
 
\item[(i)]   
{
Suppose that the toric period $\gamma_{\theta,u}$ vanishes. In view of Theorem~\ref{ml}, we then have
$$\lambda(a_0+u\theta)=-\epsilon(\pi).$$
This implies that ${\rm pr}(a_0+u\theta)\in \ker(\lambda)$, otherwise $\lambda({\rm pr}(a_0+u\theta))$ would be a non-trivial $q$-th power root of unity 
since $\lambda|_{(1+qO_K)/(1+q\BZ_q)}$ is primitive with ${(1+qO_K)}/(1+q\BZ_q)\simeq \BZ/q^{m-1}\BZ$.
  }

The furthermore part follows from the fact that: As $u$ varies, $\lambda({\rm pr}(a_0+\theta u))$ can be any $q^{m-1}$-th root of unity, {$m\geq 2$}. ({Note that here $a_0$ is determined by $u$ and $\theta$.})

Indeed, suppose that 
$a^2-\theta^2u^2=-1$ with $a,u\in \BZ_q^\times$.
 Then for any norm $1$ element $\alpha$ in $O_K^1\cap (1+qO_K)$, $a'+\theta u':=({a}+\theta u)\alpha$ also satisfies the equation 
$a'^2-\theta^2u'^2=-1$ with $a',u'\in \BZ_q^\times$. On the other hand, the norm map $1+qO_K\ra 1+q\BZ_q$ restricted to $1+q\BZ_q$ is surjective, and then so is $O_K^1\cap (1+qO_K)\ra (1+qO_K)/(1+q\BZ_q)$. Since the map 
{$$({a}+\theta u)O_K^1\cap (1+qO_K)\ra (1+qO_K)/(1+q\BZ_q)$$ is surjective}
and $\lambda|_{1+qO_K}$ is a primitive character on $(1+qO_K)/(1+q\BZ_q)\simeq \BZ/q^{m-1}\BZ$, the fact follows.

\item[(ii)]

Take $u$ and $\theta$ such that $\lambda({\rm pr}(a_0+u\theta))\neq 1$. Such an $u$ exists by the analysis in (i).

{We rely on the following:}
\begin{fact}\label{f2}
  Let $\zeta$ be a primitive $k$-th root of unity with $k\neq 2$. Then {${\rm N}_{\BQ(\zeta)^+/\BQ}(\pm2+\zeta+\ov{\zeta})$} {divides either any odd prime factor of $k$ or divides $2$ if $k$ is a power of $2$.}
  \end{fact}
  \begin{proof}[Proof of Fact \ref{f2}]Note that
  \[\begin{aligned}
    x^{\varphi(k)} \prod_{s\in(\BZ/k\BZ)^\times}(x+x^{-1}+\zeta^s+\ov{\zeta}^s)
   =&\prod_{s\in(\BZ/k\BZ)^\times}(x^2+1+x\zeta^s+x\ov{\zeta}^s)
   =&\Phi_{k}(-x)^2
  \end{aligned},\]where $\varphi$ is the Euler function and $\Phi_k$ the $k$-th cyclotomic polynomial.
  Therefore,  
  \[\begin{aligned}
   & {\rm N}_{\BQ(\zeta)^+/\BQ}(\pm 2+\zeta+\ov{\zeta})^2
    =&\Phi_k(\mp 1)^2(\pm 1)^{\varphi(k)}.
  \end{aligned}\]
  
Recall that
$\Phi_k(x)|\frac{x^k-1}{x^d-1}$ for any proper divisor $d|k$.
    Thus if $r|k$ is an odd prime, we have 
    {$$\Phi_k(\pm 1)\ |\ \frac{(\pm 1)^{k}-1}{(\pm 1)^{k/r}-1}\ |\ r.$$} 
    If $k=2^{s}$ with $s\geq 2$, then $\Phi_k(\pm 1)
     |\ \frac{(\pm 1)^{k}-1}{(\pm 1)^{k/2}-1}\ |\ 2$. 
  \end{proof}

{The assertion is a consequence of Fact \ref{f2}: we apply it for $k$ being the order of $\epsilon(\pi)\lambda(a_0+u\theta)$. Note that $q|k$ and so the proof concludes.}

\end{itemize}
\end{proof}

    \begin{remark}
      An elementary argument shows the existence of an embedding $K\hookrightarrow M_{2}(\BQ_q)$ such that the toric period associated to $(\pi,1)$ is non-zero. In contrast, the above result gives the existence with respect to which the toric period is an $\ell$-adic unit for a given prime $\ell$. The latter is crucial for our application.
    \end{remark}
    
        \subsubsection{Strategy}
    The Kirillov model is integral to our method. 
    
    We first obtain an expression for the matrix coefficients of the newform under toric action in terms of a linear combination of twist epsilon factors, without assuming supercuspidality (cf.~Theorem~\ref{co}). This relies on the action of Atkin--Lehner operator on twists of the newform (cf.~Proposition \ref{all}). 
    {In our supercuspidal case, the epsilon factor of a $\GL_2(\BQ_q)$-representation equals that of the associated character of $K^\times$ (cf.~Lemma \ref{rt2}).} 
We explicitly calculate the latter using an approach of Murase and Sugano \cite{MS}.
This transforms the toric period into a twisted sum of a Jacobi sum and values of $\lambda$ (cf.~Lemma \ref{js}). The former turns out to be elementary (cf.~Lemma \ref{bF'}), leading to an expression for the toric period in terms of values of $\lambda$ (cf.~Proposition \ref{mm}).

    We begin with preliminaries on the Kirillov model in ~\S\ref{ki}. Then \S\ref{epst} presents the  connection with epsilon factors, and \S\ref{expt} of the latter with values of $\lambda$.
     The proof of Theorem~\ref{ml} concludes in {\S\ref{mlc}.}

    \begin{remark}\noindent
    \begin{itemize}
        \item[(i)]   Our approach perhaps applies to self-dual pairs $(\pi,\chi)$ over $K$ with $\cond{\chi}\leq \cond{\pi}$.   
         \item[(ii)] For $m$ even, one may also resort to the compact induction model (cf.~\cite{HSY17,HN}).
        \end{itemize}
    \end{remark}

\subsection{Preliminaries on Kirillov model}
\label{ki}
\subsubsection{The model}
Let $\pi$ be an irreducible admissible representation of $\PGL_2(\BQ_q)$. Let $\psi$ be a non-trivial character of $\BQ_q$. 

{Recall that Kirillov model} $\CK(\pi,\psi)$ of $\pi$ with respect to $\psi$ is a model of $\pi$ in the space of {locally constant functions} such that upon restriction to the upper triangle subgroup $B(\BQ_q)$ the action is given by \[\pi\begin{pmatrix}
  a&b\\&d
\end{pmatrix}f(x)=\psi(bx/d)f(ax/d),\quad a,d, x\in \BQ_q^\times, b\in \BQ_q.\]
{The space of Schwartz functions is a  finite codimensional subspace of $\CK(\pi,\psi)$,  and it equals $\CK(\pi,\psi)$ if and only if $\pi$ is supercuspidal (cf.~\cite[\S2]{Jacquet_Langlands:GLtwo}).}

\subsubsection{Newforms and twists}{Denote by $q^n$ the conductor of $\pi$.}
Recall that the space of newforms in $\pi$ is the subspace fixed by $U_0(q^n)$. 

{In the following we choose $\psi$ to be unramified and identify $\pi$ with its Kirillov model $\CK(\pi,\psi)$,}
{and assume that $n\geq 2$.}

We have the following explicit description of newforms (cf.~\cite[p.~23]{Schmidt:newform}).
\begin{lem}\label{nl}
Suppose that $n\geq 2$. Then the space of newform is $\BC\cdot 1_{\BZ_q^\times}$.
\end{lem}
\begin{cor}\label{cor} Suppose that $n\geq 2$.
 For $k\geq 1$, we have
    \[\sum_{a\in \BZ/q^k\BZ}\pi\begin{pmatrix}
      1&\frac{a}{q^k}\\
      &1
    \end{pmatrix}f=0.\]
    In particular, the action of $U_q=\sum_{a\in \BZ/q\BZ}\begin{pmatrix}
      q&a\\&1
    \end{pmatrix}$ on newforms is with eigenvalue $0$. 

\end{cor}
\begin{proof}

Note that the unipotent action does not change the support of a newform.
 For $x\in \BZ_q^\times$, we have
  \[\begin{aligned}
    \sum_{a\in \BZ/q^k\BZ}\pi\begin{pmatrix}
    1&\frac{a}{q^k}\\
    &1
  \end{pmatrix}f(x)=&\sum_{a\in \BZ/q^k\BZ}\psi\left(\frac{ax}{q^k}\right)f(x)\\
  =&\sum_{a\in \BZ/q^k\BZ}\psi\left(\frac{a}{q^k}\right)f(x).\\
  \end{aligned}\]
Here the last equality follows by taking $f=1_{\BZ_q^\times}$, which also implies that 
 $\sum_{a\in \BZ/q^k\BZ}\pi\begin{pmatrix}
  1&\frac{a}{q^k}\\
  &1
\end{pmatrix}f(x)=0$ since the sum of the $q^k$-th root of unity is zero.

  \end{proof}
 
 In the following,
 we consider action of the Atkin-Lehner operator $w_{\pi}=\begin{pmatrix}
    &1 \\ q^n &
\end{pmatrix}$ on vectors of the form $\chi 1_{\BZ_q^\times}$ for $\chi$ a character of $\BZ_q^\times$.

 For {$f\in \pi$}, and $\chi$ a character of $\BZ_q^\times$, write 
 \[\wh{f}(\chi,t)=\sum_{n\in \BZ}\wh{f}_n(\chi)t^n,\]where $\wh{f}_n(\chi)=\int_{\BZ_q^\times}\chi(x)f(q^nx)dx$. The action of $w:=\begin{pmatrix}
  &1\\-1&
\end{pmatrix}$ on $f$ is given  by 
\[\wh{\pi(w)f}(\chi,t)=C(\chi,t)\wh{f}(\chi^{-1},t^{-1})\]for
\[C(\chi,q^{s-1/2})=\frac{L(1-s,{\pi}\otimes\chi)\epsilon(s,\pi\otimes\chi^{-1},\psi)}{L(s,\pi\otimes\chi^{-1})}.\]
Here $\chi$ is viewed as a character of $\BQ_q^\times$ by $\chi(q)=1$, $L(s,\pi\otimes\chi^{-1})$ and 
$\epsilon(s,\pi\otimes\chi^{-1},\psi)$ are the $L$- and epsilon-factors associated to $\pi\otimes\chi^{-1}$ respectively. (cf.~\cite{Jacquet_Langlands:GLtwo} lines above Corollary 2.19.)

\begin{prop}\label{all} {Let $\chi$ be a character of $\Q_q^\times$ as above and $f_\chi=\chi1_{\BZ_q^\times}\in \pi$.} 
{Suppose that $\cond(\pi\otimes\chi)=\cond(\pi)=q^n$, $n\geq 2$.}
Then we have
  \[\pi(w)f_\chi=\epsilon(1/2,\pi\otimes\chi^{-1}, \psi)\chi^{-1}(q^{n}\cdot) 1_{q^{-n}\BZ_q^\times}.\]
  In particular, $\pi(w_{\pi})f_\chi=\pi\left(\begin{pmatrix}
    1&\\ &-q^{n}
  \end{pmatrix}w\right)f=\epsilon(1/2, \pi\otimes\chi^{-1},\psi)\chi(-1)f_{\chi^{-1}}$. 
  \end{prop}
  \begin{proof}
    To determine $\pi(w)f_\chi$, it suffices to consider $\wh{\pi(w)f_\chi}(\nu,t)$ for all characters $\nu$ of $\BZ_q^\times$.
    
   Since {$\cond{(\pi\otimes\chi)}=\cond{(\pi\otimes\chi^{-1})}\geq q^2$}, the associated local $L$-factor is just $1$. 
  Note that $C(\chi,q^{s-1/2})=\epsilon(s,\pi\otimes\chi^{-1},\psi)=\epsilon(1/2,\pi\otimes\chi^{-1},\psi)q^{-n(s-1/2)}$, where we utilise the hypothesis that $\cond{(\pi\otimes\chi^{-1})}=\cond{(\pi)}={q^n}$.

{It follows that} \[\wh{\pi(w)f_\chi}(\nu,t)=\epsilon(1/2,\pi\otimes\chi^{-1}, \psi)t^{-n}\begin{cases}
   1, &\nu=\chi,\\
    0,& \text{otherwise}.
  \end{cases}\]
  Therefore
  \[
  \wh{\pi(w)f}_{\chi,m}(\nu)=\begin{cases}
    \epsilon(1/2, \pi\otimes\chi^{-1}, \psi),\quad& m=-n,  \nu= \chi,\\ 
0,&\text{otherwise}.
  \end{cases}\]
  The proof concludes.
  \end{proof}

{Since} $\pi$ has trivial central character, we denote $\epsilon(1/2,\pi,\psi)$ simply by $\epsilon(\pi)\in \{\pm 1\}$.

\subsection{Toric periods and epsilon factors}\label{epst}
Let the setting be as in \S\ref{ss:lt-set}, {except that we allow $\pi$ to be any {unitary} irreducible admissible representation of $\PGL_2(\BQ_q)$ with {exponential} conductor $n=2m$, $m\geq 2$.}

This subsection links the toric period $\gamma_{\theta.u}$ with twists of epsilon factors. 
\subsubsection{Main result}

\begin{thm}\label{co}  Let $f\in \pi$ be a newform.  We have
\[\begin{aligned}
&\sum_{x\in O_{K}^\times/O_{K,q^{m}}^\times}(\pi(\iota(x))f,f)=\frac{q-q\eta(-1)\epsilon(\pi)\epsilon(\pi\otimes\eta)}{q-1}(f,f)\\
 &+q^{\lfloor\frac{m}{2}\rfloor} \sum_{\substack{\chi\in \wh{(\BZ/q^m\BZ)^\times }\\ primitive}} \frac{G(\chi,\psi)^2}{\varphi(q^m)^2}\chi^{-1}(u^2\theta^2)\epsilon(1/2,\pi\otimes\chi,\psi)(f,f)\\
    +&\begin{cases}
      \displaystyle 2q^{(m-1)/2}\sum_{v\in 1-q^{m-1}(\BZ/q\BZ)^{\times 2}}  \sum_{\substack{\chi\in \wh{(\BZ/q^m\BZ)^\times }\\ primitive}} \frac{G(\chi,\psi)^2}{\varphi(q^m)^2}\chi^{-1}(u^2\theta^2v)\epsilon(1/2,\pi\otimes\chi,\psi)(f,f), &\quad \text{$m$ odd},\\
      0,&\quad \text{$m$ even}.
    \end{cases}
  \end{aligned}\]
Here $v$ is viewed as an element in $\BZ/q^{m}\BZ$,
$\eta$ the non-trivial  quadratic character\footnote{viewed as character of $\BQ_q^\times$ via $\eta(q)=1$} of $\BZ_q^\times$, and 
$G(\eta,\psi)=\sum_{a\in \BF_q^\times}\psi\left(\frac{a}{q}\right)\eta(a)$, $G(\chi,\psi)=\sum_{a\in (\BZ/q^m\BZ)^\times}\psi\left(\frac{a}{q^m}\right)\chi(a)$ are the Gauss sum.
\end{thm}

The result is a simple consequence of the following proposition, whose formulation relies on the fact that ${O_{K}^\times/O_{K,q^m}^\times}$ is represented by
\[\{a+\theta\ |\ a\in \BZ/q^m\BZ\}\sqcup \{1+b\theta\ |\ b\in q\BZ/q^m\BZ\}.\]
{Here we view an element of $\BZ/q^k\BZ$ as an element in $\BZ_q$ by choosing a lift, a convention often followed.}
\begin{prop}\label{mp}\ We have the following.
\begin{itemize}
\item[(i)]  \[\sum_{\{a\in \BZ/q^m\BZ\ |\   q\nmid a\}}(\pi(\iota(a+\theta))f,f)=0.\]
\item[(ii)]  For $1\leq t\leq m$,
  \[\begin{aligned}
    &\sum_{a\in q^{t}(\BZ/q^{m-t}\BZ)^\times}(\pi(\iota(a+\theta))f,f)\\
    =&\begin{cases}
    \displaystyle  0,&\quad m-2t>1,\\
    \displaystyle 2q^{(m-1)/2}\sum_{v\in 1-q^{m-1}(\BZ/q\BZ)^{\times 2}}  \sum_{\substack{\chi\in \wh{(\BZ/q^m\BZ)^\times }\\ primitive}} \frac{G(\chi,\psi)^2}{\varphi(q^m)^2}\chi^{-1}(u^2\theta^2v)\epsilon(1/2,\pi\otimes\chi,\psi)(f,f),&\quad m=2t+1,\\
    \displaystyle \varphi(q^{m-t})\sum_{\substack{\chi\in \wh{(\BZ/q^m\BZ)^\times }\\ primitive}} \frac{G(\chi,\psi)^2}{\varphi(q^m)^2}\chi^{-1}(u^2\theta^2)\epsilon(1/2,\pi\otimes\chi,\psi)(f,f),&\quad m-2t\leq 0.
  \end{cases}
\end{aligned}\]
Here we regard $1-q^{m-1}(\BZ/q\BZ)^{\times 2}\subset (\BZ/q^m\BZ)^\times$.
\item[(iii)]  For $1\leq s\leq m$, 
  \[\sum_{\{b\in q\BZ/q^{m}\BZ\ |\  q^{s}\parallel b\}}(\pi(\iota(1+b\theta))f,f)=(f,f)\begin{cases}\
    0,&\quad \text{ if $m-s>1$},\\
    \frac{1-q\eta(-1)\epsilon(\pi)\epsilon(\pi\otimes\eta)}{q-1},&\quad \text{ if $m-s=1$},\\
  f,&\quad \text{ if $m-s=0$}.\\
  \end{cases}\]
\end{itemize}
\end{prop}

Our approach is based on the Kirillov model, which leads to an explicit formula for matrix coefficients of a newform under the action of $\iota(K^\times)$ in terms of twist epsilon factors. Proposition \ref{mp}  is a consequence of Propositions \ref{a}, \ref{2}, \ref{3}, \ref{4} and \ref{5} below.

Throughout this subsection, identify $\pi$ with it's Kirillov model with respect to an unramified character $\psi$ of $\BQ_q$ (cf.~\S\ref{ki}). Then we may choose $f$ to be $1_{\BZ_q^\times}$ in the Kirillov model since $n\geq 2$ (cf.~Lemma~\ref{nl}). 

By Bruhat decomposition \[\GL_2(\BQ_q)=B(\BQ_q)\begin{pmatrix}
  & q^{-m} \\ q^m&
\end{pmatrix}N(\BQ_q)\sqcup B(\BQ_q),\] 
where $B$ is the subgroup of upper triangle matrices and $N\subset B$ the unipotent subgroup. 
In view of the Bruhat decomposition and explicit action of $B(\BQ_q)$ on the Kirillov model, it suffices to consider matrix coefficients of 
$w_{q^m}=\begin{pmatrix}
  & q^{-m} \\ q^m&
\end{pmatrix}$ on the 
{twist newforms} $\chi \BZ_q^\times$ in $\pi$. An analysis of the latter gives rise to twist epsilon factors (cf.~Proposition \ref{all}).

{As for the explicit matrix coefficients, we separate the analysis into three cases, which correspond to the sub cases of Proposition \ref{mp}.
We first consider matrix coefficients under the action of $\iota(a+\theta)$ with $a\in \BZ/q^m\BZ$. 
}
\subsubsection{Case I: {$a+\theta$ with} $q\nmid a$}\label{s21}
We begin with a preliminary. 
\begin{lem}\label{b}
 If $a-u\theta^2$ is a unit, then 
  \[\begin{aligned}
   & (\pi(\iota(a+\theta))f,f)
   = &\left(\pi\begin{pmatrix}
    1&\frac{1}{q^m}\\&1
  \end{pmatrix}\begin{pmatrix}
    & q^{-m} \\ q^m&
  \end{pmatrix}\begin{pmatrix}
  1&\frac{-(a+\theta^2u)(a-u\theta^2)}{q^{m}(a^2-\theta^2)} \\ &1
  \end{pmatrix}f,f\right).
  \end{aligned}\] 
  
  \end{lem}
  \begin{proof}
  Consider the Bruhat decomposition 
\[\begin{aligned}
  &\iota(a+\theta)   =&\begin{pmatrix}
    \frac{a-u\theta^2}{\theta^2}&\\&1
  \end{pmatrix}\begin{pmatrix}
    1&\frac{1}{q^m}\\&1
  \end{pmatrix}\begin{pmatrix}
    &q^{-m}\\
    q^m&
  \end{pmatrix}\begin{pmatrix}
  1&\frac{-(a+\theta^2u)(a-u\theta^2)}{q^{m}(a^2-\theta^2)} \\ &1
  \end{pmatrix}\begin{pmatrix}
   \theta^2&\\ &\frac{-(a^2-\theta^2)}{a-u\theta^2}
  \end{pmatrix}.
\end{aligned}\]
  Since $a-u\theta^2$ is a unit, note that $\begin{pmatrix}
      \frac{a-u\theta^2}{\theta^2}&\\&1
    \end{pmatrix}$ and $\begin{pmatrix}
      \theta^2&\\ &\frac{-(a^2-\theta^2)}{a-u\theta^2}
     \end{pmatrix}$ are in $B(\BZ_q)$. As $f$ is $B(\BZ_q)$-invariant, the lemma follows.
  
  \end{proof}
We separate the analysis  into the following sub cases. 
  \bigskip 

\underline{The case $q|a+u\theta^2$}\

  \begin{prop}\label{a}Let $1\leq r\leq m$, and $C_r=\{a\in \BZ/q^m\BZ\ |\ q^r\parallel a+u\theta^2\}$. 
  Then 
  $$
  \sum_{a\in C_r}\pi\begin{pmatrix}
      1&\frac{-(a+\theta^2u)(a-u\theta^2)}{q^{m}(a^2-\theta^2)} \\ &1
      \end{pmatrix}f =
    \begin{cases}\
    0,&\quad \text{ if $m-r>1$},\\
  -f,&\quad \text{ if $m-r=1$},\\
  f,&\quad \text{ if $m-r=0$}.\\
  \end{cases}  
      $$
  In particular,
  \[\sum_{\{a\in \BZ/q^m\BZ\ |\   q|a+u\theta^2\}}(\pi(\iota(a+\theta))f,f)=0.\]
  \end{prop}
  \begin{proof}

  For $a\in \BZ/q^{m}\BZ$, put $r=v_q(a+u\theta^2)$ with $0\leq r\leq m$. Write $a+u\theta^2=q^r v$.

Note that the unipotent group action does not change the support of newform 
$f=1_{\BZ_q^\times}$. 
We have
\[\sum_{a\in C_r}\pi\begin{pmatrix}
  1&\frac{-(a+\theta^2u)(a-u\theta^2)}{q^{m}(a^2-\theta^2)} \\ &1
  \end{pmatrix}f(x)=\sum_{v\in(\BZ/q^{m-r}\BZ)^\times}\psi\left(\frac{-xv(a-u\theta^2)}{q^{m-r}(a^2-\theta^2)}\right)f(x),\quad x\in \BZ_q^\times.\]

  As $a$ runs over $C_r$, $v=(a+u\theta^2)/q^r$ runs over $(\BZ/q^{m-r}\BZ)^\times$.
 For a fixed $x\in \BZ_q^\times$, consider 
  \[\delta: (\BZ/q^{m-r}\BZ)^\times\ra (\BZ/q^{m-r}\BZ)^\times,\quad v\mapsto \frac{-v(a-u\theta^2)}{(a^2-\theta^2)}=\frac{-v(q^{r}v-2u\theta^2)}{q^rv(q^{r}v-2u\theta^2)+u^2\theta^4-\theta^2}.\]
{Note that $\delta$ is an isomorphism since $a-u
  \theta^2$ and $a^2-\theta^2$ are units.}

  Thus 
 \[\begin{aligned}
  \sum_{a\in C_r}\pi\begin{pmatrix}
    1&\frac{-(a+\theta^2u)(a-u\theta^2)}{q^{m}(a^2-\theta^2)} \\ &1
    \end{pmatrix}f(x)
  =&\left(\sum_{v\in (\BZ/q^{m-r}\BZ)^\times}\psi\left(\frac{\delta(v)x}{q^{m-r}}\right)\right)f(x)\quad (w=\delta(v))\\
  =&\left(\sum_{w\in (\BZ/q^{m-r}\BZ)^\times}\psi\left(\frac{w}{q^{m-r}}\right)\right)f(x).
 \end{aligned}.\]

 Since $\psi$ is an unramified character of $\BQ_q$, $\psi(\frac{w}{q^{m-r}})$ runs over all $q^{m-r}$-th primitive roots of unity, concluding the proof. `In particular' part follows from Lemma \ref{b}.
  \end{proof}

\underline{The case $q|a-u\theta^2$}\
\begin{prop}\label{2}
We have 
  \[\sum_{\{a\in \BZ/q^m\BZ\ |\   q|a-u\theta^2\}}(\pi(\iota(a+\theta))f,f)=0.\]
\end{prop}
\begin{proof}
  Note that $(a+\theta)^{-1}=\frac{-1}{a^2-\theta^2}(-a+\theta)$. So we have 
 \[\begin{aligned}
     \sum_{\{a\in \BZ/q^m\BZ\ |\   q|a-u\theta^2\}}(\pi(\iota(a+\theta))f,f)
    =&  \sum_{\{a\in \BZ/q^m\BZ\ |\   q|a-u\theta^2\}}(f,\pi(\iota(a+\theta)^{-1})f)\\
    =& \sum_{\{a\in \BZ/q^m\BZ\ |\  q|  a+u\theta^2\}}(f,\pi(\iota({a}+\theta)){f}).\quad  
    \end{aligned}\]
   The latter vanishes by Proposition \ref{a}.
\end{proof}

\bigskip

\underline{The case $q\nmid a(a^2-u^2\theta^4)$}\ 

\bigskip

For this remaining case, the main result is Proposition~\ref{3} below. 

In view of Lemma \ref{b} we consider the map
 \[\kappa: \{a\in \BZ/q^{m}\BZ\ \Big|\ a\nequiv \pm u\theta^2\pmod{q}\}\ra (\BZ/q^{m}\BZ)^\times,\quad  a\mapsto \frac{-(a+\theta^2u)(a-u\theta^2)}{a^2-\theta^2}=\frac{(u\theta^2)^2-\theta^2}{a^2-\theta^2}-1.\] 
For $c\in \{1,\cdots,q-1\}$ not congruent to $ \pm u\theta^2$ modulo ${q}$, its restriction to
  \[S_c=\{a\in \BZ/q^{m}\BZ\ \Big|\ a\equiv c\pmod{q}\}\]
  is given by the following.
  
  \begin{fact}\label{ff}
    $\kappa(S_c)$ is a fiber of the natural projection map $(\BZ/q^m\BZ)^\times\ra (\BZ/q\BZ)^\times$.
  
  \end{fact}
  \begin{proof}
  Note that $\kappa(a)=\kappa(a')$ if and only if $a^2\equiv a'^2\pmod{q^m}$, thus $\kappa|_{S_c}$ is injective. Moreover, the image $\kappa(S_c)$ is constant modulo $q$. Therefore $\kappa(S_c)$ is the fiber of the projection map $(\BZ/q^m\BZ)^\times\ra (\BZ/q\BZ)^\times$ by comparing the cardinality.
  \end{proof}

The following fact will also be useful.
 \begin{fact}\label{zf}
  Let $k\geq 1$ be an integer and $\zeta$ a $q^k$-th primitive root of unity, and $s\leq k$ an integer. Then {for $a\in \BZ/q^{k}\BZ$,}
  \[\sum_{b\in \BZ/q^{k-s}\BZ}\zeta^{a+q^sb}=\zeta^{a}\sum_{ b\in \BZ/q^{k-s}\BZ}(\zeta^{q^s})^b=\begin{cases}0,\quad &\text{$k>s$},\\
    \zeta^{a},\quad &\text{$k=s$}.\\
  \end{cases}\]
  \end{fact}

  \begin{prop}\label{3} For each $c\in \{1,\cdots,q-1\}$ not congruent to $ \pm u\theta^2$ modulo ${q}$, we have
    \[\sum_{a\in S_c}(\pi(\iota(a+\theta))f,f)=0\]
   
  \end{prop}
  \begin{proof} In view of Lemma~\ref{b}, it suffices to consider 
\[\begin{aligned}
   \sum_{a\in S_c}\pi\begin{pmatrix}
    1&\frac{-(a+\theta^2u)(a-u\theta^2)}{q^{m}(a^2-\theta^2)} \\ &1
    \end{pmatrix}f(x) 
    =&\sum_{a\in S_c}\pi\begin{pmatrix}
      1&\frac{\kappa(a)}{q^{m}} \\ &1
      \end{pmatrix}f(x)\\
      =&\sum_{a\in S_c}\psi\left(\frac{\kappa(a)x}{q^m}\right)f(x), 
  \end{aligned}\]
  where $x\in \BZ_q^\times$.   
  The latter vanishes\footnote{by taking $m=k$, $s=1$ and $\zeta=\psi(x/q^m)$} by Fact \ref{zf} since  $\kappa(S_c)$ is a fiber of the projection map $(\BZ/q^m\BZ)^\times\ra (\BZ/q\BZ)^\times$ {in view of Fact \ref{ff}.}

  Hence, the assertion follows from Lemma~\ref{b}.
  \end{proof}

  \subsubsection{Case II: {$a+\theta$ with} $q\mid a$}

\begin{prop}\label{4}
  For $1\leq t\leq m$, we have
  \[\begin{aligned}
    &\sum_{a\in q^{t}(\BZ/q^{m-t}\BZ)^\times}(\pi(\iota(a+\theta))f,f)\\
    =&\begin{cases}
    \displaystyle  0,&\quad m-2t>1,\\
    \displaystyle 2q^{(m-1)/2}\sum_{v\in 1-q^{m-1}(\BZ/q\BZ)^{\times 2}}  \sum_{\substack{\chi\in \wh{(\BZ/q^m\BZ)^\times }\\ primitive}} \frac{G(\chi,\psi)^2}{\varphi(q^m)^2}\chi^{-1}(u^2\theta^2v)\epsilon(1/2,\pi\otimes\chi,\psi)(f,f),&\quad m=2t+1,\\
    \displaystyle \varphi(q^{m-t})\sum_{\substack{\chi\in \wh{(\BZ/q^m\BZ)^\times }\\ primitive}} \frac{G(\chi,\psi)^2}{\varphi(q^m)^2}\chi^{-1}(u^2\theta^2)\epsilon(1/2,\pi\otimes\chi,\psi)(f,f),&\quad m-2t\leq 0.
  \end{cases}
\end{aligned}\]
Here we regard $1-q^{m-1}(\BZ/q\BZ)^{\times 2}\subset (\BZ/q^m\BZ)^\times$.
  \end{prop}

We begin with some preliminaries. 

First, note that Lemma \ref{b} still applies.
Put $t=v_q(a)\in \{1,\cdots, m\}$ and consider the map
\begin{equation}\label{tfb}
\kappa_t:q^{t}(\BZ/q^{m-t}\BZ)^\times\ra (\BZ/q^m\BZ)^\times,\quad a\mapsto \frac{(u\theta^2)^2-\theta^2}{a^2-\theta^2}-1,
\end{equation}
where we regard $q^{t}(\BZ/q^{m-t}\BZ)^\times \subset (\BZ/q^{m}\BZ)$. 
\begin{fact}\label{ffc}\
\begin{itemize}
  \item [(i)] For $t\in \{1,\cdots, m\}$ with  $m\leq 2t$, we have $\kappa_t=-\theta^2u^2$.
  \item [(ii)] For $t\in \{1,\cdots, m\}$ with $m\geq 2t+1$ and $r\in \{1,\cdots, q-1\}$, put \[S_{t,r}=q^t r(1+q\BZ_q)/(1+q^{m-t}\BZ_q).\]  
  \begin{itemize}
    \item[(a)] If $m-2t=1$, then $\kappa_t|_{S_{t,r}}=\frac{(u\theta^2)^2-\theta^2}{q^{2t}r^2-\theta^2}-1$.
    \item [(b)] If $m>2t+1$, then \[\kappa_{t}({S_{t,r}})=\{y\in (\BZ/q^{m}\BZ)^\times\ |\ y\equiv\frac{(u\theta^2)^2-\theta^2}{q^{2t}r^2-\theta^2}-1\pmod{ q^{2t+1}}\},\] which is a fiber of the projection map $(\BZ/q^{m}\BZ)^\times\ra (\BZ/q^{2t+1}\BZ)^\times.$ Furthermore, the function 
    $\kappa_{t}|_{S_{t,r}}$ is exactly $q^{t}$ to $1$. 
  \end{itemize}
\end{itemize}
\end{fact}
\begin{proof}
The assertions in parts {(i)} and (a) are apparent.
The following considers (b).

 Note that $\kappa_t$ can be written as a composite 
\[\kappa_t:q^{t}(\BZ/q^{m-t}\BZ)^\times{\xrightarrow{j}} (\BZ/q^{m+t}\BZ)^\times\ra(\BZ/q^m\BZ)^\times,\]where the first map $j$ is $a\mapsto \frac{(u\theta^2)^2-\theta^2}{a^2-\theta^2}-1$ and the second natural quotient.
It is enough to show that $j$ restricted to $S_{t,r}$ is injective and its image is  \[\left\{y\in (\BZ/q^{m+t}\BZ)^\times\ |\ y\equiv\frac{(u\theta^2)^2-\theta^2}{q^{2t}r^2-\theta^2}-1\pmod{ q^{2t+1}}\right\},\] 
which is a fiber of the projective map $(\BZ/q^{m+t}\BZ)^\times\ra (\BZ/q^{2t+1}\BZ)^\times.$

Now we prove the claim. The image of $j$ modulo $q^{2t+1}$ is the constant $\frac{(u\theta^2)^2-\theta^2}{q^{2t}r^2-\theta^2}-1$. Note that if $a,a'\in S_{t,r}$ have the same image then $a^2=a'^2\pmod{q^{m+t}}$ and so $a=a'\in S_{t,r}$. 
Hence $j$ restricted to $S_{t,r}$ is injective, and the claim follows by comparing the cardinality. 

\end{proof}

\begin{lem}\label{ss'}
  For $1\leq t\leq m$, we have
  \[ \sum_{a\in q^{t}(\BZ/q^{m-t}\BZ)^\times}\pi\begin{pmatrix}
    1&\frac{-(a+\theta^2u)(a-u\theta^2)}{q^{m}(a^2-\theta^2)} \\ &1
    \end{pmatrix}f=\begin{cases}
    0,&\quad m-2t>1,\\
\displaystyle 2q^{(m-1)/2}\sum_{v\in 1-q^{m-1}(\BZ/q\BZ)^{\times 2}}\pi\begin{pmatrix}
  1& \frac{-u^2\theta^2v}{q^{m}}\\ & 1
\end{pmatrix} f  ,&\quad m-2t=1,\\
   \varphi(q^{m-t})\pi \begin{pmatrix}
      1&\frac{-u^2\theta^2}{q^{m}} \\ &1
      \end{pmatrix}f ,&\quad m-2t\leq 0.\\  
    \end{cases}\]
\end{lem}

\begin{proof}
Fix $t\geq 1$. 
As before, we may take $f=1_{\BZ_q^\times}$. 

For $x\in \BZ_q^\times$ and $\kappa_t$ as in \eqref{tfb}, 
\[\begin{aligned}
  \sum_{a\in q^{t}(\BZ/q^{m-t}\BZ)^\times }\pi\begin{pmatrix}
    1&\frac{-(a+\theta^2u)(a-u\theta^2)}{q^{m}(a^2-\theta^2)} \\ &1
    \end{pmatrix}f(x)
    =&\sum_{a\in q^{t}(\BZ/q^{m-t}\BZ)^\times }\psi\left(\frac{\kappa_t(a)x}{q^{m}}\right)f(x).\\
\end{aligned}\]
In view of Fact \ref{ffc} the following holds. 
\begin{itemize}
  \item [(i)]If $m-2t\leq 0$, the image of 
  \[q^{t}(\BZ/q^{m-t}\BZ)^\times\ra (\BZ/q^{m}\BZ)^\times,\quad a\mapsto \frac{-(a+\theta^2u)(a-u\theta^2)}{(a^2-\theta^2)}\] is the constant
  $-\theta^2u^2,$ and the third case follows.
  \item [(ii)] If $m-2t=1$, then $\kappa$ on each $S_{t,r}$ is the constant\footnote{To see this congruence, note that
  $\frac{1}{q^{m-1}r^2-\theta^2}\equiv -\frac{(\theta^2+q^{m-1}r^2)}{\theta^4}\pmod{q^m}$, 
where $m\geq 2$ and $q^{m-1}r^2-\theta^2\in \BZ_q^\times$, and so
\[\begin{aligned}
  &\frac{(u\theta^2)^2-\theta^2}{q^{m-1}r^2-\theta^2}-1
  \equiv&-\theta^2u^2\left(1-q^{m-1}r^2\left(\frac{1}{\theta^2}-\frac{1}{u^2\theta^4}\right)\right)\pmod{q^m}.
\end{aligned}\]
} 
  $$\frac{(u\theta^2)^2-\theta^2}{q^{m-1}r^2-\theta^2}-1\equiv -\theta^2u^2\left(1-q^{m-1}\frac{r^2(u^2\theta^2-1)}{\theta^4u^2}\right)\pmod{q^m},$$
  {where $S_{t,r}$ is as in Fact \ref{ffc}.}
  
  As $r$ varies in $\{1,\cdots q-1\}$, note that $\frac{r^2(u^2\theta^2-1)}{\theta^4u^2}$ varies over $(\BZ/q\BZ)^{\times 2}$. Thus, 
  \[\begin{aligned}
    \sum_{a\in q^{t}(\BZ/q^{m-t}\BZ)^\times }\pi\begin{pmatrix}
      1&\frac{-(a+\theta^2u)(a-u\theta^2)}{q^{m}(a^2-\theta^2)} \\ &1
      \end{pmatrix}f
      &=\sum_{r=1}^{q-1}\sum_{a\in S_{t,r}}\pi\begin{pmatrix}
        1&\frac{-(a+\theta^2u)(a-u\theta^2)}{q^{m}(a^2-\theta^2)} \\ &1
        \end{pmatrix}f\\ 
        &=\sum_{r=1}^{q-1}\# S_{t,r}\cdot \pi\begin{pmatrix}
          1&\frac{-\theta^2u^2(1-q^{m-1}\frac{r^2(u^2\theta^2-1)}{\theta^4u^2})}{q^{m}} \\ &1
          \end{pmatrix}f\\
        &=q^{(m-1)/2}\cdot \sum_{w\in (\BZ/q\BZ)^\times}\pi\begin{pmatrix}
          1&\frac{-\theta^2u^2(1-q^{m-1}w^2) }{q^{m}}\\ &1
          \end{pmatrix}f\\
        &=2q^{(m-1)/2}\cdot \sum_{v\in 1-q^{m-1}(\BZ/q\BZ)^{\times 2}}\pi\begin{pmatrix}
          1&\frac{-\theta^2u^2v }{q^{m}}\\ &1
          \end{pmatrix}f.
  \end{aligned}\]
  \item [(iii)] Suppose that $m-2t> 1$. Then the map $\kappa_t$ on each $S_{t,r}$ is $q^{t}$ to $1$ and the image $\kappa_t (S_{t,r})$ is a fiber of the projection $(\BZ/q^m\BZ)^\times\ra (\BZ/q^{2t+1}\BZ)^\times$ 
  by Fact \ref{ffc}. Thus it follows from Fact \ref{zf} 
   that 
   \[\sum_{a\in S_{t,r} }\psi\left(\frac{\kappa_t(a)x}{q^{m}}\right)f(x)=0\] for each $r$, concluding the proof of first case.
\end{itemize}

\end{proof}

  \begin{proof}[Proof of Proposition \ref{4}]
    In light of Lemmas \ref{b} and \ref{ss'}, it suffices to show: 
    for $v\in (\BZ/q^m\BZ)^\times$, 
    \[\begin{aligned}
      \left(\pi\begin{pmatrix}
        1&\frac{1}{q^m}\\&1
      \end{pmatrix}\begin{pmatrix}& q^{-m}\\ q^m &\end{pmatrix}\begin{pmatrix}
        1&\frac{-\theta^2u^2 v}{q^{m}} \\ &1
        \end{pmatrix}f,f\right)=\sum_{\substack{\chi\in \wh{(\BZ/q^m\BZ)^\times }\\ primitive}} \frac{G(\chi,\psi)^2}{\varphi(q^m)^2}\chi^{-1}(u^2\theta^2v)\epsilon(1/2,\pi\otimes\chi,\psi)(f,f).
    \end{aligned}\]
    
We have
  \[\begin{aligned}
    \pi\begin{pmatrix}
      1&\frac{-\theta^2u^2 v}{q^{m}} \\ &1
      \end{pmatrix}f
      =&\psi\left(\frac{-\theta^2u^2 v\cdot }{q^m}\right)1_{\BZ_{q^\times}}(\cdot)\\
      =&\sum_{\substack{\chi\in \wh{(\BZ/q^m\BZ)^\times }}}\frac{G(\chi^{-1},\psi)}{\varphi(q^m)}\chi(-u^2\theta^2v) f_{\chi}\\
      =&\sum_{\substack{\chi\in \wh{(\BZ/q^m\BZ)^\times }\\ primitive}}\frac{G(\chi^{-1},\psi)}{\varphi(q^m)}\chi(-u^2\theta^2v) f_{\chi}.
  \end{aligned}\]
  Here the second equality just amounts to Fourier expansion\footnote{
  For functions $f,h$ on a finite abelian group $G$, let $(\ ,\ )_G$ be the natural Hermitian pairing on $G$ given by $(f,h)_{G}=\sum_{g\in G}f(g)\ov{h}(g)$. Then we have
  \[f=\sum_{\chi\in \wh{G}}{\frac{(f,\chi)_G}{(\chi,\chi)_G}}\chi.\]}, and
  the last equality follows from the fact that the Gauss sum $$G(\chi^{-1},\psi):=\sum_{u\in (\BZ/q^m\BZ)^\times}\chi^{-1}(u)\psi(u/q^m)$$ is non-zero only for primitive $\chi \in \wh{(\BZ/q^m\BZ)^\times}$.
 Similarly, 
    \[\pi\begin{pmatrix}
      1&\frac{-1}{q^{m}} \\ &1
      \end{pmatrix}f=\sum_{\substack{\chi\in \wh{(\BZ/q^m\BZ)^\times }\\ primitive}}\frac{G(\chi^{-1},\psi)}{\varphi(q^m)}\chi(-1)  f_{\chi}.\]

      Now we have 
      \[\begin{aligned}
  &\left(\pi\begin{pmatrix}
    1&\frac{1}{q^m}\\&1
  \end{pmatrix}\begin{pmatrix}& q^{-m}\\ q^m &\end{pmatrix}\begin{pmatrix}
    1&\frac{-\theta^2u^2v}{q^{m}} \\ &1
    \end{pmatrix}f,f\right)\\
    =&\left(\pi(w_{q^m})\pi\begin{pmatrix}
      1&\frac{-\theta^2u^2v}{q^{m}} \\ &1
      \end{pmatrix}f,\pi\begin{pmatrix}
        1&\frac{-1}{q^m}\\&1
      \end{pmatrix}f\right)\\   
      =&\left(\pi(w_{q^m})\sum_{\substack{\chi\in \wh{(\BZ/q^m\BZ)^\times }\\ primitive}}\frac{G(\chi^{-1},\psi)}{\varphi(q^m)}\chi(-u^2\theta^2v) f_{\chi},\sum_{\substack{\chi\in \wh{(\BZ/q^m\BZ)^\times }\\ primitive}}\frac{G(\chi^{-1},\psi)}{\varphi(q^m)}\chi(-1)  f_{\chi}\right)\\  
      =&\sum_{\substack{\chi\in \wh{(\BZ/q^m\BZ)^\times }\\ primitive}} \frac{G(\chi,\psi)^2}{\varphi(q^m)^2}\chi^{-1}(u^2\theta^2v)\epsilon(1/2,\pi\otimes\chi,\psi)(f,f).
      \end{aligned}\]
      Here the last equality just follows from Proposition \ref{all}, {and the facts that $G(\chi^{-1},\psi)\chi^{-1}(-1)=\ov{G(\chi,\psi)}$ and $(\ ,\ )$ is a Hermitian pairing on $\pi$.}
  \end{proof}

\subsubsection{Case III: $1+b\theta$} This subsection considers $\iota(1+b\theta)$ for $b\in q\BZ/q^m\BZ$. 

The main result:

\begin{prop}\label{5}
  For $1\leq s\leq m$, we have
  \[\sum_{\{b\in q\BZ/q^{m}\BZ\ |\  q^{s}\parallel b\}}(\pi(\iota(1+b\theta))f,f)=(f,f)\begin{cases}\
    0,&\quad \text{ if $m-s>1$},\\
    \frac{1-q\eta(-1)\epsilon(\pi)\epsilon(\pi\otimes\eta)}{q-1},&\quad \text{ if $m-s=1$},\\
  f,&\quad \text{ if $m-s=0$}.\\
  \end{cases}\]

  \end{prop}
 We begin with some preliminaries. 
\begin{lem}\label{l29}
Let $b\in q\BZ/q^m\BZ$. Write $b=q^sw$ with $1\leq s\leq m$ and $w$ in $(\BZ/q^{m-s}\BZ)^\times$. Then 
\[(\pi(\iota(1+b\theta))f,f)=\epsilon(\pi)\left(\pi\begin{pmatrix}
 1 &\frac{1}{q^{m-s}}\\&1
\end{pmatrix}\begin{pmatrix} & q^{-m}\\ q^m&\end{pmatrix}\begin{pmatrix}
 1 & \frac{\theta^2 w^2(1-u^2\theta^2)}{q^{m-s}(1-b^2\theta^2)}\\ &1
\end{pmatrix}f,f\right).\]
\end{lem}
\begin{proof}

Consider the Bruhat decomposition 
\[\iota(1+b\theta)w_{q^m}=\begin{pmatrix}
      \frac{w(1-u^2\theta^2)}{(1+b\theta^2 u)} & \\ & 1
    \end{pmatrix}\begin{pmatrix}
     1 &\frac{1}{q^{m-s}}\\&1
    \end{pmatrix}\begin{pmatrix} & q^{-m}\\ q^m&\end{pmatrix}\begin{pmatrix}
     1 & \frac{\theta^2 w^2(1-u^2\theta^2)}{q^{m-s}(1-b^2\theta^2)}\\ &1
    \end{pmatrix}\begin{pmatrix}
      1+b\theta^2 u& \\ & \frac{1-b^2\theta^2}{w(1-u^2\theta^2)}
  \end{pmatrix}.\]
Note that $f$ is fixed by {$B(\BZ_q)$} and 
 $\pi(w_{q^m})f=\epsilon(\pi)f$, concluding the proof.
\end{proof}

In view of Lemma \ref{l29} we are led to the following.

\begin{lem}\label{l210}For $1\leq s\leq m$, we have
\[\sum_{\{b\in q\BZ/q^{m}\BZ\ |\  q^{s}\parallel b\}}\pi\begin{pmatrix}
  1 & \frac{\theta^2 w^2(1-u^2\theta^2)}{q^{m-s}(1-b^2\theta^2)}\\ &1
 \end{pmatrix}f=\begin{cases}\
  0,&\quad \text{ if $m-s>1$},\\
  -f-\eta(-1)G(\eta,\psi)f_{\eta},&\quad \text{ if $m-s=1$},\\
f,&\quad \text{ if $m-s=0$}.\\
\end{cases}\] Here $w=b/q^s$ with $w\in (\BZ/q^{m-s})^\times$, $\eta$ is the non-trivial  quadratic character of $\BZ_q^\times$, 
and $f_\eta=\eta1_{\BZ_{q^\times}}$ in the Kirillov model of $\pi$.
\end{lem}
\begin{proof}
For $q^s\parallel b$, write $b=q^s w$ with $w\in (\BZ/q^{m-s}\BZ)^\times$.

We have
\[\pi\begin{pmatrix}
  1 & \frac{\theta^2 w^2(1-u^2\theta^2)}{q^{m-s}(1-b^2\theta^2)}\\ &1
\end{pmatrix}f(x)=\psi\left(\frac{ w^2}{q^{m-s}(1-q^{2s}w^2\theta^2)}\cdot  {x\theta^2 (1-u^2\theta^2)}\right)f(x).\] 
Since $\theta^2 (1-u^2\theta^2)\in\BZ_q^\times$, 
the third case follows.

{Suppose that $m-s=1$.}  For $x\in \BZ_q^\times$, {put $k=\frac{\theta^2 w^2(u^2\theta^2-1)}{1-b^2\theta^2}$.} We have
\[\begin{aligned}
  \sum_{\{b\in q\BZ/q^{m}\BZ\ |\  q^{m-1}\parallel b\}}\pi\begin{pmatrix}
  1 & \frac{\theta^2 w^2(1-u^2\theta^2)}{q^{m-s}(1-b^2\theta^2)}\\ &1
 \end{pmatrix}f(x)
 =&2\sum_{k\in \BF_q^\times\bs \BF_{q}^{\times 2}}\psi\left(\frac{-xk}{q}\right)f(x)\\
 =&\sum_{k\in \BF_q^\times}\psi\left(\frac{-xk}{q}\right)(1-\eta(k))f(x)\\
  =&-f-\eta(-1)G(\eta,\psi)f_{\eta}.
\end{aligned}\] 
Here the first equality relies on $u^2\theta^2-1\in \BZ_q^{\times 2}$, $\theta^2\in \BZ_q^\times\bs \BZ_{q}^{\times 2}$ and $1-b^2\theta^2\equiv 1\pmod{q^m}$ under the assumption $q^{m-1}\parallel b$ and $m\geq 2$.

{Now suppose that $m-s\geq 2$.}
Consider the map 
\[\delta:(\BZ/q^{m-s}\BZ)^\times \ra (\BZ/q^{m-s}\BZ)^\times,\quad w \mapsto \frac{w^2}{1-q^{2s}w^2\theta^2}.\]  
Note that the map $\delta$ is $2$ to $1$ and its image is a disjoint union of fiber of the natural projection map 
$(\BZ/q^{m-s}\BZ)^\times\ra (\BZ/q\BZ)^\times$. To see this claim, note that $\delta(w)=\delta(w')$ if and only if $w^2=w'^2$ in $(\BZ/q^{m-s}\BZ)^\times$. On the other hand, the image of the projection map 
\[(\BZ/q^{m-s}\BZ)^\times \ra (\BZ/q\BZ)^\times,\quad w \mapsto \frac{w^2}{1-q^{2s}w^2\theta^2}\]
 is $\BF_{q}^{\times 2}$. Comparing the cardinality, the claim follows.

If $m-s>1$, then the first case follows form the preceding paragraph and Fact \ref{zf}: for each fiber $S$ of the projection map, we have
\[\begin{aligned}
  \sum_{w\in \delta^{-1}(S)}\psi\left(\frac{ w^2}{q^{m-s}(1-q^{2s}w^2\theta^2)}\cdot  {x\theta^2 (1-u^2\theta^2)}\right)
  =&2 \sum_{v\in S}\psi\left(\frac{ v}{q^{m-s}}\cdot  {x\theta^2 (1-u^2\theta^2)}\right), 
\end{aligned}\] which vanishes by Fact \ref{zf}.

\end{proof}
\bigskip
  \begin{proof}[\underline{Proof of Proposition \ref{5}}]
\noindent\\

    While the first case is just a consequence of 
Lemmas \ref{l29} and \ref{l210}, 
the third follows from Lemmas \ref{l29}, and \ref{l210} and the fact that
    \[\pi\begin{pmatrix}& q^{-m}\\ q^m &\end{pmatrix}f=\epsilon(\pi)f.\]
    
 Now we consider the second case. In view of Lemmas \ref{l29} and \ref{l210} it suffices to show that 
\begin{equation}\label{pq}
 \begin{aligned}  &\left(\pi\begin{pmatrix}
      1&{q^{-1}}\\&1
    \end{pmatrix}\begin{pmatrix}& q^{-m}\\ q^m &\end{pmatrix}(-f-\eta(-1)G(\eta,\psi)f_{\eta}),f\right)
    =&\frac{\epsilon(\pi)-q\eta(-1)\epsilon(\pi\otimes\eta)}{q-1}\left(f,f\right).
    \end{aligned}
    \end{equation}

   By Proposition \ref{all}, 
    \[\pi\begin{pmatrix}& q^{-m}\\ q^m &\end{pmatrix}f_\eta=\epsilon(\pi\otimes\eta)\eta(-1)f_\eta,\] since $\eta=\eta^{-1}$, and so 
      \[\begin{aligned}
        &\left(\pi\begin{pmatrix}
        1&q^{-1}\\&1
      \end{pmatrix}\begin{pmatrix}& q^{-m}\\ q^m &\end{pmatrix}(-f-\eta(-1)G(\eta,\psi)f_{\eta}),f\right)
      =&\left(\pi\begin{pmatrix}
        1&q^{-1}\\&1
      \end{pmatrix}(-\epsilon(\pi)f-G(\eta,\psi)\epsilon(\pi\otimes\eta)f_{\eta}),f\right).
    \end{aligned}\]
Note that
    \[\begin{aligned}
      {\pi}\left(\begin{pmatrix}
          1&q^{-1}\\&1
        \end{pmatrix}f_\eta(x),f\right)
       =&\left(\frac{1}{q-1}\sum_{u\in (\BZ/q\BZ)^\times} {\pi}\begin{pmatrix}
          u&uq^{-1}\\&1
        \end{pmatrix}f_\eta(x),f\right)
        =& \frac{{G(\eta,\psi)}}{q-1}(f,f) 
      \end{aligned}\]
and that 
      \[\begin{aligned}
        \left({\pi}\begin{pmatrix}
          1&q^{-1}\\&1
        \end{pmatrix}f,f\right)
        =&\frac{1}{q-1}\sum_{u\in \BF_q^\times} \left({\pi}\begin{pmatrix}
          u& \\ &1
        \end{pmatrix}\begin{pmatrix}
          1&\frac{1}{q}\\&1
        \end{pmatrix}\begin{pmatrix}
          u^{-1}& \\ &1
        \end{pmatrix}f,f\right)\\
        =&\frac{1}{q-1}\sum_{u\in \BF_q^\times} \left({\pi}\begin{pmatrix}
          1&\frac{u}{q}\\&1
        \end{pmatrix}f,f\right)\\
        =&-\frac{1}{q-1}(f,f).
      \end{aligned}\]
  Here the last equality follows from the defining action of unipotent elements, and the fact that summation of the $q$-th primitive root of unity is $-1$.
  
      Therefore, we have
      \[\begin{aligned}
        \left(\pi\begin{pmatrix}
        1&q^{-1}\\&1
      \end{pmatrix}\begin{pmatrix}& q^{-m}\\ q^m &\end{pmatrix}(-f-\eta(-1)G(\eta,\psi)f_{\eta}),f\right)
      =&\left(\pi\begin{pmatrix}
        1&q^{-1}\\&1
      \end{pmatrix}(-\epsilon(\pi)f-G(\eta,\psi)\epsilon(\pi\otimes\eta)f_{\eta}),f\right)\\
      =&\frac{\epsilon(\pi)-q\eta(-1)\epsilon(\pi\otimes\eta)}{q-1}\left(f,f\right).   \end{aligned}\]
  Here the last equality uses $G(\eta,\psi)^2=q\eta(-1)$ for $\eta$ quadratic, concluding the proof of \eqref{pq}.

  \end{proof}
  
  \subsection{Explicit toric period formula}\label{expt}
  {Let the setting be as in \S\ref{ss:lt-set}.}
  This subsection concludes the proof of Theorem~\ref{ml} for toric period of newform. It is a consequence of  
 Theorem \ref{co} in combination with Lemma \ref{rt} and Proposition \ref{mm} below, the latter being an explicit formula for twisted epsilon factors appearing in Theorem~\ref{ml}. {It crucially relies on $\pi=\pi_\lambda$ being supercuspidal.}

\subsubsection{Epsilon factors}

\begin{lem}\label{rt}
For $\eta$ the non-trivial quadratic character of $\BZ_q^\times$,
 we have \[\epsilon(\pi\otimes\eta)=-\eta(-1)\epsilon(\pi).\]
\end{lem}
\begin{proof}
Recall that $\pi=\pi_{\lambda}$. Since $K$ and $\psi$ are unramified and $q$ odd,  \[\epsilon(\pi)=\epsilon(1/2,\pi,\psi)=\lambda_{K}(\psi)\epsilon(1/2,\lambda,\psi_K)=(-1)^{m}\lambda(\theta)\] 
for $\lambda_{K}(\psi)=\frac{\int_{\BZ_q^\times}\eta_K(u)\psi(u)d u}{|\int_{\BZ_q^\times}\eta_K(u)\psi(u)d u|}=1$. 
 Here the second equality follows from \cite[Thm.~4.7]{Jacquet_Langlands:GLtwo}, and the last from \cite[Prop.~3.7]{MS}.

  In particular, 
   comparing the root number of $\epsilon(\pi)$ with $\epsilon(\pi\otimes\eta)$, 
   we have \[\epsilon(\pi\otimes\eta)=\eta(N_{K/\BQ_q}(\theta))\epsilon(\pi)=-\eta(-1)\epsilon(\pi).\]
\end{proof}

\begin{lem}\label{rt2}
For any $\chi\in \wh{(\BZ/q^{m}\BZ)^\times}$,
 the character $\lambda\chi_{K}$ also has conductor $q^m$ and 

          \[\epsilon(1/2,\pi\otimes\chi,\psi)=(-1)^m\lambda(\theta)\frac{G((\lambda\chi_K)^{-1},\psi_K)}{G(\lambda^{-1},\psi_K)}.\] 
          Here $\chi_K=\chi\circ {{\rm{N}}_{K/\BQ}}$, $G({\lambda'},\psi_K):=\sum_{O_K^\times/O_{K,q^m}^\times}{\lambda'}(x)\psi_{K}\left(\frac{x}{q^m}\right)$ with $\psi_K=\tr_{K/\BQ_q}\psi$ and $\lambda'$ a character with conductor $q^m$.
\end{lem}
\begin{proof}
Note that $\chi_K$ has conductor $q^m$ and that $\lambda|_{1+q^{m-1}O_K}$ does not factor through norm, and so $\lambda\chi_{K}$ has conductor $q^m$. In particular, $G((\lambda\chi_K)^{-1},\psi_K)$ is well defined and non-zero.

We have {$\pi\otimes\chi=\pi_{\lambda}\otimes \chi=\pi_{\lambda\chi_K}$ (cf.~\cite[Thm.~4.7]{Jacquet_Langlands:GLtwo}).}
As in the proof of Lemma \ref{rt}, 
\[\frac{\epsilon(1/2,\pi_{\lambda\chi_K},\psi)}{\epsilon(1/2,\pi,\psi)}=\frac{\epsilon(1/2,\lambda\chi_K,\psi_K)}{\epsilon(1/2,\lambda,\psi_K)}.\]
In view of the definition of epsilon factor in terms of Gauss sum (cf.~\cite[p.~281]{MS}) we have

\[\epsilon(1/2,\pi\otimes\chi,\psi)=\epsilon(1/2,\pi,\psi)\frac{G((\lambda\chi_K)^{-1},\psi_K)}{G(\lambda^{-1},\psi_K)}.\] 
Therefore \cite[Prop.~3.7]{MS} concludes the proof. 
\end{proof}
\subsubsection{Analysis of twist Gauss sum}
In this subsection we obtain explicit formulas for the twist Gauss sums appearing in Proposition \ref{mp}.

\begin{lem}\label{js}
  For $\chi\in \wh{(\BZ/q^m\BZ)^\times}$ primitive and $v\in (\BZ/q^m\BZ)^\times$, we have
  \[ G(\chi,\psi)^2\chi^{-1}(u^2\theta^2v)\frac{G((\lambda\chi_K)^{-1},\psi_K)}{G(\lambda^{-1},\psi_K)}=\chi(-4{v^{-1}})J(\chi,\chi)\sum_{a\in \BZ/q^m\BZ}(\lambda\chi_{K})^{-1}(a+\theta u)\] 
 and \[J(\chi,\chi)=\sum_{x\in (\BZ/q^m\BZ)^\times}\chi(x)\chi(1-x).\] 
\end{lem}

\begin{proof}Simply denote $\theta u$ by $\theta'$.
Note that $\ov{\theta'}=-\theta'$, and 
so 
\begin{equation}\label{g1}
\begin{aligned}
 & \chi^{-1}(u^2\theta^2v)\frac{G((\lambda\chi_K)^{-1},\psi_K)}{G(\lambda^{-1},\psi_K)}
  =&\chi(-4{v^{-1}})\frac{G\left((\lambda\chi_K)^{-1},\psi_K\left(\frac{\cdot}{2\theta'}\right)\right)}{G\left(\lambda^{-1},\psi_K\left(\frac{\cdot}{2\theta'}\right)\right)}.
\end{aligned}
\end{equation}

In the following we analyse the Gauss sum in the numerator based on the fact that  
\[O_K^\times =\BZ_q^\times\oplus \BZ_q\theta'\sqcup q\BZ_q\oplus  \BZ_q^\times \theta'\] and $\tr_{K/\BQ_q}\left(\frac{a+b\theta'}{q^{m}2\theta'}\right)=\frac{b}{q^m}$.

To begin $G((\lambda\chi_K)^{-1},\psi_K(\frac{\cdot}{2\theta'}))=I+J$, where \[I=\sum_{a\in (\BZ/q^m\BZ)^\times }\sum_{b\in \BZ/q^m\BZ }(\lambda\chi_{K})^{-1}(a+b\theta')\psi\left(\frac{b}{q^m}\right)\]
\[J=\sum_{a\in q\BZ/q^m\BZ }\sum_{b\in (\BZ/q^m\BZ)^\times }(\lambda\chi_{K})^{-1}(a+b\theta')\psi\left(\frac{b}{q^m}\right).\]

Note that 
\[\begin{aligned}
  I=&\sum_{b\in \BZ/q^m\BZ }(\lambda\chi_{K})^{-1}(1+b\theta')\left(\sum_{a\in (\BZ/q^m\BZ)^\times}\chi^{-2}(a)\psi\left(\frac{ba}{q^m}\right)\right)\\
  =&\sum_{b\in (\BZ/q^m\BZ)^\times }(\lambda\chi_{K})^{-1}(1+b\theta')\left(\sum_{a\in (\BZ/q^m\BZ)^\times}\chi^{-2}(a)\psi\left(\frac{ba}{q^m}\right)\right).\end{aligned}\]
The last equality follows from:   $\sum_{a\in (\BZ/q^m\BZ)^\times}\chi^{-2}(a)\psi\left(\frac{ba}{q^m}\right)$ is non-zero only for $b\in (\BZ/q^{m}\BZ)^\times$ since $\chi^{-2}$ is still a primitive character modulo $q^m$. 
  Thus
  \[\begin{aligned}
  I=&G(\chi^{-2},\psi)\sum_{b\in (\BZ/q^m\BZ)^\times }(\lambda\chi_{K})^{-1}(1+b\theta')\chi_K(b)\\
  =&G(\chi^{-2},\psi)\sum_{b\in (\BZ/q^m\BZ)^\times }(\lambda\chi_{K})^{-1}(b^{-1}+\theta')\\
  =&G(\chi^{-2},\psi)\sum_{b\in (\BZ/q^m\BZ)^\times }(\lambda\chi_{K})^{-1}(b+\theta'). 
\end{aligned}\]
Similarly, we have
\[\begin{aligned}
  J
  =&G(\chi^{-2},\psi)\sum_{a\in q\BZ/q^m\BZ }(\lambda\chi_{K})^{-1}(a+\theta'). 
\end{aligned}\]
Hence
\begin{equation}\label{g2}
I+J=G(\chi^{-2},\psi)\sum_{a\in \BZ/q^m\BZ }(\lambda\chi_K)^{-1}(a+\theta').
\end{equation}

Note that 
\begin{equation}\label{g3}
G(\chi,\psi)^2G(\chi^{-2},\psi)=q^mJ(\chi,\chi)
\end{equation}
since 
$G(\chi,\psi)^2=G(\chi^{2},\psi)J(\chi,\chi)$
 and $G(\chi^{2},\psi)G(\chi^{-2},\psi)=q^m$ 
 as $\chi,\chi^2$ are primitive modulo $q^m$.  

Now we have 
\[\begin{aligned}
 G(\chi,\psi)^2\chi^{-1}(u^2\theta^2v)\frac{G((\lambda\chi_K)^{-1},\psi_K)}{G(\lambda^{-1},\psi_K)}
=&G(\chi,\psi)^2\chi(-4{v^{-1}})\frac{G\left((\lambda\chi_K)^{-1},\psi_K\left(\frac{\cdot}{2\theta'}\right)\right)}{G\left(\lambda^{-1},\psi_K\left(\frac{\cdot}{2\theta'}\right)\right)}\\
 =&\frac{\chi(-4{v^{-1}})}{G(\lambda^{-1},\psi_K(\frac{\cdot}{2\theta'}))}G(\chi,\psi)^2G(\chi^{-2},\psi)\sum_{a\in \BZ/q^m\BZ }(\lambda\chi_{K})^{-1}(a+\theta')\\
 =&\left(\chi(-4{v^{-1}})J(\chi,\chi)\sum_{a\in \BZ/q^m\BZ }(\lambda\chi_{K})^{-1}(a+\theta')\right).\\
\end{aligned}\]
Here the first equality follows from \eqref{g1}, the second from \eqref{g2} and the third from \eqref{g3}.

As seen in the {proof} of \cite[Prop.~3.7]{MS} $G\left(\lambda^{-1},\psi_K\left(\frac{\cdot}{2\theta'}\right)\right)=q^m$, and so the proof concludes.
\end{proof}

In view of Theorem \ref{co} and Lemma \ref{js} it is natural to consider the following function on $\BZ_q$:
\begin{equation}\label{def:fq}
F_v(a)=\sum_{\substack{\chi\in \wh{(\BZ/q^m\BZ)^\times }\\ primitive}}\chi(-4v)J(\chi,\chi)\chi^{-1}(a^2-\theta'^2) 
\end{equation}
 for $v\in (\BZ/q^m\BZ)^\times$. Note that $F_v$ depends on $\theta'=\theta u$.

To explicitly describe $F_v$, we recall the following elementary facts. 
\begin{fact}\label{f}Let $m\geq 2$ be an integer.
  For $a\in (\BZ/q^{m}\BZ)^\times$, 
  \[\sum_{\substack{\chi\in \wh{(\BZ/q^m\BZ)^\times }\\ primitive}}\chi(a)=
  \begin{cases}
  (q-1)^2q^{m-2},\quad& a\equiv 1\pmod{q^{m}},\\
  -(q-1)q^{m-2},\quad & a\equiv 1\pmod{q^{m-1}}, a\nequiv 1\pmod{q^{m}},\\
  0,\quad &a\nequiv 1\pmod{q^{m-1}}.\\
  \end{cases}\]
  \end{fact}

    \begin{fact}\label{quad}
      Consider a quadratic equation 
    \begin{equation}X^2+2bX+c=0\label{eq:qu}\end{equation} 
    with $a,b\in \BZ/q^k\BZ$ for an odd prime $q$ and $k\geq 1$.
    \begin{itemize}
      \item[\tiny{$\bullet$}]  If $b^2-c$ is not a square in $\BZ/q^k\BZ$, then \eqref{eq:qu} has no solution in $\BZ/q^k\BZ$.
      \item[\tiny{$\bullet$}] If $b^2-c=v^2$ is a square with $v\in(\BZ/q^k\BZ)^\times$, then \eqref{eq:qu} has exactly $2$ solutions. 
    \item [\tiny{$\bullet$}] {If $b^2-c=v^2$ is a square with $q|v$, let $t={\rm ord}_q(b^2-c) \in \{1,\cdots k\}$. 
    \begin{itemize}
      \item [--] If $t= k$, then \eqref{eq:qu} has $q^{\lfloor k/2\rfloor}$ solutions. 
      \item [--] If $t<k$,write $t=2r$ with $1\leq r<k/2$, then \eqref{eq:qu} has $2q^{r}$ solutions.
  \end{itemize}
}
    \end{itemize}
    \end{fact}  
    We have the following explicit description of the function $F_v$ in certain cases\footnote{In view of Theorem \ref{co} and Lemma \ref{js} the congruence conditions appearing in Lemma~\ref{bF'} suffice for our application.}.
      {
        \begin{lem}\label{bF'}Let $m\geq 2$ be an integer, $v\in (\BZ/q^m\BZ)^\times$ with 
        $v\equiv 1\pmod{q^{m-1}}$
        and $F_v$ a function on $\BZ_q$ as in \eqref{def:fq}.
          Let $a_0\in\BZ_q^\times$ be a solution of $1+(a^2-{\theta^2u^2})=0$.
           \begin{itemize}
             \item [(i)] If $m$ is even and $v=1$, then 
             \[\begin{aligned}
               F_v=&q^{\frac{3m}{2}-2}(q-1)^2\left(1_{a_0(1+q^m\BZ_q)}+1_{-a_0(1+q^m\BZ_q)}-\frac{1}{q-1}\sum_{\substack{d\in\{a_0(1+q^{m-1}u)\ |\   1 \leq u \leq q-1\}}}(1_{ d(1+q^m\BZ_q)}+1_{ -d(1+q^m\BZ_q)})\right).
           \end{aligned}\] 
           \item[(ii)] If $m$ is odd, write $v=1+wq^{m-1}$ for $w\in \BF_q$, then
           \[\begin{aligned}
            F_v= &q^{\frac{3(m-1)}{2}}(q-1)\left(\sum_{\substack{ d\in\{a_0(1+q^{m-1}u)\ |\  u=0,\cdots ,q-1\}\\ w+2a_0^2u\in \BF_{q}^{\times 2}}}(1_{ d(1+q^m\BZ_q)}+1_{ -d(1+q^m\BZ_q)})\right)\\
    -&q^{\frac{3(m-1)}{2}}(q-1)\left(\sum_{\substack{ d\in\{a_0(1+q^{m-1}u)\ |\  u=0,\cdots ,q-1\}\\ w+2a_0^2u\in \BF_q^\times\bs \BF_{q}^{\times 2}}}(1_{ d(1+q^m\BZ_q)}+1_{ -d(1+q^m\BZ_q)})\right).
           \end{aligned}\]
           \end{itemize}
             \end{lem}}
             \begin{proof}
              In the following we simply denote $\theta u$ by $\theta'$ and   fix $a\in \BZ_q$. 
                  
                  Note that
                  \[F_v(a)=\sum_{x\in (\BZ/q^m\BZ)^\times}\left(\sum_{\substack{\chi\in \wh{(\BZ/q^m\BZ)^\times }\\ primitive}}\chi(v(4x^2-4x)(a^2-\theta'^2)^{-1})\right).\] Thus in view of Fact \ref{f} a necessary condition for $F_v(a)\neq 0$ is that the equation 
                  \begin{equation}v(4x^2-4x)\equiv (a^2-\theta'^2)\pmod{q^{m-1}}\label{eq:qum}\end{equation}
                  has a solution $x\in (\BZ/q^m\BZ)^\times$. (Note that  $a^2-\theta'^2$ is a $q$-adic unit.) 
                 
          Let $A$ be the number of solutions {$x$} of \begin{equation}v(4x^2-4x)\equiv (a^2-\theta'^2)\pmod{q^{m-1}}\label{eq:qumA}\end{equation} in $(\BZ/q^m\BZ)^\times$ and $B$ that of \begin{equation}v(4x^2-4x)\equiv (a^2-\theta'^2)\pmod{q^{m}}\label{eq:qumB}\end{equation} in $(\BZ/q^m\BZ)^\times$.
          By  Fact \ref{f},  \[F_v(a)=B(q-1)^2q^{m-2}-(A-B)(q-1)q^{m-2}.\]

              Note that the discriminant of \eqref{eq:qum} in $\BZ/q^{m-1}\BZ$ equals \[16(v^2+v(a^2-\theta'^2))\equiv16(1+(a^2-\theta'^2))\pmod{q^{m-1}}\]since $m\geq 2$ and $v\equiv 1\pmod{q^{m-1}}$.
                In particular, $F_v(a)\neq 0$ implies that 
                \[1+(a^2-\theta'^2)\pmod{q}\in \BF_{q}^{ 2}.\]
              
                We now consider the following cases. 
                \begin{itemize}
                  \item Suppose that $1+(a^2-\theta'^2)\pmod{q}\in \BF_{q}^{\times 2}$. 
                  
                  Then the discriminant of \eqref{eq:qum} is a square and a unit in $\BZ/q^{m-1}\BZ$. Hence
           \eqref{eq:qumA}
               has $2$ solutions in $(\BZ/q^{m-1}\BZ)^\times$ by Fact \ref{quad} and $2q$ solutions in $(\BZ/q^{m}\BZ)^\times$ and exactly two of such $x$ satisfy
               \eqref{eq:qumB}
               by Fact \ref{quad} again. So $F_v(a)=0$.
            
               \item Suppose that $1+(a^2-\theta'^2)\equiv 0\pmod{q}$. Then $a$ is a unit. We consider the following three subcases. 
             
              \begin{itemize}
                \item [(a)]Moreover, suppose that $v^2+v(a^2-\theta'^2)\equiv 0\pmod{q^m}$, and so $1+(a^2-\theta'^2)\equiv 0\pmod{q^{m-1}}$. 
                Then 
                \eqref{eq:qumA}
                has 
                $q^{\lfloor {(m-1)}/2\rfloor}$ solutions in $(\BZ/q^{m-1}\BZ)^\times$ by Fact \ref{quad} and hence  $q\cdot q^{\lfloor {(m-1)}/2\rfloor}$ solutions in $(\BZ/q^m\BZ)^\times$.
          On the other hand, 
                \eqref{eq:qumB}
                has $q^{\lfloor {m}/2\rfloor}$ solutions in $(\BZ/q^m\BZ)^\times$ by Fact \ref{quad}.
                So         \[F_v(a)=\begin{cases}
              q^{m/2}(q-1)^2q^{m-2},\quad &\text{if $m$ is even}\\
              0,\quad &\text{if $m$ is odd.}\\
              \end{cases}\] 
                 \item [(b)]Suppose that $1+(a^2-\theta'^2)\equiv 0\pmod{q^{m-1}}$ but $v^2+v(a^2-\theta'^2)\nequiv 0\pmod{q^m}$.

            Then 
                \eqref{eq:qumA} has 
                  $q^{\lfloor {(m-1)}/2\rfloor}$ solutions in $(\BZ/q^{m-1}\BZ)^\times$ 
                  and 
                  $q\cdot q^{\lfloor {(m-1)}/2\rfloor}$ solutions in $(\BZ/q^m\BZ)^\times$.

                 Note that \eqref{eq:qumB}
                 has a solution only if $m$ is odd\footnote{Only then the discriminant has even $q$-adic valuation.}, and in this case, it has \[\begin{cases}2 q^{{(m-1)}/2},\quad &\text{if $v^2+v(a^2-\theta'^2)\in (\BZ/q^m\BZ)^2$}\\ 
                  0 ,\quad &\text{if $v^2+v(a^2-\theta'^2)\notin (\BZ/q^m\BZ)^2$}
                 \end{cases}\]solutions. So 
                 $$F_v(a)=\begin{cases}
                   -q\cdot q^{\lfloor {(m-1)}/2\rfloor}(q-1)q^{m-2} ,\quad &\text{if $m$ is even}\\
                   q\cdot q^{\lfloor {(m-1)}/2\rfloor}(q-1)q^{m-2} ,\quad &\text{if $m$ is odd and if $v^2+v(a^2-\theta'^2)\in (\BZ/q^m\BZ)^2$.}\\
                   -q\cdot q^{\lfloor {(m-1)}/2\rfloor}(q-1)q^{m-2}\quad &\text{if $m$ is odd and if $v^2+v(a^2-\theta'^2)\notin (\BZ/q^m\BZ)^2$.}\\
                   \end{cases}$$
       The following calculation gives criterion for $v^2+v(a^2-\theta'^2)\in (\BZ/q^m\BZ)^2$. Write $a=a_0(1+q^{m-1}u)$ for $u\in \BF_q$ and $v=1+wq^{m-1}$ for $w\in \BF_q$.
       Then 
       \[\begin{aligned}
         &v^2+v(a^2-\theta'^2)
         \equiv &(w+2a_0^2u)q^{m-1}\pmod{q^m}.
       \end{aligned}\]
             \item [(c)]
              Suppose that $1+(a^2-\theta'^2)\nequiv 0\pmod{q^{m-1}}$, and put $$t=v_q(1+(a^2-\theta'^2))=v_q(v^2+v(a^2-\theta'^2))\in \{1,\cdots,m-2\}.$$
        If $t$ is odd, then   
                  \eqref{eq:qumA} has no solution, and hence $$F_v(a)=0.$$ 
                  
                  Suppose that $t$ is even with $t=2r$. 
                  If $1+(a^2-\theta'^2)$ is not a square in $\BZ/q^m\BZ$, then neither \eqref{eq:qumA} nor \eqref{eq:qumB} have a solution in $\BZ/q^m\BZ$. It follows that $F_v(a)=0$.
                On the other hand, if $1+(a^2-\theta'^2)$ is a square in $\BZ/q^m\BZ$ and then so is $v^2+v(a^2-\theta'^2)$. 
                Consequently 
                \eqref{eq:qumA}
                has $2q^{r}$ solution in $(\BZ/q^{m-1}\BZ)$ by Fact \ref{quad} and in turn has $2q^{r+1}$ solutions in $(\BZ/q^{m}\BZ)$. 
                Note that \eqref{eq:qumB}
                has $2q^r$ solutions in  $(\BZ/q^{m}\BZ)$ by Fact \ref{quad}. Thus we still have $$F_v(a)=0.$$ 
               \end{itemize}
               \end{itemize}
                \end{proof}
      \subsubsection{Explicit formulas for sum of twist epsilon factors}
\begin{prop}\label{mm} In our setting the following holds. 
  \begin{itemize}
    \item [(i)] {If $m$ is even, then \[\begin{aligned}
      & \sum_{\substack{\chi\in \wh{(\BZ/q^m\BZ)^\times }\\ primitive}} G(\chi,\psi)^2\chi^{-1}(u^2\theta^2)\epsilon(1/2,\pi\otimes\chi,\psi)\\
      =&
        q^{\frac{3m}{2}-1}(q-1)\lambda(\theta)\left(\lambda^{-1}(a_0+{\theta u})+\lambda^{-1}(-a_0+\theta u)\right).
    \end{aligned}\] }
    \item[(ii)]  { If $m$ is odd, then
    \[\begin{aligned}
      &\sum_{\substack{\chi\in \wh{(\BZ/q^m\BZ)^\times }\\ primitive}} G(\chi,\psi)^2\chi^{-1}(u^2\theta^2)\epsilon(1/2,\pi\otimes\chi,\psi)\\
      +&2\sum_{v\in 1-q^{m-1}(\BZ/q\BZ)^{\times 2}}  \sum_{\substack{\chi\in \wh{(\BZ/q^m\BZ)^\times },\\ primitive}} G(\chi,\psi)^2\chi^{-1}(u^2\theta^2v)\epsilon(1/2,\pi\otimes\chi,\psi)\\
      =&(-1)^m\lambda(\theta)q^{\frac{3m-1}{2}}(q-1)\left(\lambda^{-1}(a_0+\theta')+\lambda^{-1}(-a_0+\theta')\right).
    \end{aligned}\]}
  \end{itemize}

\end{prop}

\begin{proof}  
{In the following we simply denote $\theta u$ by $\theta'$.}

  We first consider the case (i). 
  Note that
  \[\begin{aligned}
    &\sum_{\substack{\chi\in \wh{(\BZ/q^m\BZ)^\times }\\ primitive}} G(\chi,\psi)^2\chi^{-1}(u^2\theta^2)\epsilon(1/2,\pi\otimes\chi,\psi)\quad\\
    =&(-1)^m\lambda(\theta)\left(\sum_{\substack{\chi\in \wh{(\BZ/q^m\BZ)^\times }\\ primitive}}\chi(-4)J(\chi,\chi)\sum_{a\in \BZ/q^m\BZ}(\lambda\chi_{K})^{-1}(a+\theta')\right)\quad\\
    =&(-1)^m\lambda(\theta)q^{\frac{3m}{2}-2}(q-1)^2\\
   & \left(\lambda^{-1}(a_0+\theta')+\lambda^{-1}(-a_0+\theta')-\frac{1}{q-1}\left(\sum_{v=1}^{q-1}\lambda^{-1}(a_0(1+q^{m-1}v)+\theta')+\lambda^{-1}(-a_0(1+q^{m-1}v)+\theta')\right)\right)\\
    =&(-1)^m\lambda(\theta)q^{\frac{3m}{2}-1}(q-1)\left(\lambda^{-1}(a_0+\theta')+\lambda^{-1}(-a_0+\theta')\right).\\     
   \end{aligned}\]
   Here the first equality follows from Lemmas \ref{rt2}, and \ref{js}, the second from Lemma~\ref{bF'}, and the last from:
   \begin{equation}\begin{aligned}
    &\sum_{v=1}^{q-1}\lambda^{-1}(a_0(1+q^{m-1}v)+\theta')+\lambda^{-1}(-a_0(1+q^{m-1}v)+\theta')\\
    =&\sum_{v=1}^{q-1}\left(\lambda^{-1}(a_0+\theta')\lambda^{-1}\left(1+\frac{q^{m-1}v}{a_0+\theta'}\right)+\lambda^{-1}(-a_0+\theta')\lambda^{-1}\left(1+\frac{q^{m-1}v}{-a_0+\theta'}\right)\right)\\
    =&-\lambda^{-1}(a_0+\theta')-\lambda^{-1}(-a_0+\theta').
   \end{aligned} \label{eq:dirac}\end{equation}

   Now consider the case (ii). Just as above, by Lemmas \ref{rt2}, \ref{js} and \ref{bF'}, we have
   \begin{equation}\label{odd1}\begin{aligned}
    &\sum_{\substack{\chi\in \wh{(\BZ/q^m\BZ)^\times }\\ primitive}} G(\chi,\psi)^2\chi^{-1}(u^2\theta^2)\epsilon(1/2,\pi\otimes\chi,\psi)\quad \\
    =& (-1)^m\lambda(\theta)q^{\frac{3(m-1)}{2}}(q-1)\left(\sum_{\substack{ a=a_0(1+q^{m-1}u)\in a_0(1+q^{m-1}(\BZ/q\BZ)^\times)\\ \left(\frac{2u}{q}\right)=+1}}(\lambda^{-1}(a+\theta')+\lambda^{-1}(-a+\theta'))\right.\\
   &- \left.\sum_{\substack{ a=a_0(1+q^{m-1}u)\in a_0(1+q^{m-1}(\BZ/q\BZ)^\times)\\
    \left(\frac{2u}{q}\right)=-1}}(\lambda^{-1}(a+\theta')+\lambda^{-1}(-a+\theta'))\right).
   \end{aligned}\end{equation}
{Again by Lemmas \ref{rt2}, \ref{js} and \ref{bF'}, and an analysis similar to \eqref{odd1}, we have} 

   \[\begin{aligned}
    &\sum_{v\in 1-q^{m-1}(\BZ/q\BZ)^{\times 2}}  \sum_{\substack{\chi\in \wh{(\BZ/q^m\BZ)^\times }\\ primitive}} G(\chi,\psi)^2\chi^{-1}(u^2\theta^2v)\epsilon(1/2,\pi\otimes\chi,\psi)\quad\\
    =&{(-1)^m\lambda(\theta)}\sum_{v\in 1+q^{m-1}(\BZ/q\BZ)^{\times 2}}\left(\sum_{\substack{\chi\in \wh{(\BZ/q^m\BZ)^\times }\\ primitive}}\chi(-4v)J(\chi,\chi)\sum_{a\in \BZ/q^m\BZ}(\lambda\chi_{K})^{-1}(a+\theta')\right)\quad\\
    =&{(-1)^m\lambda(\theta)}q^{\frac{3(m-1)}{2}}(q-1)\left(\sum_{v\in 1+q^{m-1}(\BZ/q\BZ)^{\times 2}}\sum_{\substack{ a\in a_0(1+q^{m-1}(\BZ/q\BZ))\\ w+2a_0^2u\in \BF_{q}^{\times 2}}}(\lambda^{-1}(a+\theta')+\lambda^{-1}(-a+\theta'))\right)\quad\\
    -&{(-1)^m\lambda(\theta)}q^{\frac{3(m-1)}{2}}(q-1)\left(\sum_{v\in 1+q^{m-1}(\BZ/q\BZ)^{\times 2}}\sum_{\substack{ a\in a_0(1+q^{m-1}(\BZ/q\BZ))\\ w+2a_0^2u\in\BF_{q}^\times\bs \BF_{q}^{\times 2}}}(\lambda^{-1}(a+\theta')+\lambda^{-1}(-a+\theta'))\right), 
   \end{aligned} \]
where $v=1+q^{m-1}w$ and $a=a_0(1+q^{m-1}u)$.

Put \[S_1(u)=\{w\in  \BF_{q}^{\times 2}\ |\ 2a_0^2u+w\in \BF_{q}^{\times 2}\},\] and \[S_2(u)=\{w\in  \BF_{q}^{\times 2}\ |\ 2a_0^2u+w\in \BF_q^\times\bs \BF_{q}^{\times 2}\}.\]

\begin{fact}\label{fact: q-sz} Let $q$ be an odd prime.

  Given $k\in \BF_q$, the cardinality of solutions of $x^2-y^2=k$ in $\BF_q^2$ is given by $\begin{cases}
  q-1,&\quad k\neq 0,\\
  2q-1,&\quad k=0.
  \end{cases}$
  \end{fact}
  \begin{proof}[Proof of Fact \ref{fact: q-sz}]
    First consider the case $k\neq 0$. Consider the projective curve $$x^2-y^2=kz^2$$ which has $q+1$ solutions in $\BP^2(\BF_q)$ and has $2$ solutions in $\BP^2(\BF_q)$ with $z=0$. Thus for $k\neq 0$, the equation $x^2-y^2=k$ has $q-1$ solutions in $\BF_q^{2}$.

    If $k=0$, then the equation becomes $x^2-y^2=0$, which has $2(q-1)$ non-zero solutions and $1$ zero solutions in $\BF_q^2$.
    \end{proof}

    As a consequence, we have 
\[\# S_1(u)=\begin{cases}
\frac{q-1}{4}, &\quad u\neq 0, q\equiv 1\pmod{4}, \eta(2u)=-1,\\
\frac{q-5}{4}, &\quad u\neq 0, q\equiv 1\pmod{4}, \eta(2u)=+1,\\
\frac{q-3}{4}, &\quad u\neq 0,   q\equiv 3\pmod{4},\\
\frac{q-1}{2}, &\quad u=0.
\end{cases}\] 
If $u\neq 0$, then \[\#S_1(u)+\# S_2(u)=\frac{q-1}{2}-\frac{\eta(-2u)+1}{2}.\]

From the above analysis, note that the substraction 
\[\begin{aligned}
  &{(-1)^m\lambda(\theta)}q^{\frac{3(m-1)}{2}}(q-1)\left(\sum_{v\in 1+q^{m-1}(\BZ/q\BZ)^{\times 2}}\sum_{\substack{ a\in a_0(1+q^{m-1}(\BZ/q\BZ))\\ w+2a_0^2u\in \BF_{q}^{\times 2}}}(\lambda^{-1}(a+\theta')+\lambda^{-1}(-a+\theta'))\right)\quad\\
    -&{(-1)^m\lambda(\theta)}q^{\frac{3(m-1)}{2}}(q-1)\left(\sum_{v\in 1+q^{m-1}(\BZ/q\BZ)^{\times 2}}\sum_{\substack{ a\in a_0(1+q^{m-1}(\BZ/q\BZ))\\ w+2a_0^2u\in\BF_{q}^\times\bs \BF_{q}^{\times 2}}}(\lambda^{-1}(a+\theta')+\lambda^{-1}(-a+\theta'))\right)\end{aligned}\] 
    equals  
  \begin{equation}\label{odd2}{(-1)^m\lambda(\theta)}q^{\frac{3(m-1)}{2}}(q-1) \left(\frac{q-1}{2}(\lambda^{-1}(a_0+\theta')+\lambda^{-1}(-a_0+\theta'))-
    \sum_{\substack{ a\in a_0(1+q^{m-1}(\BZ/q\BZ))\\ \eta(2u)=+1}}(\lambda^{-1}(a+\theta')+\lambda^{-1}(-a+\theta'))
  \right).\end{equation} 

  Combining \eqref{odd1} and \eqref{odd2}, we have
  \[\begin{aligned}
    &\sum_{\substack{\chi\in \wh{(\BZ/q^m\BZ)^\times }\\ primitive}} G(\chi,\psi)^2\chi^{-1}(u^2\theta^2)\epsilon(1/2,\pi\otimes\chi,\psi)\\
    +&2\sum_{v\in 1-q^{m-1}(\BZ/q\BZ)^{\times 2}}  \sum_{\substack{\chi\in \wh{(\BZ/q^m\BZ)^\times },\\ primitive}} G(\chi,\psi)^2\chi^{-1}(u^2\theta^2v)\epsilon(1/2,\pi\otimes\chi,\psi)\\
    =&(-1)^m\lambda(\theta)q^{\frac{3(m-1)}{2}}(q-1)\left((q-1)(\lambda^{-1}(a_0+\theta')+\lambda^{-1}(-a_0+\theta'))\right.\\ 
    &\left.-\sum_{\substack{ a\in a_0(1+q^{m-1}(\BZ/q\BZ)^\times)}}(\lambda^{-1}(a+\theta')+\lambda^{-1}(-a+\theta'))\right)\\
    =&(-1)^m\lambda(\theta)q^{\frac{3m-1}{2}}(q-1)\left(\lambda^{-1}(a_0+\theta')+\lambda^{-1}(-a_0+\theta')\right).
  \end{aligned}\]
Here the last equality proceeds as follows:
   For $a=a_0(1+q^{m-1}u)$, we have \[(a_0(1+q^{m-1}u))+\theta'\equiv (a_0+\theta')(1+q^{m-1}\frac{ua_0}{a_0+\theta'})\pmod{p^m}\]and hence $\lambda^{-1}(a+\theta')=\lambda^{-1}(a_0+\theta')\lambda^{-1}(1+q^{m-1}\frac{2ua_0}{2(a_0+\theta')})$.
Note that $\lambda^{-1}(1+q^{m-1}\frac{a_0}{2(a_0+\theta')})=\zeta_q$ for
$\zeta_q$ a $q$-th primitive root of unity, and 
as $u$ runs over $\BF_q^\times$, then $\lambda^{-1}(1+q^{m-1}\frac{a_0u}{2(a_0+\theta')})$ runs over all $q$-th primitive root of unity.
Since \[\sum_{a\in \BF_q^\times}\zeta_q^a=-1,\] the equality follows.

\end{proof}

\subsubsection{Proof of Theorem~\ref{ml}}\label{mlc} 
The following is a combination of the previous results of this section.

By Lemma \ref{rt}, we have
\[\frac{q-q\eta(-1)\epsilon(\pi)\epsilon(\pi\otimes\eta)}{q-1}=\frac{2q}{q-1}.\]
By Proposition \ref{mm} for $m$ even, we have

\[q^{\lfloor\frac{m}{2}\rfloor} \sum_{\substack{\chi\in \wh{(\BZ/q^m\BZ)^\times }\\ primitive}} \frac{G(\chi,\psi)^2}{\varphi(q^m)^2}\chi^{-1}(u^2\theta^2)\epsilon(1/2,\pi\otimes\chi,\psi)=\frac{q}{q-1}\lambda(\theta)\left(\lambda^{-1}(a_0+\theta u)+\lambda^{-1}(-a_0+{\theta u})\right), \]

and for $m$ odd, {
\[\begin{aligned}
&q^{\lfloor\frac{m}{2}\rfloor} \sum_{\substack{\chi\in \wh{(\BZ/q^m\BZ)^\times }\\ primitive}} \frac{G(\chi,\psi)^2}{\varphi(q^m)^2}\chi^{-1}(u^2\theta^2)\epsilon(1/2,\pi\otimes\chi,\psi)\\
+&2q^{(m-1)/2}\sum_{v\in 1-q^{m-1}(\BZ/q\BZ)^{\times 2}}  \sum_{\substack{\chi\in \wh{(\BZ/q^m\BZ)^\times }\\ primitive}} \frac{G(\chi,\psi)^2}{\varphi(q^m)^2}\chi^{-1}(u^2\theta^2v)\epsilon(1/2,\pi\otimes\chi,\psi)\quad\\
=&\frac{-q}{q-1}\lambda(\theta)(\lambda^{-1}(a_0+\theta u)+\lambda^{-1}(-a_0+\theta u)).
\end{aligned}
\]}
Recall that $\epsilon(\pi)=(-1)^m\lambda(\theta)$, and so the result follows from Theorem \ref{co}.

\end{document}